\documentclass[11pt]{amsart}
\usepackage{latexsym, mathrsfs, color, tikz, multirow, bbm, mathtools, amsmath, amssymb, enumitem, quiver, tikz-cd, circuitikz, dotlessi, thmtools, thm-restate}
\usepackage{graphics}
\usepackage{subcaption}
\input{xy}
\xyoption{all}

\usepackage[colorlinks=true, pdfstartview=FitV,
 linkcolor=black,citecolor=black,urlcolor=black]{hyperref}
 
\usepackage{tikz}
\usepackage{tikz-cd}
\usepackage{comment}
\usepackage{adjustbox}

\numberwithin{equation}{section}
\theoremstyle{plain}
\newtheorem{Thm}[equation]{Theorem}
\newtheorem{Prop}[equation]{Proposition}
\newtheorem{Cor}[equation]{Corollary}
\newtheorem{Lem}[equation]{Lemma}

\newtheorem{Ex}[equation]{Example}

\newtheorem*{T-homeo}{Theorem~\ref{N^N-homeo-thm}}
\newtheorem*{T-hered}{Theorem~\ref{hereditary-meta-thm}}
\newtheorem*{T-LF-inf-seq}{Theorem~\ref{LF-inf-mut-seq-thm}}
\newtheorem*{T-AF-inf-seq}{Theorem~\ref{AF-inf-mut-seq-thm}}
\newtheorem*{T-mut-class-top}{Theorem~\ref{mut-class-top-thm}}
\newtheorem*{T-Fraisse-singleton}{Theorem~\ref{Fraisse-singleton-thm}}
\newtheorem*{T-Fraisse-inf-seq}{Theorem~\ref{Fraisse-inf-seq-thm}}

\theoremstyle{definition}
\newtheorem{Def}[equation]{Definition}

\newtheorem{Rmk}[equation]{Remark}

\newenvironment{red}{\relax\color{red}}{\relax}
\newenvironment{blue}{\relax\color{blue}}{\hspace*{.5ex}\relax}

\newcommand{\NN}{\mathbb{N}}
\newcommand{\ZZ}{\mathbb{Z}}

\newcommand{\ber}{\begin{red}}
\newcommand{\er}{\end{red}}
\newcommand{\beb}{\begin{blue}}
\newcommand{\eb}{\end{blue}}

\newcommand{\AF}{\mathbf{AF}}
\newcommand{\LF}{\mathbf{LF}}

\begin{document}
	
\title[Topologizing infinite quivers and their mutations]{Topologizing infinite quivers and their mutations}

\author[Benjamin Grant]{Benjamin Grant}
\address{Department of Mathematics, University of Connecticut,
	Storrs, CT 06269, U.S.A.}
\email{benjamin.grant@uconn.edu}

\begin{abstract}
    We define several topological spaces whose points are quivers with a given infinite vertex set $X$. In the special case when $X$ is countably infinite, we show that two of the spaces of interest are homeomorphic to the Baire space $\NN^\NN$. We study properties of countably infinite quivers as subspaces of these topological spaces and prove a ``meta-theorem'' about hereditary properties of quivers. Furthermore, we approach the question of convergence for infinite mutation sequences in these spaces, providing a complete characterization of the (non-)density of the domains of convergence and divergence of infinite mutation sequences in one of these spaces and a partial characterization in the other. We then draw attention to a very special infinite quiver which we call the \emph{Fra\"iss\'e quiver} that draws a clear contrast between the behavior of finite and infinite mutation sequences. Finally, we reproduce (a very mild modification of) a previously-constructed topological space due to Ervin and Jackson as a subquotient of one of the spaces of interest. 
    MSC Classification (2020): 05E18, 13F60, 22F05, 54H05
\end{abstract}

\maketitle

%% Section 1: Introduction %%
\section{Introduction}\label{sec-1}

\subsection{Scientific context}
Around a quarter-century ago, Fomin and Zelevinsky began a highly influential series of papers introducing \emph{cluster algebras,} which are a particular class of commutative rings equipped with special combinatorial structures \cite{fomin_cluster_2001, fomin_cluster_2003, berenstein_cluster_2005, fomin_cluster_2007}. The central constructs governing the combinatorics of cluster algebras are \emph{quivers} and \emph{quiver mutations.} A quiver is a directed multigraph. Throughout this article, we always assume our quivers have no loops or oriented 2-cycles. Quiver mutations are certain involutive transformations on quivers defined locally with reference to a choice of vertex. While cluster algebras technically require the broader algebraic context of \emph{seeds} and \emph{seed mutations,} quivers and their mutations can already be quite complex and are worth study in their own right.

This paper aims to offer a topological perspective on infinite quivers and their mutations. Specifically, we introduce several topological spaces of quivers with a given infinite vertex set $X$: the spaces $\AF^X_{weak}$ and $\AF^X_{str}$ of \emph{arrow finite} quivers on $X$ with \emph{weak} and \emph{strong} topologies, and their subspaces $\LF^X_{weak}$ and $\LF^X_{str}$ of \emph{locally finite} quivers. An \emph{arrow finite} quiver is a quiver in which there are only finitely many arrows between any two given vertices, while a \emph{locally finite} quiver is a quiver in which each vertex has finite degree. We show that the spaces $\LF^X_{str}$ and $\AF^X_{weak}$ for $X$ countable are both homeomorphic to the Baire space $\NN^\NN$, the product of countably infinitely many copies of the discrete space $\NN$. In addition, we prove that quiver mutations give rise to self-homeomorphisms on all of the spaces of interest.

Typically, one is content with studying \emph{finite} quivers without concern for their infinite counterparts. However, there have been a number of developments in the theory of infinite quivers and infinite-rank cluster algebras in the past decade or so; see \cite{holm_cluster_2012,grabowski_cluster_2014,gratz_cluster_2015,liu_cluster_2017,baur_transfinite_2018,canakci_infinite_2019,paquette_completions_2021,gratz_ind-cluster_2025}. We therefore find it appropriate, useful, and necessary to consider infinite quivers from a topological point of view, guided by the principle that infinite analogs of finite structures are easier to understand and manipulate when bound by appropriate topologies. This principle has already manifest itself quite centrally in the cluster algebras and representation theory literature under the study of \emph{quivers with potential,} for example, where infinite-dimensional path algebras are equipped with topologies and algebra homomorphisms between them are required to be continuous \cite{derksen_quivers_2008, derksen_quivers_2010}. That said, we note that our approach also allows us to consider finite quivers right alongside infinite quivers merely as special cases of them.

Our particular topological approach carries another interesting feature. The space $\NN^\NN$ is a \emph{Polish space,} meaning it is a separable space homeomorphic to a complete metric space. Polish spaces are the primary objects of study in \emph{descriptive set theory,} a branch of mathematical logic concerned with various complexity hierarchies of subsets of Polish spaces as well as relative consistency results about subsets of Polish spaces under various cousins of ZFC set theory. In fact, the motivation for considering the spaces in this paper came from descriptive set theory; spaces of first-order structures in countable languages are common examples of Polish spaces (when given the right topologies), and this provides a concrete way to understand the complexities of their properties. Because of the desirable properties that Polish spaces enjoy, we restrict to the special cases of $\LF^X_{str}$ and $\AF^X_{weak}$ for $X$ countable from Section~\ref{sec-4} onward, and simply refer to these spaces as $\LF$ and $\AF$. We hope that the Polish spaces explored here can provide a new entry point for researchers outside of cluster combinatorics to contribute to the study.

We also wish to emphasize the growing body of work on constructing invariants of (finite) quivers under mutation and/or restriction to full subquivers, see for example \cite{beineke_cluster-cyclic_2011,lee_positivity_2015,fomin_universal_2021,ervin_topology_2024,ervin_new_2024,neville_mutation-acyclic_2024,fomin_cyclically_2024,kaufman_mutation_2024,seven_congruence_2024}. Many of these invariants come in the form of \emph{properties}, i.e., answers to ``yes or no'' questions (e.g., ``does this quiver have a reddening sequence?'' and ``is this quiver mutation-equivalent to an acyclic quiver?''). The topological approach presented in this paper allows one to \emph{compare} properties of locally finite or arrow finite quivers by comparing them as subspaces of $\LF$ or $\AF$ (see Theorem~\ref{hereditary-meta-thm} below for an example). For example, some properties correspond to open sets, some to closed sets, and some to more complicated Borel sets. This adds an additional dimension to the studies of these properties, and while this dimension only truly exists for these spaces of infinite quivers, we hope it informs the theory of finite quivers and their mutations in a meaningful way.

During the preparation of this paper, the author became aware of a recent preprint \cite{bucher_reddening_2025} due to E. Bucher and E. Howard which takes a different approach to the study of infinite quivers and their mutations, defining them using sequences of finite quivers. In this way, their approach resembles ours, and it would be interesting to see a more detailed comparison. Using their analysis of reddening sequences for infinite quivers, for example, one could perhaps study the set of infinite quivers in $\LF$ or $\AF$ admitting a reddening sequence. There are many other questions one could ask about the relationships between these works; the author looks forward to seeing where the topological program of infinite quivers and their mutations is headed.

\subsection{Main results}
Our main results are of five flavors:

    \textbf{1. Comparisons to a well-known space.} We show that two of our spaces, the spaces $\LF^X_{str}$ and $\AF^X_{weak}$ for a countably infinite set $X$, are homeomorphic to a well-known space with many desiderata, the Baire space $\NN^\NN$. (While this does not change the homeomorphism type of $\NN^\NN$, the reader should note now that we take $\NN$ to be the set of \emph{positive} integers throughout.)

    \begin{Thm}\label{N^N-homeo-thm}
        Let $X$ be a countably infinite set. Then the spaces $\AF^X_{weak}$ and $\LF^X_{str}$ are both homeomorphic to $\NN^\NN$.
    \end{Thm}

    Due to this observation, we restrict from the conclusion of this theorem onward to the spaces $\LF=\LF^X_{str}$ and $\AF=\AF^X_{weak}$ for $X$ countable.

    \vspace{0.25cm}

    \textbf{2. Topological properties of subspaces.} We establish topological properties of subspaces of $\LF$ and $\AF$ consisting of quivers with certain properties. We show in Proposition~\ref{mutation-acyclicity-prop}, for example, that the subspace $\mathbf{MA}_\LF\subseteq\LF$ of mutation-acyclic locally finite quivers is neither open nor closed (but is $F_\sigma$), is not dense, but has dense complement. Similar statements may be made about other quiver properties; we curate a small but detailed catalog containing such propositions. More generally, we prove a meta-theorem showing the density properties mentioned here for mutation-acyclicity follow from the fact that mutation-acyclicity is a \emph{hereditary} property, i.e., mutation-acyclicity is preserved under restriction to finite full subquivers:

    \begin{Thm}\label{hereditary-meta-thm}
        Let $\mathcal P\subseteq\AF$ be a hereditary property of arrow finite quivers such that there exists at least one finite quiver \emph{not} carrying the property $\mathcal P$. Then it holds that $\mathcal P\cap\LF$ is not dense in $\LF$, nor is $\mathcal P$ is dense in $\AF$. On the other hand, $\LF\setminus\mathcal P$ is dense in $\LF$, and similarly $\AF\setminus\mathcal P$ is dense in $\AF$. Moreover, $\mathcal P=\mathcal P_\mathcal F$ is the property of avoiding $\mathcal F$ for some (possibly infinite) family $\mathcal F$ of finite unlabeled quivers (see Definition~\ref{F-avoiding-def} for the precise definition) if and only if $\mathcal P$ is closed in $\AF$.
    \end{Thm}

    \vspace{0.25cm}

    \textbf{3. Infinite mutation sequences.} We consider infinite sequences of mutations and study their domains of convergence and divergence. We say that an infinite mutation sequence $\mu_\mathbf i=\ldots\mu_{i_3}\mu_{i_2}\mu_{i_1}$ \emph{converges (or diverges)} on a quiver $Q\in\LF$ (resp. $Q\in\AF$) if the sequence $(\mu_{i_j}\circ\ldots\circ\mu_{i_2}\circ\mu_{i_1}(Q))_{j\in\NN}$ converges (or diverges) in the space $\LF$ (resp. $\AF$). We use $\mathcal C_\LF(\mu_\mathbf i)$ (resp. $\mathcal C_\AF(\mu_\mathbf i)$) to denote the subset of quivers in $\LF$ (resp. $\AF$) on which $\mu_\mathbf i$ converges in $\LF$ (resp. $\AF$). We also let $\mathcal D_\LF(\mu_\mathbf i):=\LF\setminus\mathcal C_\LF(\mu_\mathbf i)$ and $\mathcal D_\AF(\mu_\mathbf i):=\AF\setminus\mathcal C_\AF(\mu_\mathbf i)$. In $\AF$, we show that every infinite mutation sequence has a dense domain of divergence by means of an explicit construction. However, the domain of convergence $\mathcal C_\AF(\mu_\mathbf i)$ of an infinite mutation sequence $\mu_\mathbf i$ on $\AF$ is a bit more complicated; we only provide a partial characterization of its density. This is captured in the following theorem:
    
    \begin{Thm}\label{AF-inf-mut-seq-thm}
        Let $\mu_{\mathbf i}:=\ldots\mu_{i_3}\mu_{i_2}\mu_{i_1}$ be an infinite mutation sequence on $\AF$. Then the following hold of the sets $\mathcal D_\AF(\mu_\mathbf i)$ and $\mathcal C_\AF(\mu_\mathbf i)$:
        \begin{enumerate}[label=(\roman*)]
            \item $\mathcal D_\AF(\mu_\mathbf i)$ is dense in $\AF$.
            \item If there does not exist $i\in\NN$ appearing infinitely many times in $\mathbf i$, then $\mathcal C_\AF(\mu_\mathbf i)$ is dense in $\AF$.
            \item If there are only finitely many $i\in\NN$ appearing in $\mathbf i$, then $\mathcal C_\AF(\mu_\mathbf i)$ is not dense in $\AF$.
        \end{enumerate}
    \end{Thm}
    
    This indeed does \emph{not} provide a complete characterization of the (non-)density of the domains of convergence of infinite mutation sequences on $\AF$, since we do not make any claims about what occurs if there exist infinitely many $i\in\NN$ appearing in $\mathbf i$ \emph{and} at least one $i\in\NN$ which appears infinitely many times in $\mathbf i$.
    
    However, our study of convergence for infinite mutation sequences on $\LF$ does indeed lead us to a complete characterization of the (non-)density of both the domains of convergence and divergence for infinite mutation sequences on $\LF$. After careful consideration of the combinatorics of infinite mutation sequences and their reductions and after proving a result about certain c-vectors of some quivers mutation-equivalent to a finite abundant acyclic quiver, we prove the following theorem:
    
    \begin{Thm}\label{LF-inf-mut-seq-thm}
        Let $\mu_{\mathbf i}:=\ldots\mu_{i_3}\mu_{i_2}\mu_{i_1}$ be an infinite mutation sequence on $\LF$. Exactly one of the following possibilities occurs:
        \begin{enumerate}[label=(\roman*)]
            \item $\mathcal C_\LF(\mu_\mathbf i)=\LF$ (and therefore $\mathcal D_\LF(\mu_\mathbf i)=\varnothing$). This occurs precisely when every vertex appearing in $\mathbf i$ appears only finitely many times and the reduction $\overline\mu_\mathbf i$ (see Definition~\ref{reduction-def} for a precise definition) is finite.
            \item $\mathcal C_\LF(\mu_\mathbf i)$ and $\mathcal D_\LF(\mu_\mathbf i)$ are both dense in $\LF$. This occurs precisely when every vertex appearing in $\mathbf i$ appears only finitely many times and the reduction $\overline\mu_\mathbf i$ is infinite.
            \item $\mathcal C_\LF(\mu_\mathbf i)$ is not dense. This happens precisely when the set $S_\infty\subseteq\NN$ of vertices appearing infinitely many times in $\mathbf i$ is nonempty. In this case, $\mathcal D_\LF(\mu_\mathbf i)$ is dense in $\LF$ if and only if one (or both) of the following occurs:
            \begin{enumerate}
                \item $S_\infty\subseteq\NN$ is infinite.
                \item $\overline\mu_{\mathbf i_{\NN\setminus S_\infty}}$ is infinite, where $\mathbf i_{\NN\setminus S_\infty}$ is the subsequence of $\mathbf i$ consisting only of those $i\in\NN$ that do not appear infinitely many times in $\mathbf i$.
            \end{enumerate}
        \end{enumerate}
    \end{Thm}

    We should remark that these theorems address a question about only one direction of the relationship between infinite quivers and infinite mutation sequences. That is, we answer a question which takes an infinite mutation sequence $\mu_\mathbf i$ as given and asks about the quivers $Q$ with a certain property in relation to $\mu_\mathbf i$ (in this case, the property of convergence/divergence under $\mu_\mathbf i$). It would be quite interesting to go ``the other way,'' meaning we fix some quiver $Q$ and ask about the mutation sequences $\mu_\mathbf i$ with a certain property in relation to $Q$. For example, one could naturally consider infinite mutation sequences as forming a space homeomorphic to $\NN^\NN$ in their own right and then ask about the subspace of mutation sequences which are convergent/divergent on a given quiver $Q$ (in either $\LF$ or $\AF$). With the exception of the case when $Q$ is finite, however, this seems to be a more difficult question than the one that Theorem~$\ref{LF-inf-mut-seq-thm}$ addresses, and could be the subject of a separate paper.

    \vspace{0.25cm}

    \textbf{4. The Fra\"iss\'e quiver.} We define and examine the mutation properties of a very special countably infinite quiver that we call the \emph{Fra\"iss\'e quiver} $Q_\mathcal Q$ (due to its construction as the \emph{Fra\"iss\'e limit} of the class $\mathcal Q$ of finite quivers). This quiver is uniquely characterized among countably infinite quivers by the properties of (i) having every finite quiver as a full subquiver and (ii) being homogeneous (a certain symmetry property; see Definition~\ref{homog-def}). We prove two seemingly contradictory theorems about this quiver, the first in more generality than is necessary for studying the Fra\"iss\'e quiver in isolation.

    \begin{Thm}\label{Fraisse-singleton-thm}
        Let $\mathcal K$ be a Fra\"iss\'e class of quivers (cf. Definition~\ref{fraisse-class-def}) which is additionally closed under mutation. Then, up to identifying isomorphic quivers, the (finite-)mutation-equivalence class of its \emph{Fra\"iss\'e limit} $Q_\mathcal K$ is a singleton.
    \end{Thm}

    On the other hand, our second theorem shows that the Fra\"iss\'e quiver $Q_\mathcal Q$ is much more flexible under \emph{infinite} mutation sequences, and can be made to approximate \emph{any} countably infinite quiver under an appropriate infinite mutation sequence:

    \begin{Thm}\label{Fraisse-inf-seq-thm}
        Let $F\in\AF$ be isomorphic to the Fra\"iss\'e quiver $Q_\mathcal Q$ and let $Q\in\AF$ be arbitrary. Then there exists an infinite mutation sequence $\mu_\mathbf i$ which converges on $F$ to $Q$, i.e., $\mu_\mathbf i(F)=Q$.
    \end{Thm}

    In tandem, these theorems show that infinite mutation sequences behave \emph{very} differently than finite mutation sequences in general.

    \vspace{0.25cm}

    \textbf{5. The mutation class topology.} Recently, T. Ervin and B. Jackson \cite{ervin_topology_2024} have constructed a topological space $\mathcal M$ whose points are mutation-equivalence classes of finite unlabeled quivers. The topology they construct, which they call the \emph{mutation class topology,} is the Alexandrov topology induced by the poset structure on the set of mutation classes given by embeddings of full subquivers. They establish several topological properties of the space $\mathcal M$ beyond what is true for Alexandrov topologies in general. We establish a connection between the subspace $\mathcal M'\subseteq\mathcal M$ of their mutation class space consisting of finite unlabeled quivers \emph{without isolated vertices} and our space $\AF$:

    \begin{Thm}\label{mut-class-top-thm}
        Let $\mathbf{Fin}\subseteq\AF$ be the subspace of $\AF$ consisting of finite quivers, and let $\sim$ denote the equivalence relation on $\mathbf{Fin}$ induced by (finite-)mutation-equivalence and isomorphism. Then $\mathbf{Fin}/\!\sim$ is homeomorphic to $\mathcal M'$ by the map sending $[Q]_\sim\in\mathbf{Fin}/\!\sim$ to the mutation-equivalence class of $Q'$, the quiver obtained from $Q\in\mathbf{Fin}$ by removing all of its isolated vertices \emph{if $Q$ is not the quiver with no arrows} or to the mutation-equivalence class of the one-vertex quiver if it is.
    \end{Thm}

\subsection{Organization} The layout of the paper is as follows. We recall basic definitions and facts from the theory of quiver mutations and provide appropriate modifications of these definitions and facts for our purposes in Section~\ref{sec-2}. We provide definitions of the spaces $\LF^X_{str}$, $\LF^X_{weak}$, $\AF^X_{str}$, and $\AF^X_{weak}$ of interest in Section~\ref{sec-3}, prove that mutations give self-homeomorphisms on these spaces, and prove Theorem~\ref{N^N-homeo-thm}. We then define a handful of subspaces of $\LF=\LF^\NN_{str}$ and $\AF=\AF^\NN_{weak}$ of interest and examine some of their topological properties in Section~\ref{sec-4}, in which we also prove Theorem~\ref{hereditary-meta-thm}. From there, we turn our attention to infinite mutation sequences and prove Theorems~\ref{AF-inf-mut-seq-thm} and \ref{LF-inf-mut-seq-thm} in Section~\ref{sec-5} to partially classify the density of the domains of convergence and divergence of infinite mutation sequences on $\AF$ and $\LF$, providing a complete classification in the $\LF$ case. Along the way to Theorem~\ref{LF-inf-mut-seq-thm}, we first take a detailed look at the combinatorial properties of infinite mutation sequences and then work to establish a certain ``c-vector stabilization'' phenomenon for a large class of finite quivers and finite mutation sequences applied to them. In Section~\ref{sec-6}, we define and examine the Fra\"iss\'e quiver and demonstrate a stark difference between its finite and infinite mutation classes in $\AF$ with Theorems~\ref{Fraisse-singleton-thm} and \ref{Fraisse-inf-seq-thm}. Section~\ref{sec-7} recalls the mutation class space $\mathcal M$ of Ervin and Jackson and gives a proof of Theorem~\ref{mut-class-top-thm}. Finally, we discuss possible avenues for future work concerning this topological approach to the theory of infinite quivers and their mutations in Section~\ref{sec-8}.

\subsection*{Acknowledgments}
The author would like to thank Garett Cunningham, Chinmay Dharmendra, Tucker Ervin, Elsa Frankel, Yuxiao Fu, Dion Mann, Scott Neville, Khrystyna Serhiyenko, Reed Solomon, Samuel Stuart, Oscar Quester, and the Cluster Algebras Reading Group at the University of Connecticut for a number of interesting conversations and helpful feedback throughout various stages of the preparation of this paper which greatly influenced its content, direction, and form. The author would especially like to thank his colleague, Blake Jackson, and his advisor, Ralf Schiffler, for their continued support with this project. This work was partially supported by a U.S. Department of Education GAANN Fellowship as well as the NSF grant DMS-2348909.

%% Section 2: Preliminaries %%
\section{Preliminaries}\label{sec-2}

We recall some basic ingredients from the theory of quivers and their mutations in a form that we find most convenient for our purposes.

\begin{Def}\label{quiv-def}
    Let $X$ be a set. An \emph{arrow finite quiver $Q$ with vertex set $X$ without loops or oriented $2$-cycles (or just ``quiver on $X$'')} is a skew-symmetric function $Q:X\times X\to\ZZ$, i.e. $Q$ has the property that for all $x,y\in X$, it holds that $Q(x,y)=-Q(y,x)$. We view such a function as a quiver in the traditional sense by interpreting $Q(x,y)$ as the signed number of arrows from the vertex $x$ to the vertex $y$ (see Remark~\ref{quiv-diagram-rmk} below).

    We let $\AF^X$ denote the set of all arrow finite quivers on a given vertex set $X$.
\end{Def}

\begin{Rmk}\label{quiv-diagram-rmk}
    We often represent quivers $Q:X\times X\to\ZZ$ using diagrams, placing a node for each element of $X$ and arrows between those nodes according to $Q$. When $|Q(x,y)|>1$, we may either draw out exactly $|Q(x,y)|$ arrows from $x$ to $y$ (or vice versa depending on the sign of $Q(x,y)$) or place a single arrow from $x$ to $y$ (or vice versa) and label that arrow with the number $|Q(x,y)|$.
\end{Rmk}

\begin{Ex}\label{quiv-ex1}
    The quiver $Q:\NN\times\NN\to\ZZ$ given diagrammatically by
    % https://q.uiver.app/#q=WzAsNSxbMCwwLCIxIl0sWzEsMCwiMiJdLFsyLDAsIjMiXSxbMywwLCI0Il0sWzQsMCwiXFxkb3RzIl0sWzEsMl0sWzIsM10sWzMsNF0sWzAsMV1d
    \[\begin{tikzcd}[cramped]
    	1 & 2 & 3 & 4 & \dots
    	\arrow[from=1-1, to=1-2]
    	\arrow[from=1-2, to=1-3]
    	\arrow[from=1-3, to=1-4]
    	\arrow[from=1-4, to=1-5]
    \end{tikzcd}\]
    is an arrow finite quiver with vertex set $\NN$. This quiver's underlying undirected graph is an infinite ray, so we say this quiver is \emph{of type $\mathbb A_\infty$}.
\end{Ex}

\begin{Ex}\label{quiv-ex2}
    The quiver $Q:\ZZ^2\times\ZZ^2\to\ZZ$ given diagrammatically by
% https://q.uiver.app/#q=WzAsMjEsWzAsMiwiXFxkb3RzIl0sWzEsMiwiKC0xLDApIl0sWzIsMiwiKDAsMCkiXSxbMywyLCIoMSwwKSJdLFs0LDIsIlxcZG90cyJdLFsxLDEsIigtMSwxKSJdLFsyLDEsIigwLDEpIl0sWzMsMSwiKDEsMSkiXSxbMSwzLCIoLTEsLTEpIl0sWzIsMywiKDAsLTEpIl0sWzMsMywiKDEsLTEpIl0sWzEsMCwiXFx2ZG90cyJdLFsyLDAsIlxcdmRvdHMiXSxbMywwLCJcXHZkb3RzIl0sWzAsMSwiXFxkb3RzIl0sWzQsMSwiXFxkb3RzIl0sWzAsMywiXFxkb3RzIl0sWzQsMywiXFxkb3RzIl0sWzEsNCwiXFx2ZG90cyJdLFsyLDQsIlxcdmRvdHMiXSxbMyw0LCJcXHZkb3RzIl0sWzEsMl0sWzIsM10sWzMsNF0sWzAsMV0sWzEsNV0sWzIsNl0sWzMsN10sWzgsMV0sWzksMl0sWzEwLDNdLFs1LDExXSxbNiwxMl0sWzcsMTNdLFsxNCw1XSxbNywxNV0sWzUsNl0sWzYsN10sWzE2LDhdLFs5LDEwXSxbMTAsMTddLFsxOCw4XSxbMTksOV0sWzIwLDEwXSxbOCw5XV0=
\[\begin{tikzcd}[cramped]
	& \vdots & \vdots & \vdots & \\
	\dots & {(-1,1)} & {(0,1)} & {(1,1)} & \dots \\
	\dots & {(-1,0)} & {(0,0)} & {(1,0)} & \dots \\
	\dots & {(-1,-1)} & {(0,-1)} & {(1,-1)} & \dots \\
	& \vdots & \vdots & \vdots
	\arrow[from=2-1, to=2-2]
	\arrow[from=2-2, to=1-2]
	\arrow[from=2-2, to=2-3]
	\arrow[from=2-3, to=1-3]
	\arrow[from=2-3, to=2-4]
	\arrow[from=2-4, to=1-4]
	\arrow[from=2-4, to=2-5]
	\arrow[from=3-1, to=3-2]
	\arrow[from=3-2, to=2-2]
	\arrow[from=3-2, to=3-3]
	\arrow[from=3-3, to=2-3]
	\arrow[from=3-3, to=3-4]
	\arrow[from=3-4, to=2-4]
	\arrow[from=3-4, to=3-5]
	\arrow[from=4-1, to=4-2]
	\arrow[from=4-2, to=3-2]
	\arrow[from=4-2, to=4-3]
	\arrow[from=4-3, to=3-3]
	\arrow[from=4-3, to=4-4]
	\arrow[from=4-4, to=3-4]
	\arrow[from=4-4, to=4-5]
	\arrow[from=5-2, to=4-2]
	\arrow[from=5-3, to=4-3]
	\arrow[from=5-4, to=4-4]
\end{tikzcd}\]
is an arrow finite quiver with vertex set $\ZZ^2$.
\end{Ex}

\begin{Ex}\label{quiv-ex3}
    The quiver $Q:\mathbb R\times\mathbb R\to\ZZ$ given for all $r,s\in\mathbb R$ by
    $$Q(r,s):=\begin{cases}
        1 & \text{if }s>r\text{ and }r-s\in\mathbb Q\\
        -1 & \text{if }s<r\text{ and }r-s\in\mathbb Q\\
        0 & \text{otherwise}
    \end{cases}$$
    is an arrow finite quiver (since there is at most one arrow between any two vertices) with vertex set $\mathbb R$ in which every vertex has countably infinite degree.
\end{Ex}

In order to define mutations, we have required by Definition~\ref{quiv-def} that quivers are free from loops and oriented 2-cycles, and we have also required that quivers are \emph{arrow finite,} i.e., the number of arrows between any two vertices is finite. Typically, the latter requirement mentioned here is implicitly addressed by asking that quivers are finite, meaning their sets of vertices and arrows are both finite. Naturally, we explicitly choose not to impose this condition in this paper and define finiteness in a slightly different way (see Remark~\ref{fin-quiv-rmk} below).

We may now define quiver mutations in our setting.

\begin{Def}\label{quiv-mut-def}
    Let $Q:X\times X\to\ZZ$ be a quiver without loops or oriented $2$-cycles on a vertex set $X$ and let $x\in X$. We define the \emph{mutation of $Q$ in direction $x$} (or the \emph{mutation of $Q$ at $x$}) to be the quiver $\mu_x(Q)=Q':X\times X\to\ZZ$ on $X$ defined for any $v,w\in X$ as follows:
    $$Q'(v,w)=\begin{cases}
        -Q(v,w) & \text{if }v=x\text{ or }x=w\\
        Q(v,w)+Q(v,x)[Q(x,w)]_++[-Q(v,x)]_+Q(x,w) & \text{if }v\neq x\neq w,
    \end{cases}$$
    where $[a]_+:=\max\{a,0\}$.

    One can alternatively define quiver mutations diagrammatically. Given a quiver $Q:X\times X\to\ZZ$ and a vertex $x\in X$, one can mutate $Q$ at $x$ to obtain $\mu_x(Q)$ via the following three-step process:
    \begin{enumerate}
        \item For every path $v\to x\to w$ of length $2$ centered at $x$, add a new arrow $v\to w$.
        \item Reverse every arrow incident to $x$.
        \item Delete any oriented $2$-cycles that may have arisen in the first step.
    \end{enumerate}
\end{Def}

We note right away that mutation is an \emph{involution,} meaning that $\mu_x(\mu_x(Q))=Q$ for any $x\in X$. Importantly, though, iterated mutations at \emph{different} vertices of a quiver need not cancel one another out. We say that two quivers $Q$ and $Q'$ on a vertex set $X$ are \emph{mutation-equivalent} if there is a \emph{finite} sequence of mutations connecting one to the other. Since individual mutations are involutive it holds that finite mutation sequences are invertible, so this relation truly is an equivalence relation.

\begin{Ex}\label{quiv-mut-ex1}
    Suppose $Q:\{1,2,3\}\times\{1,2,3\}\to\ZZ$ is given diagrammatically by:
    % https://q.uiver.app/#q=WzAsMyxbMCwxLCIxIl0sWzEsMCwiMiJdLFsyLDEsIjMiXSxbMCwxXSxbMSwyXSxbMiwwXV0=
    
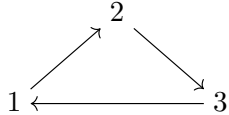
\[\begin{tikzcd}[cramped]
    	& 2 & \\
    	1 && 3
    	\arrow[from=1-2, to=2-3]
    	\arrow[from=2-1, to=1-2]
    	\arrow[from=2-3, to=2-1]
    \end{tikzcd}\]
    Mutating this quiver at the vertex $2$ produces the following quiver $\mu_2(Q):\{1,2,3\}\times\{1,2,3\}\to\ZZ$ given by:
    % https://q.uiver.app/#q=WzAsMyxbMSwwLCIyIl0sWzAsMCwiMSJdLFsyLDAsIjMiXSxbMiwwXSxbMCwxXV0=
    \[\begin{tikzcd}[cramped]
    	1 & 2 & 3
    	\arrow[from=1-2, to=1-1]
    	\arrow[from=1-3, to=1-2]
    \end{tikzcd}\]
\end{Ex}

\begin{Ex}\label{quiv-mut-ex2}
    Let $Q:\NN\times\NN\to\ZZ$ be given diagrammatically by:
    % https://q.uiver.app/#q=WzAsNCxbMSwwLCIyIl0sWzAsMCwiMSJdLFsyLDAsIjMiXSxbMywwLCJcXGRvdHMiXSxbMSwwXSxbMSwyLCIiLDIseyJjdXJ2ZSI6LTF9XSxbMSwzLCIiLDIseyJjdXJ2ZSI6LTJ9XV0=
    
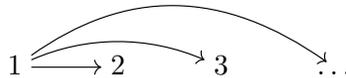
\[\begin{tikzcd}[cramped]
    	1 & 2 & 3 & \dots
    	\arrow[from=1-1, to=1-2]
    	\arrow[curve={height=-12pt}, from=1-1, to=1-3]
    	\arrow[curve={height=-30pt}, from=1-1, to=1-4]
    \end{tikzcd}\]
    where we have a single arrow $1\to n$ for all $n\geq2$. Mutating this quiver at the vertex $1$ produces the following:
    % https://q.uiver.app/#q=WzAsNCxbMSwwLCIyIl0sWzAsMCwiMSJdLFsyLDAsIjMiXSxbMywwLCJcXGRvdHMiXSxbMCwxXSxbMiwxLCIiLDAseyJjdXJ2ZSI6MX1dLFszLDEsIiIsMCx7ImN1cnZlIjoyfV1d
    
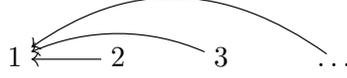
\[\begin{tikzcd}[cramped]
    	1 & 2 & 3 & \dots
    	\arrow[from=1-2, to=1-1]
    	\arrow[curve={height=12pt}, from=1-3, to=1-1]
    	\arrow[curve={height=30pt}, from=1-4, to=1-1]
    \end{tikzcd}\]
    Mutating again at the vertex $3$ subsequently produces:
    % https://q.uiver.app/#q=WzAsNCxbMSwwLCIyIl0sWzAsMCwiMSJdLFsyLDAsIjMiXSxbMywwLCJcXGRvdHMiXSxbMCwxXSxbMywxLCIiLDAseyJjdXJ2ZSI6Mn1dLFsxLDIsIiIsMCx7ImN1cnZlIjotMX1dXQ==
    
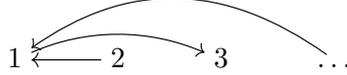
\[\begin{tikzcd}[cramped]
    	1 & 2 & 3 & \dots
    	\arrow[curve={height=-12pt}, from=1-1, to=1-3]
    	\arrow[from=1-2, to=1-1]
    	\arrow[curve={height=30pt}, from=1-4, to=1-1]
    \end{tikzcd}\]
\end{Ex}

Another operation on quivers which will be relevant is \emph{restriction.}

\begin{Def}\label{restr-def}
    Let $Q:X\times X\to\ZZ$ be a quiver on a set $X$ and let $V\subseteq X$ be a subset of vertices. The \emph{restriction} of $Q$ to $V$ is the quiver $\rho_V(Q):=Q':X\times X\to\ZZ$ with vertex set $X$ given by
    $$Q'(x,y)=\begin{cases}
        Q(x,y) & \text{if }x\in V\text{ and }y\in V\\
        0 & \text{otherwise}
    \end{cases}$$
    for all $x,y\in X$. We say that any quiver $Q'$ obtainable from $Q$ in this way is a \emph{full subquiver} of $Q$ and that $Q'$ \emph{embeds} into $Q$, occasionally written $Q'\leq Q$.
\end{Def}

This definition varies slightly from what is traditional: typically one discards the vertices in $X\setminus V$ altogether. We keep them in $\rho_V(Q)$ as isolated vertices intentionally in order to keep the vertex sets of $Q$ and $\rho_V(Q)$ the same. This also allows us to view $\rho_V:\AF^X\to\AF^X$ as a map from $\AF^X$ to itself.

An important property of restriction is that it commutes with mutations at vertices in the corresponding full subquiver:

\begin{Lem}\label{mut-restr-commute-lemma}
    If $Q:X\times X\to\ZZ$ is an arrow finite quiver on a vertex set $X$ and $V\subseteq X$, then $\mu_x(\rho_V(Q))=\rho_V(\mu_x(Q))$ for all $x\in V$. Equivalently, we have that $\mu_x\circ\rho_V=\rho_V\circ\mu_x$ as maps $\AF^X\to\AF^X$.
\end{Lem}
    \begin{proof}
        Let $Q:X\times X\to\ZZ$ be an arrow finite quiver on $X$ and let $V\subseteq X$, $x\in V$ be fixed. For any $v,w\in X$, we have by Definition~\ref{quiv-mut-def} that
        $$\mu_x(Q)(v,w)=\begin{cases}
        -Q(v,w) & \text{if }v=x\text{ or }x=w\\
        Q(v,w)+Q(v,x)[Q(x,w)]_++[-Q(v,x)]_+Q(x,w) & \text{if }v\neq x\neq w.
        \end{cases}$$
        Applying $\rho_V$ then yields by Definition~\ref{restr-def}
        $$\rho_V\circ\mu_x(Q)(v,w)$$
        $$=\begin{cases}
        -Q(v,w) & \text{if }(v=x\text{ or }x=w)\text{ and }(v\in V\text{ and }w\in V)\\
        Q(v,w)+Q(v,x)[Q(x,w)]_++[-Q(v,x)]_+Q(x,w) & \text{if }v\neq x\neq w\text{ and }(v\in V\text{ and }w\in V)\\
        0 & \text{if }v\not\in V\text{ or }w\not\in V.
        \end{cases}$$
        On the other hand, observe that applying $\rho_V$ first gives us
        $$\rho_V(Q)(v,w)=\begin{cases}
        Q(v,w) & \text{if }v\in V\text{ and }w\in V\\
        0 & \text{if }v\not\in V\text{ or }w\not\in V,
        \end{cases}$$
        which, when followed by $\mu_x$, leaves us with
        $$\mu_x\circ\rho_V(Q)(v,w)$$
        $$=\begin{cases}
        -\rho_V(Q)(v,w) & \text{if }v=x\text{ or }x=w\\
        \rho_V(Q)(v,w)+\rho_V(Q)(v,x)[\rho_V(Q)(x,w)]_++[-\rho_V(Q)(v,x)]_+\rho_V(Q)(x,w) & \text{if }v\neq x\neq w
        \end{cases}$$
        $$=\begin{cases}
        -Q(v,w) & \text{if }(v\in V\text{ and }w\in V)\text{ and }(v=x\text{ or }x=w)\\
        Q(v,w)+Q(v,x)[Q(x,w)]_++[-Q(v,x)]_+Q(x,w) & \text{if }(v\in V\text{ and }w\in V)\text{ and }v\neq x\neq w\\
        0 & \text{if }v\not\in V\text{ or }w\not\in V
        \end{cases}$$
        $$=\rho_V\circ\mu_x(Q)(v,w).$$
        Thus, we see that $\rho_V\circ\mu_x(Q)=\mu_x\circ\rho_V(Q)$, proving the claim.
    \end{proof}

The following notion is a modification of the restriction operation that will be used in Section~\ref{sec-3} and beyond. To the best of our knowledge, this definition has not appeared previously in the literature.

\begin{Def}\label{overfull-def}
    Let $Q:X\times X\to\ZZ$ be a quiver on a vertex set $X$ and let $V\subseteq X$ be a subset of vertices of $Q$. The \emph{overfull subquiver of $Q$ induced by $V$} (or just the \emph{overfill of $Q$ on $V$}) is the quiver $\lambda_V(Q):=Q':X\times X\to\ZZ$ on the vertex set $X$ given by
    $$Q'(x,y)=\begin{cases}
        Q(x,y) & \text{if }x\in V\text{ or }y\in V\\
        0 & \text{otherwise}
    \end{cases}$$
    for all $x,y\in X$. We say that any quiver $Q'$ obtainable from a quiver $Q$ as the overfill on some subset $X'\subseteq X$ is an \emph{overfull subquiver} of $Q$.

    As with the restriction maps $\rho_V:\AF^X\to\AF^X$, we similarly view the operation of taking the overfill of a quiver on some set of vertices $V\subseteq X$ as a map $\lambda_V:\AF^X\to\AF^X$.
\end{Def}

We have a lemma similar to Lemma~\ref{mut-restr-commute-lemma} describing the relationship between taking the overfill of a quiver on a set of vertices $V$ and mutating within $V$:

\begin{Lem}\label{mut-overfill-near-commute-lemma}
    If $Q:X\times X\to\ZZ$ is an arrow finite quiver on a vertex set $X$ and $V\subseteq X$, then $\lambda_V(\mu_x(\lambda_V(Q)))=\lambda_V(\mu_x(Q))$ for all $x\in V$. Equivalently, we have that $\lambda_V\circ\mu_x\circ\lambda_V=\lambda_V\circ\mu_x$ as maps $\AF^X\to\AF^X$.
\end{Lem}
    \begin{proof}
        We proceed similarly to the proof of Lemma~\ref{mut-restr-commute-lemma}. Let $Q:X\times X\to\ZZ$ be an arrow finite quiver on a set $X$ and let $V\subseteq X$, $x\in V$ be fixed. For any $v,w\in X$, we see by Definition~\ref{quiv-mut-def} that
        $$\mu_x(Q)(v,w)=\begin{cases}
        -Q(v,w) & \text{if }v=x\text{ or }x=w\\
        Q(v,w)+Q(v,x)[Q(x,w)]_++[-Q(v,x)]_+Q(x,w) & \text{if }v\neq x\neq w.
        \end{cases}$$
        By Definition~\ref{overfull-def}, applying $\lambda_V$ then gives us
        $$\lambda_V\circ\mu_x(Q)(v,w)$$
        $$=\begin{cases}
        -Q(v,w) & \text{if }(v=x\text{ or }x=w)\text{ and }(v\in V\text{ or }w\in V)\\
        Q(v,w)+Q(v,x)[Q(x,w)]_++[-Q(v,x)]_+Q(x,w) & \text{if }v\neq x\neq w\text{ and }(v\in V\text{ or }w\in V)\\
        0 & \text{if }v\not\in V\text{ and }w\not\in V.
        \end{cases}$$
        On the other hand, if we begin with $\lambda_V$, we see that
        $$\lambda_V(Q)(v,w)=\begin{cases}
        Q(v,w) & \text{if }v\in V\text{ or }w\in V\\
        0 & \text{if }v\not\in V\text{ and }w\not\in V.
        \end{cases}$$
        Applying $\mu_x$ doesn't \emph{quite} get us back to $\lambda_V\circ\mu_x(Q)$, since there can be paths of length two in $\lambda_V(Q)$ of the form $v\to x\to w$ with $v,w\not\in V$. Indeed, we have
        $$\mu_x\circ\lambda_V(Q)(v,w)$$
        $$=\begin{cases}
        -\lambda_V(Q)(v,w) & \text{if }v=x\text{ or }x=w\\
        \lambda_V(Q)(v,w)+\lambda_V(Q)(v,x)[\lambda_V(Q)(x,w)]_++[-\lambda_V(Q)(v,x)]_+\lambda_V(Q)(x,w) & \text{if }v\neq x\neq w
        \end{cases}$$
        $$=\begin{cases}
        -Q(v,w) & \text{if }(v\in V\text{ or }w\in V)\text{ and }(v=x\text{ or }x=w)\\
        Q(v,w)+Q(v,x)[Q(x,w)]_++[-Q(v,x)]_+Q(x,w) & \text{if }(v\in V\text{ or }w\in V)\text{ and }v\neq x\neq w\\
        0+Q(v,x)[Q(x,w)]_++[-Q(v,x)]_+Q(x,w) & \text{if }v\not\in V\text{ and }w\not\in V.
        \end{cases}$$
        Now we apply $\lambda_V$ once more to zero out the final case:
        $$\lambda_V\circ\mu_x\circ\lambda_V(Q)(v,w)$$
        $$=\begin{cases}
        -Q(v,w) & \text{if }(v\in V\text{ or }w\in V)\text{ and }(v=x\text{ or }x=w)\\
        Q(v,w)+Q(v,x)[Q(x,w)]_++[-Q(v,x)]_+Q(x,w) & \text{if }(v\in V\text{ or }w\in V)\text{ and }v\neq x\neq w\\
        0 & \text{if }v\not\in V\text{ and }w\not\in V
        \end{cases}$$
        $$=\lambda_V\circ\mu_x(Q)(v,w).$$
        This proves the claim.
    \end{proof}

\begin{Ex}\label{restr-overfull-ex}
    Let $Q:\ZZ\times\ZZ\to\ZZ$ be the quiver with vertex set $\ZZ$ given diagrammatically by:
    % https://q.uiver.app/#q=WzAsNyxbMCwwLCJcXGRvdHMiXSxbMSwwLCItMiJdLFsyLDAsIi0xIl0sWzMsMCwiMCJdLFs0LDAsIjEiXSxbNSwwLCIyIl0sWzYsMCwiXFxkb3RzIl0sWzMsMl0sWzMsMSwiIiwyLHsiY3VydmUiOjJ9XSxbMywwLCIiLDIseyJjdXJ2ZSI6NH1dLFsyLDFdLFsxLDBdLFszLDRdLFs0LDVdLFs1LDZdLFszLDUsIiIsMCx7ImN1cnZlIjotMn1dLFszLDYsIiIsMCx7ImN1cnZlIjotNH1dXQ==
\[\begin{tikzcd}[cramped]
	\dots & {-2} & {-1} & 0 & 1 & 2 & \dots
	\arrow[from=1-2, to=1-1]
	\arrow[from=1-3, to=1-2]
	\arrow[curve={height=24pt}, from=1-4, to=1-1]
	\arrow[curve={height=12pt}, from=1-4, to=1-2]
	\arrow[from=1-4, to=1-3]
	\arrow[from=1-4, to=1-5]
	\arrow[curve={height=-12pt}, from=1-4, to=1-6]
	\arrow[curve={height=-24pt}, from=1-4, to=1-7]
	\arrow[from=1-5, to=1-6]
	\arrow[from=1-6, to=1-7]
\end{tikzcd}\]
    in which there is a single arrow $0\to n$ for all $n\neq0$, a single arrow $n\to n+1$ if $n>0$, and a single arrow $n\to n-1$ if $n<0$. The restriction of $Q$ to the subset $\{-1,0,1\}\subseteq\ZZ$ is the quiver $Q'$ given diagrammatically by:
    % https://q.uiver.app/#q=WzAsNyxbMCwwLCJcXGRvdHMiXSxbMSwwLCItMiJdLFsyLDAsIi0xIl0sWzMsMCwiMCJdLFs0LDAsIjEiXSxbNSwwLCIyIl0sWzYsMCwiXFxkb3RzIl0sWzMsMl0sWzMsNF1d
    \[\begin{tikzcd}[cramped]
    	\dots & {-2} & {-1} & 0 & 1 & 2 & \dots
    	\arrow[from=1-4, to=1-3]
    	\arrow[from=1-4, to=1-5]
    \end{tikzcd}\]
    while the overfill of $Q$ on the subset $\{-1,0,1\}\subseteq\ZZ$ is the quiver $Q''$ given diagrammatically by:
    % https://q.uiver.app/#q=WzAsNyxbMCwwLCJcXGRvdHMiXSxbMSwwLCItMiJdLFsyLDAsIi0xIl0sWzMsMCwiMCJdLFs0LDAsIjEiXSxbNSwwLCIyIl0sWzYsMCwiXFxkb3RzIl0sWzMsMl0sWzMsMSwiIiwyLHsiY3VydmUiOjJ9XSxbMywwLCIiLDIseyJjdXJ2ZSI6NH1dLFsyLDFdLFszLDRdLFs0LDVdLFszLDUsIiIsMCx7ImN1cnZlIjotMn1dLFszLDYsIiIsMCx7ImN1cnZlIjotNH1dXQ==
\[\begin{tikzcd}[cramped]
	\dots & {-2} & {-1} & 0 & 1 & 2 & \dots
	\arrow[from=1-3, to=1-2]
	\arrow[curve={height=24pt}, from=1-4, to=1-1]
	\arrow[curve={height=12pt}, from=1-4, to=1-2]
	\arrow[from=1-4, to=1-3]
	\arrow[from=1-4, to=1-5]
	\arrow[curve={height=-12pt}, from=1-4, to=1-6]
	\arrow[curve={height=-24pt}, from=1-4, to=1-7]
	\arrow[from=1-5, to=1-6]
\end{tikzcd}\]
\end{Ex}

Note that the restriction of a quiver $Q$ to a set $X'\subseteq X$ of vertices is a full subquiver of the overfill of $Q$ on $X'$, but overfull subquivers of $Q$ need not be full subquivers of $Q$. The differences between full and overfull subquivers will allow us to define the weak and strong topologies on our spaces of interest, respectively.

\begin{Rmk}\label{fin-quiv-rmk}
    Even in the case when $X$ is infinite, one may still be interested in ``finite'' quivers with vertex set $X$. For any set $X$, we say an arrow finite quiver $Q:X\times X\to\ZZ$ is \emph{finite} if and only if
    $$\sum_{x,y\in X}|Q(x,y)|<\infty,$$
    or in other words, $Q$ only has finitely many arrows. We may equivalently characterize finite quivers on $X$ as the $Q$ whose \emph{support} $\text{supp}(Q):=\{x\in X\mid\exists y\in X:Q(x,y)\neq0\}\subseteq X$ is finite.

    We let $\mathbf{Fin}^X$ denote the set of all finite quivers without loops or oriented $2$-cycles on a given vertex set $X$.
\end{Rmk}

\begin{Ex}\label{fin-quiv-ex}
    In Example~\ref{restr-overfull-ex} above, the quiver $Q'$ is finite (while the quivers $Q$ and $Q''$ are not finite).
\end{Ex}

In addition to the notions of finiteness and arrow finiteness, we will need an intermediate finiteness condition.

\begin{Def}\label{loc-fin-def}
    Let $Q:X\times X\to\ZZ$ be an arrow finite quiver on a vertex set $X$. We say that $Q$ is \emph{locally finite} if the degree of every vertex in $Q$ is finite, meaning
    $$\sum_{y\in X}|Q(x,y)|<\infty$$
    for all $x\in X$.
    
    We define $\LF^X$ as the set of all locally finite quivers without loops or oriented $2$-cycles on a given vertex set $X$.
\end{Def}

Note that every finite quiver on $X$ is also locally finite and that every locally finite quiver is arrow finite, giving a chain of inclusions $\mathbf{Fin}^X\subseteq\LF^X\subseteq\AF^X$. Furthermore, it is straightforward to check that these sets are all closed under mutation under our running hypothesis that our quivers do not have loops or oriented $2$-cycles. Thus, by specifying a vertex $x\in X$ to mutate at, we obtain an involution $\mu_x:\AF^X\to\AF^X$ which restricts to involutions on $\LF^X$ and $\mathbf{Fin}^X$ given by sending a quiver $Q$ to $\mu_x(Q)$. In the next section, we topologize $\AF^X$ with two related topologies, then show that these mutation maps are continuous.

%% Section 3: Topologizing AF^X and mutation maps %%
\section{Topologizing $\AF^X$ and mutation maps}\label{sec-3}

In this section, we provide two topologies on $\AF^X$ and show that mutation at any vertex $x\in X$ provides a nontrivial self-homeomorphism $\mu_x:\AF^X\to\AF^X$ under either topology. We first introduce some notation in order to do so.

\begin{Def}\label{U-W-def}
    Let $Q:X\times X\to\ZZ$ be a quiver on a set $X$, and let $V\subseteq X$ be a finite subset of vertices. We define $U_{Q,V}\subseteq\AF^X$ as the set of all $Q'\in\AF^X$ such that the restriction of $Q'$ to $V$ agrees with the restriction of $Q$ to $V$:
    $$U_{Q,V}:=\left\{Q'\in\AF^X\mid\rho_V(Q')=\rho_V(Q)\right\}.$$
    Similarly, we let $W_{Q,V}\subseteq\AF^X$ be the set of all $Q'\in\AF^X$ such that the overfill of $Q'$ on $V$ agrees with the overfill of $Q$ on $V$:
    $$W_{Q,V}:=\left\{Q'\in\AF^X\mid\lambda_V(Q')=\lambda_V(Q)\right\}.$$
    Note that for any $Q\in\AF^X$ and any nested finite subsets $V\subseteq V'\subseteq X$, it holds right away that $U_{Q,V'}\subseteq U_{Q,V}$ and $W_{Q,V'}\subseteq W_{Q,V}$.
\end{Def}

The following ``re-anchoring'' lemma will enable us to make useful reductions in our arguments.

\begin{Lem}\label{U-W-reanchor-lemma}
    Fix $Q\in\AF^X$ and $V\subseteq X$ finite. Then for any $Q'\in\AF^X$,
    \begin{enumerate}[label=(\roman*)]
        \item $U_{Q',V}=U_{Q,V}$ if and only if $Q'\in U_{Q,V}$.
        \item $W_{Q',V}=W_{Q,V}$ if and only if $Q'\in W_{Q,V}$.
    \end{enumerate}
\end{Lem}
\begin{proof}
    We prove (i) and note that the proof of (ii) is entirely the same. Suppose first that $U_{Q',V}=U_{Q,V}$. Since $Q'\in U_{Q',V}$ (as follows immediately by Definition~\ref{U-W-def}), we have that $Q'\in U_{Q,V}$. Conversely, suppose that $Q'\in U_{Q,V}$. By Definition~\ref{U-W-def} above, this means that $\rho_V(Q')=\rho_V(Q)$. Now fix $Q''\in\AF^X$. One has now that $\rho_V(Q'')=\rho_V(Q')$ if and only if $\rho_V(Q'')=\rho_V(Q)$. In other words, $Q''\in U_{Q',V}$ if and only if $Q''\in U_{Q,V}$, yielding $U_{Q',V}=U_{Q,V}$.
\end{proof}

Another lemma which will be used in the proof of Lemma~\ref{weak-str-comp-lemma} below is the following.

\begin{Lem}\label{W-in-U-lemma}
    Let $Q\in\AF^X$ and let $V\subseteq X$ be finite. Then $W_{Q,V}\subseteq U_{Q,V}$.
\end{Lem}
\begin{proof}
    This essentially follows from the fact that the restriction $\rho_V(Q)$ of $Q$ to $V$ agrees with the restriction $\rho_V(\lambda_V(Q))$ on $V$ of the overfill of $Q$ on $V$ (which the reader is encouraged to check using Definitions~\ref{restr-def} and \ref{overfull-def}). Then if $Q'\in W_{Q,V}$, by Definition~\ref{U-W-def} we have $\lambda_V(Q')=\lambda_V(Q)$. Taking the restrictions to $V$ of these identical overfills then shows that
    $$\rho_V(Q')=\rho_V(\lambda_V(Q'))=\rho_V(\lambda_V(Q))=\rho_V(Q).$$
    By definition, this means that $Q'\in U_{Q,V}$, so we have $W_{Q,V}\subseteq U_{Q,V}$.
\end{proof}

We now observe the following key fact which gives us our topologies of interest.

\begin{Lem}\label{U-W-basis-lemma}
    Let $X$ be a set. The families $\mathcal U_X:=\{U_{Q,V}\mid Q\in\AF^X,V\subseteq X\text{ finite}\}$ and $\mathcal W_X:=\{W_{Q,V}\mid Q\in\AF^X,V\subseteq X\text{ finite}\}$ each form bases for topology on $\AF^X$.
\end{Lem}
    \begin{proof}
        We prove the claim for the family $\mathcal U_X$ only, as the proof is essentially identical for $\mathcal W_X$. Since $\mathcal U_X$ is readily seen to cover $\AF^X$ (as for all $Q\in\AF^X$, one has $Q\in U_{Q,V}$ for any finite $V\subseteq X$), we only must show that for any two members $U_{Q,V},U_{Q',V'}$ of $\mathcal U_X$ and any $Q''\in U_{Q,V}\cap U_{Q',V'}$, there exists $V''\subseteq X$ finite such that $U_{Q'',V''}\subseteq U_{Q,V}\cap U_{Q',V'}$.
        
        Fix some $Q''\in U_{Q,V}\cap U_{Q',V'}$ and let $V''=V\cup V'$. Then it holds that $U_{Q'',V''}\subseteq U_{Q'',V}$ and $U_{Q'',V''}\subseteq U_{Q'',V'}$, so we have $U_{Q'',V''}\subseteq U_{Q'',V}\cap U_{Q'',V'}$. The claim now follows from the fact that $U_{Q'',V}=U_{Q,V}$ and $U_{Q'',V'}=U_{Q',V'}$ by Lemma~\ref{U-W-reanchor-lemma}(i).
    \end{proof}

\begin{Def}\label{weak-str-tops-def}
    Let $X$ be a set. The \emph{weak topology} on $\AF^X$ is the topology having $\mathcal U_X$ as a basis; we use the notation $\AF^X_{weak}$ to denote this space. Similarly, the \emph{strong topology} on $\AF^X$ is the topology having $\mathcal W_X$ as a basis; we use the notation $\AF^X_{str}$ to denote this space.
\end{Def}

The terms ``weak'' and ``strong'' are justified by the following useful comparison lemma:

\begin{Lem}\label{weak-str-comp-lemma}
    Let $X$ be a set. The strong topology on $\AF^X$ is finer than the weak topology on $\AF^X$. That is, if $U\subseteq\AF^X_{weak}$ is open, then $U$ is open in $\AF^X_{str}$ as well. Equivalently, the identity map $\AF^X_{str}\to\AF^X_{weak}$ is continuous.
\end{Lem}
    \begin{proof}
        Without loss of generality, we may take $U=U_{Q,V}$ to be a basic open set in the weak topology for some $Q\in\AF^X$ and $V\subseteq X$ finite. To show that $U_{Q,V}$ is open in the strong topology, fix $Q'\in U_{Q,V}$. By Lemma~\ref{U-W-reanchor-lemma}(i) above, we know that $U_{Q',V}=U_{Q,V}$. Then by Lemma~\ref{W-in-U-lemma} above, we also have $W_{Q',V}\subseteq U_{Q',V}$. Thus, we have obtained a strong-open neighborhood $W_{Q',V}$ of $Q'$ contained entirely in the weak-open neighborhood $U_{Q,V}$. Since $Q'\in U_{Q,V}$ was chosen arbitrarily, this shows that the strong topology refines the weak topology.
    \end{proof}

\begin{Rmk}\label{strict-refinement-rmk}
One may reasonably ask if the strong topology \emph{strictly} refines the weak topology. When $X$ is finite, this is not the case (when $X$ is finite, the reader may enjoy verifying that $\AF^X_{weak}$ is a discrete space using Definition~\ref{U-W-def} and therefore its topology cannot be strictly refined). On the other hand, when $X$ is infinite, this refinement is indeed strict. We verify this by finding an infinite sequence which converges in the weak topology but not the strong topology. Let $\{x_1,x_2,\dots\}\subseteq X$ be a countably infinite subset of $X$, and for all $n\geq1$, let $Q_n\in\AF^X$ be the quiver obtained by placing a single arrow from $x_1$ to $x_{n+1}$. Then two things are simultaneously true:
\begin{enumerate}
    \item for any $V\subseteq X$ finite, the full subquiver $\rho_V(Q_n)$ of $Q_n$ stabilizes to the quiver with no arrows for large enough $n$, so $(Q_n)_{n\geq1}$ converges to the quiver with no arrows in $\AF^X_{weak}$.
    \item for any $V\subseteq X$ finite containing $x_1$, the overfull subquiver $\lambda_V(Q_n)$ of $Q_n$ fails to stabilize, guaranteeing the divergence of $(Q_n)_{n\geq1}$ in $\AF^X_{str}$.
\end{enumerate}
This distinguishes the weak and strong topologies, leading us to conclude that the strong topology \emph{strictly} refines the weak topology in the case that $X$ is infinite.
\end{Rmk}

Another important property is the following:

\begin{Lem}\label{LF-dense-in-AF-lemma}
    Let $X$ be a set. Then $\mathbf{Fin}^X$ (and therefore also $\LF^X$) is dense in $\AF^X_{weak}$. If $X$ is infinite, however, then $\LF^X$ is \emph{not} dense in $\AF^X_{str}$ (and therefore neither is $\mathbf{Fin}^X$).
\end{Lem}
    \begin{proof}
        We show first that $\mathbf{Fin}^X$ is dense in $\AF^X_{weak}$. Let $Q\in\AF^X$ and take $V\subseteq X$ finite. We wish to show that $U_{Q,V}$ contains a finite quiver. Indeed, the restriction $Q'$ of $Q$ to $V$ is a finite quiver belonging to $U_{Q,V}$. Thus, $\mathbf{Fin}^X\cap U_{Q,V}\neq\varnothing$.

        On the other hand, suppose $X$ is infinite. We build a strong-basic open set $W_{Q,V}$ for some $Q\in\AF^X$ and some finite $V\subseteq X$ containing no locally finite quivers. Fix a sequence of distinct vertices $x_0,x_1,x_2,\dots\in X$, and let $Q$ be the quiver with one arrow from $x_0$ to $x_i$ for all $i\geq 1$ and no arrows elsewhere. Then the degree of $x_0$ in any quiver $Q'\in W_{Q,\{x_0\}}$ is infinite, so $W_{Q,\{x_0\}}\cap\LF^X=\varnothing$.
    \end{proof}

Next, we turn our attention to mutation maps.

\begin{Prop}\label{mut-cont-prop}
    Let $X$ be a set and $x\in X$. The map $\mu_x:\AF^X_{weak}\to\AF^X_{weak}$ given by sending $Q\in\AF^X$ to $\mu_x(Q)\in\AF^X$ is continuous. Since $\mu_x$ is an involution, it therefore also holds that $\mu_x$ is a homeomorphism. Moreover, both of these statements are true for mutation maps when replacing $\AF^X_{weak}$ with $\AF^X_{str}$.
\end{Prop}
    \begin{proof}
        Let $U_{Q,V}$ be a basic open set in the weak topology on $\AF^X$ and fix $x\in X$. We wish to show that $\mu_x^{-1}(U_{Q,V})$ is open in $\AF^X_{weak}$. By extending $V$ to $V\cup\{x\}$ and thus shrinking $U_{Q,V}$ to $U_{Q,V\cup\{x\}}\subseteq U_{Q,V}$ if necessary, we may assume without loss of generality that $x\in V$. Indeed, expressing $U_{Q,V}$ as the union over all such ways to extend the restriction of $Q$ on $V$ to a quiver on $V\cup\{x\}$ yields:
        $$U_{Q,V}=\bigcup_{P\in U_{Q,V}}U_{P,V\cup\{x\}}.$$
        From this, it follows that if $\mu_x^{-1}(U_{P,V\cup\{x\}})$ is open in $\AF$ for all $P\in U_{Q,V}$, then so is $\mu_x^{-1}(U_{Q,V})$. Hence, we may safely assume $x\in V$.
        
        To show that $\mu_x^{-1}(U_{Q,V})$ is open in $\AF^X_{weak}$, take an arbitrary $Q'\in\mu_x^{-1}(U_{Q,V})$. By definition, this means that the restriction of $\mu_x(Q')$ to $V$ agrees with the restriction of $Q$ to $V$. By Lemma~\ref{U-W-reanchor-lemma}(i), it is equivalent to say that $U_{Q,V}=U_{\mu_x(Q'),V}$. Therefore, we may assume without loss of generality that $Q=\mu_x(Q')$ (since we have not committed to $Q$ itself, just to the basic open set $U_{Q,V}$). Now, using the fact that mutation within $V$ and restriction to $V$ commute by Lemma~\ref{mut-restr-commute-lemma}, we claim that $U_{\mu_x(Q'),V}=\mu_x(U_{Q',V})$. Indeed, if $Q''\in U_{Q',V}$, then $\rho_V(Q'')=\rho_V(Q')$ by definition. Using the fact from Lemma~\ref{mut-restr-commute-lemma} that $\mu_x\circ\rho_V$ since $x\in V$, we have that
        $$\rho_V(\mu_x(Q''))=\mu_x(\rho_V(Q''))=\mu_x(\rho_V(Q'))=\rho_V(\mu_x(Q')),$$
        giving us $\mu_x(U_{Q',V})\subseteq U_{\mu_x(Q'),V}$. The reverse containment is identical, so indeed we have that $U_{Q,V}=U_{\mu_x(Q'),V}=\mu_x(U_{Q',V})$. Applying the bijection $\mu_x^{-1}$ to both sides yields $\mu_x^{-1}(U_{Q,V})=U_{Q',V}$, which is open in $\AF^X_{weak}$. This shows that $\mu_x:\AF^X_{weak}\to\AF^X_{weak}$ is continuous.

        The case of the strong topology is very similar. Let $W_{Q,V}$ be a basic open set in $\AF^X_{str}$ and fix $x\in X$. Our goal is to show that $\mu_x^{-1}(W_{Q,V})$ is open in $\AF^X_{str}$. As above, we may safely assume that $x\in V$. Taking an arbitrary $Q'\in\mu_x^{-1}(W_{Q,V})$ and proceeding like before, we have by Lemma~\ref{U-W-reanchor-lemma}(ii) that $W_{Q,V}=W_{\mu_x(Q'),V}$; we may again assume without loss of generality that $Q=\mu_x(Q')$. Suppose $Q''\in W_{Q',V}$. By definition, this means $\lambda_V(Q'')=\lambda_V(Q')$. From Lemma~\ref{mut-overfill-near-commute-lemma} and the fact that $x\in V$, we know that $\lambda_V\circ\mu_x\circ\lambda_V=\lambda_V\circ\mu_x$ as maps $\AF^X\to\AF^X$. Applying $\lambda_V\circ\mu_x$ to both sides of $\lambda_V(Q'')=\lambda_V(Q')$ and then using this identity, we see that $\lambda_V(\mu_x(Q''))=\lambda_V(\mu_x(Q'))$. In other words, $\mu_x(Q'')\in W_{\mu_x(Q'),V}$ for all $Q''\in W_{Q',V}$. This shows that $\mu_x(W_{Q',V})\subseteq W_{\mu_x(Q'),V}$. By applying the same argument the other way around, we obtain the reverse containment, so in fact $\mu_x(W_{Q',V})=W_{\mu_x(Q'),V}=W_{Q,V}$. Applying the bijection $\mu_x^{-1}$ to both sides gives $\mu_x^{-1}(W_{Q,V})=W_{Q',V}$, which is open in $\AF^X_{str}$. This completes the proof that $\mu_x:\AF^X_{str}\to\AF^X_{str}$ is continuous.

        In both cases, $\mu_x$ is a continuous involution. In particular, its inverse $\mu_x^{-1}=\mu_x$ is also a continuous map, so in fact we see that $\mu_x$ is a homeomorphism in both cases as well.
    \end{proof}

We end this section by proving a central theorem about the spaces $\LF^X_{str}$ and $\AF^X_{weak}$ when $X$ is countably infinite. To do so, we first make a few comments about the Baire space $\NN^\NN$.

\begin{Rmk}\label{N^N-Polish-rmk}
    $\NN^\NN$ is an uncountable \emph{Polish space,} which is a separable, completely metrizable topological space. That is, $\NN^\NN$ has a countable dense subset (separability), and there exists a complete metric $d$ on $\NN^\NN$ inducing its topology (complete metrizability). Polish spaces are some of the primary objects of interest in descriptive set theory, a branch of mathematical logic partially concerned with classification problems similar to the ones we are interested in. Many naturally arising topological spaces are Polish spaces: topological manifolds, separable Banach spaces, automorphism groups of countable first-order structures, and so on. We will discuss some tools from descriptive set theory in Section~\ref{sec-4} when we examine properties of infinite quivers from a topological point of view.
\end{Rmk}

\begin{Rmk}\label{N^N-bases-rmk}
    There are a couple good choices of basis for topology for $\NN^\NN$: the \emph{product basis} and the \emph{initial segment basis.} The product basis consists of the sets of the form $U_1\times U_2\times U_3\times\dots\subseteq\NN^\NN$, where $U_i\subseteq\NN$ for all $i\in\NN$ and all but finitely many $U_i$ are equal to $\NN$. Using the fact that $\NN$ is a discrete space, one can shrink this basis a fair amount by requiring any $U_i$ not equal to $\NN$ to be a singleton. On the other hand, the \emph{initial segment basis} is parametrized by the set $\NN^{<\NN}$ of finite strings of natural numbers $\sigma=(\sigma_1,\sigma_2,\dots,\sigma_n)$. The basic open set indexed by $\sigma$ is taken to be
    $$[\![\sigma]\!]:=\{\sigma_1\}\times\{\sigma_2\}\times\dots\times\{\sigma_n\}\times\NN\times\NN\times\dots.$$
    Of course, these sets are also members of the product basis. Conversely, if $U=U_1\times U_2\times\dots\subseteq\NN^\NN$ is a product-basic open set containing some $f\in\NN^\NN$, one may find an initial-segment-basic open set $[\![f\!\upharpoonright\!n]\!]\subseteq U$ containing $f$. Here, $n=\max\{i\in\NN\,:\,U_i\neq\NN\}$ or $n=0$ if no such $i\in\NN$ exists, and $f\!\upharpoonright\!n:=(f_1,f_2,\dots,f_n)\in\NN^{<\NN}$ is the restriction of $f\in\NN^\NN$ considered as a countably infinite string $f=(f_1,f_2,f_3,\dots)$ to its initial segment of length $n$, taking the convention that $f\!\upharpoonright\!0:=\varnothing$. This shows that the topologies on $\NN^\NN$ generated by either basis are the same.
\end{Rmk}

\begin{Rmk}
    Given $f,g\in\NN^\NN$, define
    $$d(f,g):=\begin{cases}
        0 & \text{if }f=g\\
        2^{-\min\{i\,:\,f_i\neq g_i\}} & \text{if }f\neq g.
    \end{cases}$$
    One can show without too much effort that this defines a complete metric on $\NN^\NN$ generating the same topology as the bases defined above. In fact, for $\varepsilon>0$ and $f\in\NN^\NN$, it holds that the $d$-ball $B_d(f,\varepsilon)\subseteq\NN^\NN$ of radius $\varepsilon$ around $f\in\NN^\NN$ is given by $B_d(f,\varepsilon)=[\![f\!\upharpoonright\!n]\!]$, where $n=\max\{0,\lceil-\log_2(\varepsilon)\rceil\}$. Thus, it holds that $\NN^\NN$ is completely metrizable. To see that $\NN^\NN$ is separable (and therefore a Polish space), note that there are only countably many sequences $f=(f_1,f_2,f_3,\dots)\in\NN^\NN$ which are eventually constant and that they form a dense subspace of $\NN^\NN$: any sequence $g\in\NN^\NN$ may be approximated arbitrarily well by eventually constant sequences.
\end{Rmk}

We are now ready to prove our first main result, Theorem~\ref{N^N-homeo-thm}.

\begin{T-homeo}
    Let $X$ be a countably infinite set. Then the spaces $\AF^X_{weak}$ and $\LF^X_{str}$ are both homeomorphic to $\NN^\NN$.
\end{T-homeo}
    \begin{proof}
        We take $X=\NN$ for simplicity and without loss of generality. We will first show that $\AF^\NN_{weak}$ is homeomorphic to $\ZZ^{[\NN]^2}$, where $[\NN]^2$ is the set of unordered pairs $\{i,j\}$ of distinct natural numbers. (Since these pairs are unordered, we may assume without loss of generality that $i<j$; we write $\{i<j\}$ instead of just $\{i,j\}$ to remind ourselves of this.) Of course, the space $\ZZ^{[\NN]^2}$ is a countable product of countably many copies of the countably infinite discrete space $\ZZ$, so $\AF^\NN_{weak}$ will then in turn be homeomorphic to $\NN^\NN$ as claimed.
        
        We claim that the map sending an arrow finite quiver $Q:\NN\times\NN\to\ZZ$ to the function sending an unordered pair $\{i<j\}$ to $Q(i,j)\in\ZZ$ is a homeomorphism $\AF^\NN_{weak}\to\ZZ^{[\NN]^2}$. One may readily check that this is a bijection, as any skew-symmetric function $Q:\NN\times\NN\to\ZZ$ gives rise to a unique function $f_Q:[\NN]^2\to\ZZ$ under this map and vice versa. To see that this map is continuous, let $U=\prod_{\{i<j\}\in[\NN]^2}U_{\{i<j\}}\subseteq\ZZ^{[\NN]^2}$ be a basic open set in the product topology. Let $Q:\NN\times\NN\to\ZZ$ be a quiver such that $Q(i,j)\in U_{\{i<j\}}$ for all $\{i<j\}\in[\NN]^2$. We wish to find a finite set $V$ such that all $Q'\in U_{Q,V}$ have this property as well, i.e., such that $Q'(i,j)\in U_{\{i<j\}}$ for all $\{i<j\}\in[\NN]^2$. Define $J=\{\{i<j\}\mid U_{\{i<j\}}\neq\ZZ\}\subseteq[\NN]^2$ and let
        $$I=\left\{i\in\NN\mid\exists j\in\NN:\{i<j\}\in J\text{ or }\{j<i\}\in J\right\}.$$
        Then we claim that $V=I$ has the desired property. Indeed, if $Q'\in U_{Q,I}$, then $Q'(x,y)=Q(x,y)$ for all $x,y\in I$ by definition. In particular, if $\{x<y\}\in J$, then $Q'(x,y)=Q(x,y)$. Thus, we have that for $\{i<j\}\in J$, it holds that $Q'(i,j)=Q(i,j)\in U_{\{i<j\}}$. On the other hand, if $\{i<j\}\not\in J$, we have by the definition of $J$ that $U_{\{i<j\}}=\ZZ$, and therefore $Q'(i,j)\in U_{\{i<j\}}$ automatically. This shows that our bijection $\AF^\NN_{weak}\to\ZZ^{[\NN]^2}$ is continuous. To argue that it is open (and therefore a homeomorphism) is even quicker. Indeed, the image of any basic open set $U_{Q,V}$ in $\AF^\NN_{weak}$ under our bijection is the set
        $$\prod_{\{i<j\}\in[V]^2}\{Q(i,j)\}\times\prod_{\{i<j\}\in[\NN]^2\setminus[V]^2}\ZZ.$$
        Since $V$ is finite, this set is open in the product topology on $\ZZ^{[\NN]^2}$. This completes the proof that $\AF^\NN_{weak}\simeq\NN^\NN$.

        We now focus on showing that $\LF^\NN_{str}\simeq\NN^\NN$. Unlike in the above case, we do this directly in terms of the space $\NN^\NN$. For $k\geq 1$, let $p_k$ be the $k^{th}$ prime number. Let $h:\ZZ\to\ZZ_{\geq0}$ be the bijection given by
        $$h(t):=\begin{cases}
            2t & \text{if }t\geq0 \\
            -2t-1 & \text{otherwise.}
        \end{cases}$$
        Importantly, we see that $h(0)=0$. We define a map $L:\LF^\NN_{str}\to\NN^\NN$ as follows for all $Q\in\LF^\NN$:
        $$L(Q)(i):=\prod_{k=1}^\infty p_k^{h(Q(i,i+k))}$$
        for all $i\in\NN$. Observe that since $Q$ is locally finite and $h(0)=0$, the above infinite product is in fact finite for each $i\in\NN$, so the map $L$ is well-defined. We claim that this map is a homeomorphism $\LF^\NN_{str}\to\NN^\NN$. One may see that this map is bijective using the fundamental theorem of arithmetic and the bijectivity of $h$; the exponents in the prime factorization of $L(Q)(i)$ encode the counts of arrows between vertices $i\in\NN$ and all vertices $j>i$. To see that this map is continuous, suppose $[\![\sigma]\!]\subseteq\NN^\NN$ is a basic open set in the initial segment basis on $\NN^\NN$ for some $\sigma=(\sigma_1,\dots,\sigma_{n})\in\NN^{<\NN}$ and let $Q\in L^{-1}([\![\sigma]\!])$. We claim that $W_{Q,\{1,\dots,n\}}\cap\LF^\NN\subseteq L^{-1}([\![\sigma]\!])$. Indeed, if $Q'\in W_{Q,\{1,\dots,n\}}$ is locally finite, one sees that $Q'(i,i+k)=Q(i,i+k)$ for all $i\leq n$ and all $k\in\NN$ by Definition~\ref{overfull-def}. Then one has that $L(Q')(i)=L(Q)(i)$ for all $i\leq n$. In particular, since $L(Q)\in[\![\sigma]\!]$ and $|\sigma|=n$, we have that $L(Q')\in[\![\sigma]\!]$ as well. This means that $L$ is continuous.
        
        To see that $L$ is an open map, suppose $W_{Q,V}\cap\LF^\NN$ is a basic open set in $\LF^\NN_{str}$. Without loss of generality, we may assume $V=\{1,\dots,n\}$ for some $n\in\NN$. Indeed, we have that
        $$W_{Q,V}\cap\LF^\NN=\bigcup_{Q'\in W_{Q,V}\cap\LF^\NN}W_{Q',\{1,\dots,\max V\}}\cap\LF^\NN,$$
        so showing the openness of $L$ on sets of the desired form is enough. It is then straightforward to check, using a similar argument to the above, that $L(W_{Q,\{1,\dots,n\}})=[\![L(Q)\upharpoonright n]\!]$, which is open in $\NN^\NN$. This completes the proof that $L$ is a homeomorphism and therefore also that $\LF^\NN_{str}\simeq\NN^\NN$.
    \end{proof}

\begin{Rmk}\label{LF=LF^N_str-AF=AF^N_weak-rmk}
    Because of the desirable properties of Polish spaces, we take particularly keen interest in the spaces $\LF^X_{str}$ and $\AF^X_{weak}$ for $X$ countably infinite for the remainder of the paper and simply refer to these spaces as $\LF$ and $\AF$ from this point onward, almost always assuming that $X=\NN$.
    
    The reader may be interested to know what happens when $X$ is finite or uncountable. When $X$ is finite, the space $\AF^X_{weak}$ is discrete (the reader is again encouraged to confirm this if they have not already done so), and therefore every other space of quivers on $X$ mentioned so far is as well; this is rather uninteresting from the point of view of topological complexity. When $X$ is uncountable, the spaces $\LF^X_{weak}$ and $\AF^X_{weak}$ are not separable (and therefore neither are their strong counterparts). This makes working with these spaces more difficult, and there seems to be not much additional return for doing so over working only with countably infinite $X$. Thus, we are comfortable making the running assumption that $X$ is countably infinite.
    
    One may also reasonably inquire about the two spaces $\LF^X_{weak}$ and $\AF^X_{str}$ for countably infinite $X$. It turns out that \emph{neither} of these spaces is a Polish space. We show now that $\AF^X_{str}$ is not separable (and therefore not Polish) and we show that $\LF^X_{weak}$ is not Polish in the following section (see Corollary~\ref{LF^N-weak-not-Polish-cor}).
\end{Rmk}

\begin{Prop}\label{AF^N-str-not-sep-prop}
    Let $X$ be a countably infinite set. Then $\AF^X_{str}$ is not separable and therefore not a Polish space.
\end{Prop}
    \begin{proof}
        Suppose without loss of generality that $X=\NN$. Let $D=\{Q_n\}_{n\geq2}$ be a countable set in $\AF^\NN$ (the motivation for the indexing convention will be made clear shortly). We wish to find some nonempty open set $W$ in $\AF^\NN_{str}$ which fails to intersect $D$. To do so, we ``diagonally'' define a quiver $Q$ on $\NN$ by
        $$Q(i,j):=\begin{cases}
            Q_j(1,j)+1 & \text{if }i=1\text{ and }j\geq 2\\
            Q_i(i,1)-1 & \text{if }i\geq2\text{ and }j=1\\
            0 & \text{otherwise}
        \end{cases}$$
        for all $i,j\in\NN$. We have chosen to index $D=\{Q_n\}_{n\geq2}$ starting at $n=2$ so that the above cases are mutually exclusive and define a skew-symmetric function $Q:\NN\times\NN\to\ZZ$. Now we observe that for any $Q'\in W_{Q,\{1\}}\subseteq\AF^\NN_{str}$, it holds that $Q'(1,n)=Q_n(1,n)+1\neq Q_n(1,n)$ for all $n\geq 2$. In particular, $Q_n\not\in W_{Q,\{1\}}\neq\varnothing$ for all $n\geq2$, showing that $D$ is not dense in $\AF^\NN_{str}$. This completes the proof that $\AF^\NN_{str}$ is not separable and subsequently not a Polish space.
    \end{proof}

%% Section 4: Topological properties of subspaces %%
\section{Properties of quivers as subspaces and Theorem~\ref{hereditary-meta-thm}}\label{sec-4}

Having established several basic results about the spaces $\LF$ and $\AF$, we now consider how the spaces $\LF$ and $\AF$ may be used to study the complexity of various properties of quivers. The perspective we adopt is the following: properties of locally finite (resp. arrow finite) quivers on $\NN$ are in bijection with subspaces of $\LF$ (resp. $\AF$). Thus, we may understand how complicated different properties of quivers are by considering how complicated the corresponding subspaces of $\LF$ or $\AF$ are. In this section, we do just this (in both $\LF$ and $\AF$) for the following quiver properties:
\begin{itemize}
    \item finiteness (Proposition~\ref{finiteness-prop})
    \item connectedness (Proposition~\ref{connectedness-prop})
    \item acyclicity (Proposition~\ref{acyclicity-prop})
    \item mutation-acyclicity (Proposition~\ref{mutation-acyclicity-prop})
    \item existence of a pair of vertices with exactly $m$ arrows between them with $m\in S$ for an arbitrary nonempty subset $S\subseteq\ZZ_{\geq0}$ (Proposition~\ref{weights-prop})
    \item \emph{any} hereditary property (Theorem~\ref{hereditary-meta-thm})
\end{itemize}

\subsection{The Borel hierarchy}
We will need a precise way of measuring the complexity of subspaces. We do so using the \emph{Borel hierarchy,} which provides a tool for comparing the Borel subsets of a topological space. As we will see, many subspaces of $\LF$ and $\AF$ of interest are Borel, so this hierarchy will be sufficient for our purposes in this section.

\begin{Def}[\cite{kechris_classical_1995}, Section 11.B]\label{Borel-hierarchy-def}
    Let $X$ be a topological space. For each nonzero countable ordinal $\alpha$, we define three classes $\mathbf\Sigma^0_\alpha$, $\mathbf\Pi^0_\alpha$, and $\mathbf\Delta^0_\alpha$ of subsets of $X$ recursively as follows:
    \begin{itemize}
        \item The class $\mathbf\Sigma^0_1$ consists of all open subsets of $X$.
        \item The class $\mathbf\Pi^0_\alpha$ consists of the complements of members of the class $\mathbf\Sigma^0_\alpha$.
        \item A set $A\subseteq X$ is in the class $\mathbf\Sigma^0_\alpha$ if there exists a countable sequence $(\alpha_i)_{i\in\NN}$ of nonzero ordinals strictly less than $\alpha$ and a countable sequence $(A_i)_{i\in\NN}$ of subsets of $X$ such that $\bigcup_{i\in\NN}A_i=A$ and for all $i\in\NN$, it holds that $A_i\in\mathbf\Pi^0_{\alpha_i}$.
        \item A set $A\subseteq X$ is in the class $\mathbf\Delta^0_\alpha$ if and only if $A\in\mathbf\Sigma^0_\alpha\cap\mathbf\Pi^0_\alpha$.
    \end{itemize}
    Note that a subset of $X$ is Borel (i.e., belongs to the smallest family of subsets of $X$ containing open and closed sets, closed under complements, and closed under countable unions and intersections) if and only if it belongs to one of these classes. Because of this, one calls the above collection of classes the \emph{Borel hierarchy} (or \emph{boldface hierarchy}) in $X$.

    We also note briefly that $\mathbf\Pi^0_2$ subsets are precisely the $G_\delta$ subsets while $\mathbf\Sigma^0_2$ subsets are precisely the $F_\sigma$ subsets.
\end{Def}

The following proposition collects several important properties of the Borel hierarchy, which we present without proof. For proofs of the following and a general introduction to descriptive set theory, we refer the reader to A. Kechris's book \cite{kechris_classical_1995} and D. Marker's lecture notes \cite{marker_descriptive_2002}.
\begin{Prop}[\cite{marker_descriptive_2002}, Lemmas 2.5 \& 2.6]\label{Borel-hierarchy-shape-prop}
    Let $X$ be a topological space and $\alpha$ a nonzero countable ordinal. The Borel hierarchy in $X$ enjoys the following properties:
    \begin{enumerate}[label=(\roman*)]
        \item If $X$ is metrizable, then $\mathbf\Sigma^0_\alpha\cup\mathbf\Pi^0_\alpha\subseteq\mathbf\Delta^0_{\alpha+1}$.
        \item If $X$ is Polish, then the cardinality of the class of all Borel subsets of $X$ is the continuum $2^{\aleph_0}$. When $X$ is uncountable, this implies that there exist non-Borel subsets of $X$.
        \item $\mathbf\Sigma^0_\alpha$ is closed under countable unions and finite intersections. Dually, $\mathbf\Pi^0_\alpha$ is closed under countable intersections and finite unions. Jointly, $\mathbf\Delta^0_\alpha$ is closed under finite unions and finite intersections.
        \item Let $Y$ be a topological space and $f:X\to Y$ a continuous map. If $A\in\mathbf\Sigma^0_\alpha$ (resp. $A\in\mathbf\Pi^0_\alpha$, $A\in\mathbf\Delta^0_\alpha$) in $Y$, then $f^{-1}(A)\in\mathbf\Sigma^0_\alpha$ (resp. $f^{-1}(A)\in\mathbf\Pi^0_\alpha$, $f^{-1}(A)\in\mathbf\Delta^0_\alpha$) in $X$.
    \end{enumerate}
\end{Prop}

\begin{Rmk}[\cite{marker_descriptive_2002}, Corollary 2.38]\label{Borel-noncollapse-rmk}
    It is a classical result due to Lebesgue that when $X$ is an uncountable Polish space, the Borel hierarchy in $X$ does not ``collapse,'' in that for every nonzero countable ordinal $\alpha$, it holds that $\mathbf\Sigma^0_\alpha\neq\mathbf\Pi^0_\alpha$, and therefore $\mathbf\Sigma^0_\alpha\subsetneq\mathbf\Delta^0_{\alpha+1}$ by Proposition~\ref{Borel-hierarchy-shape-prop}(i). In particular, the Borel hierarchies in $\LF$ and $\AF$ do not collapse. This is in sharp contrast to any countable $T_1$ space, for example, where \emph{every} set is $\mathbf\Sigma^0_2$ (since every set is a countable union of singletons, which are closed in $T_1$ spaces), and therefore also $\mathbf\Pi^0_2$ and $\mathbf\Delta^0_2$.
\end{Rmk}

\begin{Rmk}[\cite{marker_descriptive_2002}, Definitions 4.1 \& 4.27, Theorem 4.13]\label{proj-hierarchy-rmk}
    In general, not every subset of a topological space is Borel. For Polish spaces, one may extend the Borel hierarchy to include more classes $\mathbf\Sigma^1_n$, $\mathbf\Pi^1_n$, and $\mathbf\Delta^1_n$ for all finite $n\geq 1$. The class $\mathbf\Sigma^1_1$ is the class of \emph{analytic} sets, which are defined to be continuous images of Polish spaces. A set is $\mathbf\Pi^1_n$ if its complement is $\mathbf\Sigma^1_n$. A set $A$ in a Polish space $X$ is $\mathbf\Sigma^1_{n+1}$ if there exists some other Polish space $Y$ and a $\mathbf\Pi^1_n$ subset $C\subseteq X\times Y$ such that the projection of $C$ onto the first coordinate is $A$. Lastly, a set is $\mathbf\Delta^1_n$ if it is both $\mathbf\Sigma^1_n$ and $\mathbf\Pi^1_n$. This forms a separate hierarchy called the \emph{projective hierarchy.}
    
    There is a remarkable theorem due to Lusin in 1927 stating that if $A,B\subseteq X$ are two disjoint analytic subsets of a Polish space $X$, there is a Borel set $C$ containing $A$ disjoint from $B$. An immediate consequence of this theorem is Souslin's theorem that every $\mathbf\Delta^1_1$ subset of a Polish space is Borel. The converse is true as well: all Borel subsets of a Polish space are $\mathbf\Delta^1_1$. Thus, one can (and ought to) think about the projective hierarchy as an extension of the Borel hierarchy.
\end{Rmk}

Two particularly useful notions for making comparisons between subsets are those of \emph{hardness} and \emph{completeness}. These notions allow us to provide upper and lower bounds on the complexities of subspaces.

\begin{Def}[\cite{kechris_classical_1995}, Definition 22.9]\label{hard-comp-def}
    Let $\Gamma$ be a class of sets in Polish spaces, i.e., an assignment to each Polish space $X$ a family $\Gamma(X)\subseteq\mathcal P(X)$ of subsets of $X$. If $Y$ is a Polish space, we say that a subset $B\subseteq Y$ is \emph{$\Gamma$-hard} if for any zero-dimensional Polish space $X$ and any $A\in\Gamma(X)$, there exists a continuous map $f:X\to Y$ such that $f^{-1}(B)=A$. If $B\subseteq Y$ is $\Gamma$-hard and is itself a member of the family $\Gamma(Y)$, we say that $B$ is \emph{$\Gamma$-complete}. Here, a \emph{zero-dimensional} Polish space is one with a basis consisting of clopen sets (such as $\NN^\NN$).
\end{Def}

\begin{Rmk}\label{Polish-comp-comp-rmk}
    By Proposition~\ref{Borel-hierarchy-shape-prop}(iv), one can show that a set $B\subseteq Y$ in a Polish space $Y$ is $\mathbf\Sigma^0_\alpha$-hard (resp. $\mathbf\Pi^0_\alpha$-hard) for a countable nonzero ordinal $\alpha$ if and only if there exists a Polish space $X$, a $\mathbf\Sigma^0_\alpha$-complete (resp. $\mathbf\Pi^0_\alpha$-complete) subset $A\subseteq X$, and a continuous map $f:X\to Y$ such that $f^{-1}(B)=A$. Furthermore, a set $A$ in a Polish space $X$ is $\mathbf\Sigma^0_\alpha$ (resp. $\mathbf\Pi^0_\alpha$) in $X$ if and only if there exists a Polish space $Y$, a $\mathbf\Sigma^0_\alpha$-complete (resp. $\mathbf\Pi^0_\alpha$-complete) subset $B\subseteq Y$, and a continuous map $f:X\to Y$ such that $f^{-1}(B)=A$.
\end{Rmk}

\begin{Prop}[\cite{kechris_classical_1995}, Theorem 22.10]\label{Borel-complete-prop}
    Let $X$ be a Polish space and $\alpha>0$ a countable ordinal. Then every member of $\mathbf\Sigma^0_\alpha\setminus\mathbf\Pi^0_\alpha$ is $\mathbf\Sigma^0_\alpha$-complete. Dually, every member of $\mathbf\Pi^0_\alpha\setminus\mathbf\Sigma^0_\alpha$ is $\mathbf\Pi^0_\alpha$-complete.
\end{Prop}

\begin{Rmk}[\cite{kechris_classical_1995}, Section 26.C]\label{proj-noncomplete-rmk}
    The above proposition about the \emph{Borel} hierarchy does not hold for the \emph{projective} hierarchy in ZFC alone; there exist models of ZFC in which there are non-$\mathbf\Sigma^1_1$-complete sets in $\mathbf\Sigma^1_1\setminus\mathbf\Pi^1_1$. That said, one may strengthen ZFC to obtain a set theory in all models of which such sets cannot exist (specifically, the system ZFC+$\mathbf\Sigma^1_1$-Determinacy). We do not dare to venture any closer to such considerations in this paper; we point out this particular aspect of the projective hierarchy only in order to caution the reader that it is not entirely analogous to the Borel hierarchy.
\end{Rmk}

By Remark~\ref{Borel-noncollapse-rmk} and Proposition~\ref{Borel-complete-prop}, one has that in an uncountable Polish space $X$, $\mathbf\Sigma^0_\alpha$-complete and $\mathbf\Pi^0_\alpha$-complete sets exist for every countable ordinal $\alpha>0$. We collect a few standard examples of interest here that will be referred to in the sequel.

\begin{Prop}[\cite{marker_descriptive_2002}, Examples 2.40 \& 2.43, Exercises 2.44 \& 2.45]\label{complete-ex-prop}
    The following completeness results hold:
    \begin{enumerate}[label=(\roman*)]
        \item The set $\{x\in\NN^\NN\mid\exists n:x_n=1\}$ is $\mathbf\Sigma^0_1$-complete in $\NN^\NN$.
        \item The set $\{x\in\ZZ^\NN\mid\exists n\forall m>n:x_m=0\}$ is $\mathbf\Sigma^0_2$-complete in $\ZZ^\NN$.
        \item The set $\{x\in\NN^\NN\mid x:\NN\to\NN\text{ is onto}\}$ is $\mathbf\Pi^0_2$-complete in $\NN^\NN$.
        \item The set $\displaystyle\{x\in\NN^\NN\mid\lim_{n\to\infty}x_n=\infty\}$ is $\mathbf\Pi^0_3$-complete in $\NN^\NN$.
    \end{enumerate}
\end{Prop}

One may also neatly characterize precisely which subspaces of a Polish space are themselves Polish:
\begin{Prop}[\cite{marker_descriptive_2002}, Theorem 1.33]\label{subsp-Gdelta-iff-Polish-prop}
    Let $X$ be a Polish space and $Y\subseteq X$ a subspace. Then $Y$ is Polish (under the subspace topology) if and only if it is a $\mathbf{\Pi}^0_2$ subset of $X$.
\end{Prop}
We now use this proposition to show that $\LF^X_{weak}$ for $X$ countably infinite is not Polish by showing that it is not $\mathbf\Pi^0_2$ in $\AF^X_{weak}$.

\begin{Prop}\label{LF^N-Sigma-0-2-hard-prop}
    Let $X$ be a countable set. Then $\LF^X$ is $\mathbf\Sigma^0_2$-hard in $\AF^X_{weak}$ and thus is not $\mathbf\Pi^0_2$ in $\AF^X_{weak}$.
\end{Prop}
    \begin{proof}
        Assume without loss of generality that $X=\NN$. We make use of the $\mathbf\Sigma^0_2$-complete set $A:=\{x\in\ZZ^\NN\mid\exists n\forall m>n:x_m=0\}\subseteq\ZZ^\NN$ of eventually-zero integer sequences given by Proposition~\ref{complete-ex-prop}(ii). Our goal is to build a continuous map $f:\ZZ^\NN\to\AF^\NN_{weak}$ such that $f^{-1}(\LF^\NN)=A$, which will show that $\LF^\NN$ is $\mathbf\Sigma^0_2$-hard in $\AF^\NN_{weak}$. For an integer sequence $a=(a_n)_{n\geq1}\in\ZZ^\NN$, let $f(a)$ be the quiver given by
        $$f(a)(x,y)=
        \begin{cases}
            a_{y-1} & \text{if }x=1\text{ and }y\geq2\\
            -a_{x-1} & \text{if }x\geq2\text{ and }y=1\\
            0 & \text{otherwise}
        \end{cases}$$
        for all $x,y\in\NN$. We claim that $f^{-1}(\LF^\NN)=A$. If $Q\in\AF^\NN$ is such that there exist $x,y\neq1$ with $Q(x,y)\neq0$, then $Q$ is not in the image of $f$. Thus, we only must consider quivers $Q\in\AF^\NN$ with no such pairs $x,y\neq1$. Any quiver $Q$ of this form is in $\LF^\NN$ if and only if the vertex $1$ has finite degree in $Q$. Equivalently, the sequence of integers $(Q(1,n+1))_{n\geq1}$ is eventually zero (and therefore a point in $A$). One may check that this sequence is in the preimage of $Q$ under $f$ (and is the \emph{only} sequence in the preimage of $Q$ under $f$), so indeed $f^{-1}(\LF^\NN)=A$.
        
        We next show that $f$ is continuous. Indeed, let $U_{Q,V}\subseteq\AF^\NN_{weak}$ be a basic open set for some $Q\in\AF^\NN$ and some finite $V\subseteq\NN$. We may assume without loss of generality that $V$ contains $1$ and at least one other vertex as we may express any basic open set $U_{Q,V'}$ as a union of smaller ones of this desired form. If $V$ contains two vertices $x,y\in\NN$ not equal to $1$ and $Q(x,y)\neq0$, then $f^{-1}(U_{Q,V})=\varnothing$. Otherwise, suppose $V=\{1,x_1,\dots,x_k\}$ and let $b_i=Q(1,x_i)\in\ZZ$ for $i=1,\dots,k$. Then $f^{-1}(U_{Q,V})$ consists of all integer sequences $(a_n)_{n\geq1}\in\ZZ^\NN$ such that $a_{x_i}=b_i$ for all $i=1,\dots,k$. This is precisely the basic open subset of $\ZZ^\NN$ with a factor of $\{b_i\}$ in the $x_i^{th}$ position for $i=1,\dots,k$ and a factor of $\ZZ$ elsewhere, meaning $f^{-1}(U_{Q,V})$ is open.

        Finally, to see that the $\mathbf\Sigma^0_2$-hardness of $\LF^X$ in $\AF^X_{weak}$ implies that it is not $\mathbf\Pi^0_2$ in $\AF^X_{weak}$, suppose for the sake of contradiction that $\LF^X$ were $\mathbf\Pi^0_2$ in $\AF^X_{weak}$. By Proposition~\ref{Borel-hierarchy-shape-prop}(iv), preimages of $\LF^X$ under continuous maps from Polish spaces into $\AF^X_{weak}$ are also $\mathbf\Pi^0_2$. Since $\LF^X$ is $\mathbf\Sigma^0_2$-hard in $\AF^X_{weak}$, we know by Definition~\ref{hard-comp-def} that every $\mathbf\Sigma^0_2$ subset of every zero-dimensional Polish space is such a preimage. But this contradicts Lebesgue's theorem in Remark~\ref{Borel-noncollapse-rmk} that in any uncountable Polish space there exist $\mathbf\Sigma^0_2$ subsets which are not $\mathbf\Pi^0_2$ subsets (and there do exist simultaneously uncountable and zero-dimensional Polish spaces; $\NN^\NN$ is a great example).
    \end{proof}

Combining Propositions \ref{subsp-Gdelta-iff-Polish-prop} and \ref{LF^N-Sigma-0-2-hard-prop} yields the desired corollary:
\begin{Cor}\label{LF^N-weak-not-Polish-cor}
    Let $X$ be a countable set. Then the space $\LF^X_{weak}$ is not a Polish space.
\end{Cor}

Since $\LF^X_{weak}$ is a subspace of the separable metrizable space $\AF^X_{weak}\simeq\NN^\NN$ for $X$ countable, it holds that $\LF^X_{weak}$ is a separable metrizable space. Thus, Corollary~\ref{LF^N-weak-not-Polish-cor} shows that there exists no \emph{complete} metric on $\LF^X_{weak}$ compatible with its topology.

\subsection{Borel complexity of properties of interest}

We first define quiver properties and then provide a comparison lemma.

\begin{Def}\label{property-def}
    A \emph{property of quivers} $\mathcal P$ is a true-false question to which every quiver $Q\in\AF$ has an answer. Moreover, if two quivers $Q,Q'\in\AF$ are isomorphic, we require their answers to the question posed by $\mathcal P$ to be the same. In other words, one can view a property $\mathcal P$ as a subspace $\mathcal P\subseteq\AF$ which is invariant under the action of relabeling vertices.
\end{Def}

\begin{Lem}\label{LF-AF-complexity-comparison-lemma}
    Let $\mathcal P\subseteq\AF$ be a property of arrow finite quivers, and let $\mathcal P_\LF:=\mathcal P\cap\LF\subseteq\LF$ be its restriction to locally finite quivers in $\LF$. If $\mathcal P$ is $\mathbf\Sigma^0_\alpha$ (resp. $\mathbf\Pi^0_\alpha$, $\mathbf\Delta^0_\alpha$) in $\AF$, then the same is true of $\mathcal P_\LF$ in $\LF$. Moreover, if $\mathcal P_\LF$ is dense in $\LF$, then $\mathcal P$ is dense in $\AF$.
\end{Lem}
    \begin{proof}
        Since the inclusion $\iota:\LF\to\AF$ is continuous by Lemma~\ref{weak-str-comp-lemma}, the first part of the claim follows from Proposition~\ref{Borel-hierarchy-shape-prop}(iv). The density part follows from the fact that $\LF^\NN$ is a dense subspace of $\AF$ (cf. Lemma~\ref{LF-dense-in-AF-lemma}).
    \end{proof}

Note that the converse of this lemma does \emph{not} hold: it is in general possible to find a property $\mathcal P\subseteq\AF$ which is \emph{strictly} higher in the Borel hierarchy than $\mathcal P_\LF\subseteq\LF$, and there are properties $\mathcal P\subseteq\AF$ which are dense in $\AF$ while $\mathcal P_\LF$ is not dense in $\LF$. An example simultaneously witnessing both of these claims is $\mathcal P=\AF\setminus\LF^\NN$.

We first look at the property of finiteness. We present the $\LF$ and $\AF$ variants simultaneously, distinguishing between cases as needed in the proof.

\begin{Prop}\label{finiteness-prop}
    Let $\mathbf{Fin}\subseteq\LF$ denote the subspace of $\LF$ (resp. of $\AF$) of finite quivers (see Remark~\ref{fin-quiv-rmk}). Then it holds that $\mathbf{Fin}$ is dense, has empty interior, and is $\mathbf\Sigma^0_2$-complete in $\LF$ (resp. $\AF$).
\end{Prop}
    \begin{proof}
        By Lemma~\ref{LF-AF-complexity-comparison-lemma}, it suffices to show that $\mathbf{Fin}$ and $\LF\setminus\mathbf{Fin}$ are both dense in $\LF$, that $\mathbf{Fin}$ is $\mathbf\Sigma^0_2$ in $\AF$, and that $\mathbf{Fin}$ is $\mathbf\Sigma^0_2$-hard in $\LF$.

        To see that $\mathbf{Fin}$ and its complement are both dense in $\LF$ is relatively straightforward. Let $W_{Q,V}$ be a basic open set in $\LF$ for some $Q\in\LF$ and some finite $V\subseteq\NN$. Since $Q$ is locally finite, its overfill induced by $V$ is finite, and moreover this finite overfill is itself a member of $W_{Q,V}$. Thus, $\mathbf{Fin}\cap W_{Q,V}\neq\varnothing$. On the other hand, suppose now without loss of generality that $Q$ is equal to its finite overfill induced by $V$ (so that $\text{supp}(Q)\subseteq V$). Let $m=\max V\in\NN$. Extend $Q$ to a quiver $Q'$ by adding infinitely many arrows $m+1\to m+2\to m+3\to\dots$. Clearly this quiver is both locally finite and infinite, so $\left(\LF\setminus\mathbf{Fin}\right)\cap W_{Q,V}\neq\varnothing$. Hence, we have shown that both $\mathbf{Fin}$ and its complement are dense in $\LF$.

        Next, we see rather quickly that $\mathbf{Fin}$ is $\mathbf\Sigma^0_2$ in $\AF$. In fact, this follows immediately from the observation that $\mathbf{Fin}$ is a countable set (and therefore a countable union of singletons, which are closed).

        Lastly, we show that $\mathbf{Fin}$ is $\mathbf\Sigma^0_2$-hard in $\LF$. To do so, it is equivalent to show by Proposition~\ref{Borel-complete-prop} that $\mathbf{Fin}$ is not $\mathbf\Pi^0_2$ in $\LF$. To see that $\mathbf{Fin}$ is not $\mathbf\Pi^0_2$ in $\LF$, we argue toward a contradiction; suppose it were. Then one would have by the Baire Category Theorem (which applies in $\LF\simeq\NN^\NN$) that $\varnothing=\mathbf{Fin}\cap\left(\LF\setminus\mathbf{Fin}\right)$, being an intersection of two dense $\mathbf\Pi^0_2$ sets, is itself $\mathbf\Pi^0_2$ and dense. While the former half is true, the latter half is absurd: the empty set is not dense in $\LF$. Thus, we see that $\mathbf{Fin}$ is not $\mathbf\Pi^0_2$ in $\LF$ and is therefore $\mathbf\Sigma^0_2$-hard.
    \end{proof}

We next turn our attention to the property of connectedness, which we define as follows:

\begin{Def}\label{con-N-quiv-def}
    A quiver $Q:X\times X\to\ZZ$ with vertex set $X$ is \emph{connected} if its underlying undirected graph (obtained by forgetting orientations of the arrows of $Q$) has at most one connected component which is not an isolated vertex. Equivalently, all arrows of $Q$ belong to the same connected component of $Q$.
\end{Def}

\begin{Prop}\label{connectedness-prop}
    Let $\mathbf{Con}_\LF\subseteq\LF$ (resp. $\mathbf{Con}_\AF\subseteq\AF$) be the subspace of $\LF$ (resp. $\AF$) of connected quivers. Then it holds that $\mathbf{Con}_\LF$ is not dense, has empty interior, and is $\mathbf\Pi^0_2$ in $\LF$. On the other hand, $\mathbf{Con}_\AF$ is dense, has empty interior, and is $\mathbf\Pi^0_2$ in $\AF$.
\end{Prop}
    \begin{proof}
        By Lemma~\ref{LF-AF-complexity-comparison-lemma}, it suffices to show that that $\mathbf{Con}_\LF$ has empty interior and is not dense in $\LF$, and that $\mathbf{Con}_\AF$ is dense and $\mathbf\Pi^0_2$ in $\AF$.
    
        We begin with the three density and empty interior claims. First, consider $\mathbf{Con}_\LF\subseteq\LF$. To establish its non-density, we may simply provide an open subset of $\LF$ containing no connected quiver. Indeed, if $Q$ is a disconnected finite quiver, we let $V=\text{supp}(Q)\subseteq\NN$ be its support. Then all quivers $Q'\in W_{Q,V}$ contain $Q$ as an overfull subquiver. Suppose for the sake of contradiction that there exists a connected $Q'\in W_{Q,V}$. Then there exists a finite sequence of vertices $v_0,v_1,\dots,v_k\in\NN$ for any two $v_0,v_k\in V$ such that $Q'(v_i,v_{i+1})\neq0$ for all $i=0,\dots,k-1$. Since $Q'$ contains $Q$ as an \emph{overfull} subquiver and $v_0,v_k\in V$, one has that $Q(v_0,v_1)=Q'(v_0,v_1)\neq 0$ and $Q(v_{k-1},v_k)=Q'(v_{k-1},v_k)\neq0$. Therefore, we have that $v_1,v_{k-1}\in\text{supp}(Q)=V$. One may then inductively argue that $v_i\in V$ for all $i=0\dots, k$. However, this contradicts the disconnectedness of $Q$, so one has that $\mathbf{Con}_\LF\cap W_{Q,V}=\varnothing$. Thus, $\mathbf{Con}_\LF$ is not dense in $\LF$.

        We now give an argument that $\mathbf{Con}_\LF$ has empty interior. We show the equivalent statement that $\LF\setminus\mathbf{Con}_\LF$ is dense in $\LF$. Let $W_{Q,V}$ be a basic open set in $\LF$ for some $Q\in\LF$ and $V\subseteq\NN$ finite. Since $Q$ is locally finite, the overfull subquiver of $Q$ induced by $V$ is finite as well. Thus, we may assume without loss of generality that $Q$ is finite and equal to its overfull subquiver induced by $V$. Enlarging $V$ if necessary, we may assume that $\text{supp}(Q)\subseteq V$. Let $m=\max V\in\NN$. Extend $Q$ to a quiver $Q'\in W_{Q,V}$ by placing a single arrow from $m+1$ to $m+2$. Since $Q'$ is disconnected and $Q'\in W_{Q,V}$, we have shown that $W_{Q,V}\cap\left(\LF\setminus\mathbf{Con}_\LF\right)\neq\varnothing$, as desired.

        The last density claim to show is that $\mathbf{Con}_\AF$ is dense in $\AF$. Let $Q\in\AF$ be an arrow finite quiver and let $V\subseteq\NN$ be finite. We claim that $U_{Q,V}$ contains a connected quiver. Indeed, we may assume without loss of generality again that $Q$ is finite and $\text{supp}(Q)\subseteq V$. Let $m=\max V\in\NN$. Extend $Q$ to a quiver $Q'$ by adding an arrow $i\to m+1$ for all $i\leq m$. Then $Q'$ is connected and $Q'\in U_{Q,V}$, so $\mathbf{Con}_\AF$ is dense in $\AF$.

        Finally, it remains to show that $\mathbf{Con}_\AF$ is $\mathbf\Pi^0_2$ in $\AF$. We first observe that
        $$\mathbf{Con}_\AF=\bigcap_{i,j\in\NN}\left(P_{i,j}\cup I_i\cup I_j\right),$$
        where
        \begin{align*}
        P_{i,j}&:=\left\{Q\in\AF\mid\exists\ell\geq 0,\exists i_1,\dots,i_\ell\in\NN:Q(i,i_1),Q(i_1,i_2),\dots,Q(i_\ell,j)\neq0\right\}\subseteq\AF\\
        I_i&:=\left\{Q\in\AF\mid i\text{ is isolated in }Q\right\}\subseteq\AF.
        \end{align*}
        Indeed, our definition of connectedness only requires that every \emph{non-isolated} vertex is connected to every other \emph{non-isolated} vertex; for any pair $i,j\in\NN$, then, it may hold that either $i$ or $j$ is isolated in any given connected quiver. Thus, if we could show that $P_{i,j}$ and $I_i$ were open in $\AF$ for all $i,j\in\NN$, we would be done. However, while we will in fact be able to show that $P_{i,j}$ is open, we will not be able to show that $I_i$ is open. Not all hope is lost, though, since we will be able to show that $I_i$ is closed. Then it follows that $P_{i,j}\cup I_i\cup I_j$ is $\mathbf\Delta^0_2$ for all $i,j\in\NN$, and countable intersections of $\mathbf\Delta^0_2$ sets are $\mathbf\Pi^0_2$ by Proposition~\ref{Borel-hierarchy-shape-prop}(iii). This will be enough to show that $\mathbf{Con}_\AF$ is $\mathbf\Pi^0_2$ in $\AF$.

        To see that $P_{i,j}$ is open in $\AF$, we may write $P_{i,j}$ as the following countable union:
        \begin{align*}
        P_{i,j}&=\bigcup_{(i_1,\dots,i_\ell)\in\NN^{<\NN}}\left\{Q\in\AF\mid Q(i,i_1),Q(i_1,i_2),\dots,Q(i_\ell,j)\neq0\right\}\\
        &=\bigcup_{(i_1,\dots,i_\ell)\in\NN^{<\NN}}\left\{Q\in\AF\mid Q(i,i_1)\neq0\right\}\cap\dots\cap\left\{Q\in\AF\mid Q(i_\ell,j)\neq0\right\}.
        \end{align*}
        Note that for all $v,w\in\NN$, the set $\{Q\in\AF\mid Q(v,w)\neq0\}=\AF\setminus\{Q\in\AF\mid Q(v,w)=0\}$ is the complement of a basic clopen set $U_{Q,V}$ (with $V=\{v,w\}$ and $Q$ the quiver with no arrows), so is itself clopen. Thus, the above finite intersections in the above union are clopen, so the entire union $P_{i,j}$ is open.

        To see that $I_i$ is closed in $\AF$, we may simply observe that it can be written as the intersection
        $$I_i=\bigcap_{k\neq i}\left\{Q\in\AF\mid Q(i,k)=0\right\}$$
        of basic clopen sets of the form described in the previous paragraph.

        Altogether, we have that $\mathbf{Con}_\AF$ is a countable intersection of $\mathbf\Delta^0_2$ sets in $\AF$, so $\mathbf{Con}_\AF$ is $\mathbf\Pi^0_2$ in $\AF$.
        \end{proof}

We now examine acyclicity.

\begin{Def}\label{acyc-def}
    A quiver $Q:X\times X\to\ZZ$ on a vertex set $X$ is \emph{acyclic} if it contains no oriented cycles, i.e., finite sequences of vertices $(v_1,v_2,\dots,v_\ell)\in X^\ell$ such that $Q(v_i,v_{i+1})>0$ for all $i=1,\dots,\ell$ (with indices taken modulo $\ell$).
\end{Def}

\begin{Prop}\label{acyclicity-prop}
    Let $\mathbf{Acyc}_\LF\subseteq\LF$ (resp. $\mathbf{Acyc}_\AF\subseteq\AF$) be the subspace of $\LF$ (resp. $\AF$) of acyclic quivers. Then it holds that $\mathbf{Acyc}_\LF$ (resp. $\mathbf{Acyc}_\AF$) is not dense, has empty interior, and is closed (i.e. $\mathbf\Pi^0_1$) in $\LF$ (resp. $\AF$).
\end{Prop}
    \begin{proof}
        With Lemma~\ref{LF-AF-complexity-comparison-lemma} in mind, it suffices to show that $\mathbf{Acyc}_\AF\subseteq\AF$ is closed, that $\AF\setminus\mathbf{Acyc}_\AF$ has nonempty interior in $\AF$, and that $\mathbf{Acyc}_\LF\subseteq\LF$ has empty interior. Further, we note that if we can show $\mathbf{Acyc}_\AF\subseteq\AF$ is closed, then it automatically follows that $\AF\setminus\mathbf{Acyc}_\AF$ has nonempty interior, as it is a nonempty open set (since there exist non-acyclic quivers). Thus, we are only required to show that $\mathbf{Acyc}_\AF$ is closed in $\AF$ and that $\mathbf{Acyc}_\LF$ has empty interior in $\LF$.

        We show that $\AF\setminus\mathbf{Acyc}_\AF$ is open in $\AF$. Let $Q\in\AF\setminus\mathbf{Acyc}_\AF$. By Definition~\ref{acyc-def}, it holds that there exist vertices $v_1,\dots,v_\ell\in\NN$ such that $Q(v_i,v_{i+1})>0$ for all $i=1,\dots,\ell$ (modulo $\ell$). Then it follows that every $Q'\in U_{Q,\{v_1,\dots,v_\ell\}}$ has $Q'(v_i,v_{i+1})=Q(v_i,v_{i+1})>0$ for all $i=1,\dots,\ell$ (modulo $\ell$). This shows that $U_{Q,\{v_1,\dots,v_\ell\}}\subseteq\AF\setminus\mathbf{Acyc}_\AF$ is an open neighborhood of $Q$ in $\AF\setminus\mathbf{Acyc}_\AF$, so $\AF\setminus\mathbf{Acyc}_\AF$ is open.

        Now we show that $\mathbf{Acyc}_\LF$ has empty interior in $\LF$. Indeed, suppose $Q\in\mathbf{Acyc}_\LF$, and let $W_{Q,V}$ be a basic open neighborhood of $Q$ in $\LF$ for some finite $V\subseteq\NN$. Since $Q$ is locally finite and $V$ is finite, we may assume without loss of generality that $Q$ is finite. Extend $V$ if needed to obtain $\text{supp}(Q)\subseteq V$. Let $N=\max V$. Extend $Q$ to another locally finite quiver $Q'$ by creating a $3$-cycle $N+1\to N+2\to N+3\to N+1$. Then $Q'\in W_{Q,V}\cap\left(\AF\setminus\mathbf{Acyc}_\LF\right)$. Hence, we see that $W_{Q,V}\not\subseteq\mathbf{Acyc}_\LF$, so $\mathbf{Acyc}_\LF$ has empty interior in $\LF$.
    \end{proof}

The next property is the first examined in our list to involve mutations.

\begin{Def}\label{mut-acyc-def}
    A quiver $Q$ is \emph{mutation-acyclic} if there exists a \emph{finite} sequence of mutations one may apply to $Q$ to obtain an acyclic quiver.
\end{Def}

\begin{Rmk}\label{mut-acyc-hered-rmk}
    We quickly note that mutation-acyclicity is a \emph{hereditary} property for finite quivers, meaning that if $Q$ is a finite mutation-acyclic quiver and $Q'$ is a full subquiver of $Q$, then $Q'$ is mutation-acyclic. This is a classical result in the theory of cluster algebras, but no combinatorial proof of this fact is currently known.
\end{Rmk}

\begin{Prop}\label{mutation-acyclicity-prop}
    Let $\mathbf{MA}_\LF\subseteq\LF$ (resp. $\mathbf{MA}_\AF\subseteq\AF$) be the subspace of $\LF$ (resp. $\AF$) consisting of mutation-acyclic quivers. Then $\mathbf{MA}_\LF$ (resp. $\mathbf{MA}_\AF$) is not dense, has empty interior, and is $\mathbf\Sigma^0_2$ in $\LF$ (resp. $\AF$).
\end{Prop}
    \begin{proof}
        Using Lemma~\ref{LF-AF-complexity-comparison-lemma}, we may reduce the claim to showing that $\AF\setminus\mathbf{MA}_\AF$ has nonempty interior in $\AF$, that $\mathbf{MA}_\LF$ has empty interior in $\LF$, and that $\mathbf{MA}_\AF$ is $\mathbf\Sigma^0_2$ in $\AF$.

        To show that $\AF\setminus\mathbf{MA}_\AF$ has nonempty interior in $\AF$, we first argue that it suffices to provide a single example of a finite quiver which is not mutation-acyclic. Indeed, for such a $Q$, we let $V=\text{supp}(Q)$ and argue that no quiver in $U_{Q,V}$ is mutation-acyclic. Suppose for the sake of contradiction that there is some $Q'\in U_{Q,V}$ and some finite sequence $\mathbf i=(i_1,i_2,\dots,i_\ell)$ such that $Q'':=\mu_\mathbf i(Q')=\mu_{i_\ell}\circ\dots\circ\mu_{i_1}(Q')$ is acyclic. Now extend $V$ to $V':=V\cup\{i_1,\dots,i_\ell\}$. Since full subquivers of acyclic quivers are acyclic, it holds that $Q''$ restricted to $V'$ is acyclic. Moreover, since all mutations in $\mathbf i$ occurred within $V'$, we have (by the fact that restriction to $V'$ and mutation within $V'$ commute) that $Q'$ restricted to $V'$ is mutation-acyclic. Moreover, $Q'$ restricted to $V'$ is a finite mutation-acyclic quiver with $Q\leq Q'$. From this and the property that full subquivers of finite mutation-acyclic quivers are mutation-acyclic (see Remark~\ref{mut-acyc-hered-rmk}), we see that $Q$ is mutation-acyclic, contradicting our original assumption that it isn't. To finish this argument, we need to provide an example of a finite quiver which is not mutation-acyclic. The following famous example of a finite non-mutation-acyclic quiver, the \emph{Markov quiver}, has this property:
        \\
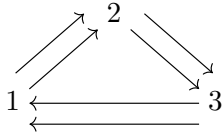
        \[\begin{tikzcd}[cramped]
        	& 2 \\
        	1 && 3
        	\arrow[from=1-2, to=2-3]
        	\arrow[shift left=3, from=1-2, to=2-3]
        	\arrow[from=2-1, to=1-2]
        	\arrow[shift left=3, from=2-1, to=1-2]
        	\arrow[from=2-3, to=2-1]
        	\arrow[shift left=3, from=2-3, to=2-1]
        \end{tikzcd}\]
        \\

        We next show that $\mathbf{MA}_\LF$ has empty interior in $\LF$. We argue this by way of the Baire Category Theorem, using the fact that \emph{the} Baire space $\LF\simeq\NN^\NN$ is also \emph{a} Baire space, i.e., one in which the Baire Category Theorem applies. We first observe that
        $$\mathbf{MA}_\LF=\bigcup_{\mathbf i=(i_1,\dots,i_\ell)\in\NN^{<\NN}}\mu_\mathbf i(\mathbf{Acyc}_\LF).$$
        Since $\mu_\mathbf i:\LF\to\LF$ is a homeomorphism for all $\mathbf i\in\NN^{<\NN}$, it follows by Proposition~\ref{acyclicity-prop} above that all of the sets $\mu_\mathbf i(\mathbf{Acyc}_\LF)$ are closed with empty interior. By the Baire Category Theorem applied to $\LF$, it then follows that $\mathbf{MA}_\LF$ has empty interior in $\LF$.

        Finally, we argue that $\mathbf{MA}_\AF$ is $\mathbf\Sigma^0_2$ in $\AF$. Indeed, we may write
        $$\mathbf{MA}_\AF=\bigcup_{\mathbf i=(i_1,\dots,i_\ell)\in\NN^{<\NN}}\mu_\mathbf i(\mathbf{Acyc}_\AF)$$
        nearly identically to the above. Since $\mu_\mathbf i:\AF\to\AF$ is a homeomorphism for all $\mathbf i\in\NN^{<\NN}$, $\mathbf{Acyc}_\AF$ is closed in $\AF$, and there are only countably many $\mathbf i\in\NN^{<\NN}$, it follows that $\mathbf{MA}_\AF$ is $\mathbf\Sigma^0_2$ in $\AF$.
    \end{proof}

We now look at properties concerning the multiplicities of arrows in a quiver.

\begin{Def}\label{weights-def}
    Let $Q:X\times X\to\ZZ$ be a quiver on a vertex set $X$ and $m\in\ZZ_{\geq0}$. We say that $Q$ \emph{has a weight equal to $m$} if there exists a pair of vertices $i,j\in X$ such that $Q(i,j)=m$.
\end{Def}

\begin{Prop}\label{weights-prop}
    Fix $S\subseteq\ZZ_{\geq0}$ nonempty. Let $\mathbf \Omega_{S,\LF}\subseteq\LF$ (resp. $\mathbf \Omega_{S,\AF}\subseteq\AF$) be the subspace of $\LF$ (resp. $\AF$) containing quivers with a weight equal to some $m\in S$. Then $\mathbf \Omega_{S,\LF}$ (resp. $\mathbf \Omega_{S,\AF}$) is dense and open (i.e. $\mathbf\Sigma^0_1$) in $\LF$ (resp. $\AF$).
\end{Prop}
    \begin{proof}
        Let $S\subseteq\ZZ_{\geq0}$. It suffices by Lemma~\ref{LF-AF-complexity-comparison-lemma} to show that $\mathbf\Omega_{S,\AF}$ is open in $\AF$ and that $\mathbf\Omega_{S,\LF}$ is dense in $\LF$.

        We first show that $\mathbf\Omega_{S,\AF}$ is open in $\AF$. Let $Q\in\mathbf\Omega_{S,\AF}$. By definition, this means that there exist $i,j\in S$ so that $Q(i,j)\in S$. Then it follows that any $Q'\in U_{Q,\{i,j\}}$ has $Q'(i,j)=Q(i,j)\in S$ as well, so $U_{Q,\{i,j\}}\subseteq\mathbf\Omega_{S,\AF}$ is an open neighborhood of $Q$ in $\mathbf\Omega_{S,\AF}$.

        To see that $\mathbf\Omega_{S,\LF}$ is dense in $\LF$, let $Q\in\LF$ and let $V\subseteq\NN$ be finite. We want to show that $W_{Q,V}$ contains some $Q'\in\mathbf\Omega_{S,\LF}$. Since $Q$ is locally finite, its overfill induced by $V$ is finite as well. Hence, we may assume without loss of generality that $Q$ is equal to its overfill induced by $V$. Extending $V$ if needed, we may also assume $\text{supp}(Q)\subseteq V$. Let $N=\max V$. Since $S\neq\varnothing$, we may pick some $m\in S$. Extend $Q$ to a new (locally) finite quiver $Q'$ by declaring $Q'(N+1,N+2)=m$. It then holds that the overfill of $Q'$ on $V$ agrees with that of $Q$, so $Q'\in W_{Q,V}$. It also holds that $Q'$ has a weight in $S$, so $Q'\in\mathbf\Omega_{S,\LF}$ as well. This shows that $\mathbf\Omega_{S,\LF}$ is dense in $\LF$.
    \end{proof}

As a quick corollary, we have that ``tameness'' is a closed property with empty interior:

\begin{Cor}\label{tameness-cor}
    Let $\mathbf T_\LF\subseteq\LF$ (resp. $\mathbf T_\AF\subseteq\AF$) be the set of all quivers in $\LF$ (resp. $\AF$) with no weights greater than $2$ and which are not mutation-equivalent to any quiver with weights greater than $2$. Then $\mathbf T_\LF$ (resp. $\mathbf T_\AF$) is closed (i.e. $\mathbf\Pi^0_1$) with empty interior in $\LF$ (resp. $\AF$).
\end{Cor}
    \begin{proof}
        It is straightforward to confirm the following:
        $$\mathbf T_\LF=\bigcap_{\mathbf i\in\NN^{<\NN}}\mu_{\mathbf i}\left(\LF\setminus\mathbf\Omega_{S,\LF}\right)$$
        and similarly
        $$\mathbf T_\AF=\bigcap_{\mathbf i\in\NN^{<\NN}}\mu_{\mathbf i}\left(\AF\setminus\mathbf\Omega_{S,\AF}\right)$$
        for $S=\{3,4,5,6,\dots\}$. By Proposition~\ref{weights-prop} above, $\LF\setminus\mathbf\Omega_{S,\LF}$ and $\AF\setminus\mathbf\Omega_{S,\AF}$ are both closed with empty interior in $\LF$ and $\AF$ (respectively). Using the fact that $\mu_\mathbf i$ yields homeomorphisms $\LF\to\LF$ and $\AF\to\AF$ for all finite sequences $\mathbf i\in\NN^{<\NN}$, one sees that the sets $\mu_\mathbf i(\LF\setminus\mathbf\Omega_{S,\LF})$ and $\mu_\mathbf i(\AF\setminus\mathbf\Omega_{S,\AF})$ are also closed with empty interior in $\LF$ and $\AF$ (respectively). It immediately follows that the intersections $\mathbf T_\LF$ and $\mathbf T_\AF$ are closed sets with empty interior in $\LF$ and $\AF$ (respectively).
    \end{proof}

\subsection{Hereditary properties}
We close this section with a theorem partially describing the topological attributes of arbitrary hereditary properties of quivers.

\begin{Def}\label{hereditary-def}
    A property $\mathcal P$ of quivers (see Definition~\ref{property-def}) is said to be \emph{hereditary} if for any $Q$ with the property $\mathcal P$, all \emph{finite} full subquivers of $Q$ also have the property $\mathcal P$.
\end{Def}

\begin{Def}\label{F-avoiding-def}
    Let $\mathcal F$ be a (possibly infinite) family of finite unlabeled quivers. The set $\mathcal P_\mathcal F\subseteq\AF$ of \emph{quivers avoiding $\mathcal F$} is defined to be the set of all $Q\in\AF$ not containing any quiver in $\mathcal F$ as a full subquiver on any subset of vertices.
\end{Def}

\begin{T-hered}
    Let $\mathcal P\subseteq\AF$ be a hereditary property of arrow finite quivers such that there exists at least one finite quiver \emph{not} carrying the property $\mathcal P$. Then it holds that $\mathcal P_\LF$ is not dense in $\LF$, nor is $\mathcal P$ is dense in $\AF$. On the other hand, $\LF\setminus\mathcal P_\LF$ is dense in $\LF$, and similarly $\AF\setminus\mathcal P$ is dense in $\AF$. Moreover, $\mathcal P=\mathcal P_\mathcal F$ is the property of avoiding $\mathcal F$ for some (possibly infinite) family $\mathcal F$ of finite unlabeled quivers if and only if $\mathcal P$ is closed in $\AF$.
\end{T-hered}
    \begin{proof}
        By Lemma~\ref{LF-AF-complexity-comparison-lemma}, it suffices for the first part of the claim to show that $\mathcal P_\LF=\mathcal P\cap\LF$ is not dense in $\LF$ and that $\AF\setminus\mathcal P$ is dense in $\AF$.

        Start by taking some $Q\in\LF\setminus\mathcal P$; note that this is only possible to do since there is some (locally) finite quiver not carrying the property $\mathcal P$, and let $V\subseteq\NN$ be finite. As in other arguments given in this section, we may assume without loss of generality since $Q$ is locally finite that $Q$ is finite and $\text{supp}(Q)\subseteq V$. Then $Q$ is a finite quiver not carrying the property $\mathcal P$. Furthermore, any $Q'\in W_{Q,V}\cap\LF$ has $Q$ as an overfull subquiver within the vertices $V$, and therefore also as a full subquiver within the vertices $V$. If such a $Q'$ were to carry the property $\mathcal P$, then so would its finite full subquiver $Q$ by the hereditarity of $\mathcal P$, a contradiction. Thus, we see that $W_{Q,V}\cap\mathcal P=\varnothing$, meaning that $\mathcal P\cap\LF$ is not dense in $\LF$.

        We next show that $\AF\setminus\mathcal P$ is dense in $\AF$. Let $Q\in\AF$ and take $V\subseteq\NN$ finite. We assume once again without loss of generality that $Q$ is finite and that $\text{supp}(Q)\subseteq V$. Suppose that $T$ is a finite quiver not carrying the property $\mathcal P$. Let $N=\max V$. Place a copy of $T$ in $Q$ using only vertices of label $N+1$ and higher to create a new quiver $Q'\in U_{Q,V}$. Then $Q'\not\in\mathcal P$, since otherwise $T$ would be in $\mathcal P$, a contradiction. Thus, we see that $U_{Q,V}\cap\left(\AF\setminus\mathcal P\right)\neq\varnothing$, so the complement of $\mathcal P$ is dense in $\AF$.

        We now proceed to the second claim in the theorem. Assume first that $\mathcal P=\mathcal P_\mathcal F$ for some family $\mathcal F$ of finite unlabeled quivers. We show that the complement of $\mathcal P$ is open in $\AF$. Suppose $Q\in\AF\setminus\mathcal P$. Since $\mathcal P=\mathcal P_\mathcal F$, there exists some finite subset $V\subseteq\NN$ of vertices such that the restriction of $Q$ to $V$ is a (finite) quiver $T$ belonging to $\mathcal F$. Then any $Q'\in U_{Q,V}$ also contains $T$ as a full subquiver, and therefore we have $Q'\in\AF\setminus\mathcal P$ as well. Thus, we have that $\mathcal P$ is closed in $\AF$.
        
        On the other hand, suppose that $\mathcal P$ is closed in $\AF$. We claim that the family
        $$\mathcal F:=\{T\mid T\text{ is not isomorphic to a finite full subquiver of any }Q\in\mathcal P\}$$
        of finite unlabeled quivers is such that $\mathcal P=\mathcal P_\mathcal F$. The forward containment $\mathcal P\subseteq\mathcal P_\mathcal F$ follows by definition: if $Q\in\mathcal P$, then no finite full subquiver of $Q$ belongs to $\mathcal F$, so $Q\in\mathcal P_\mathcal F$. For the backward containment, let $Q\in\mathcal P_\mathcal F$. Unpacking our definition of $\mathcal P_\mathcal F$, this means precisely that every finite full subquiver $T$ of $Q$ appears as a finite full subquiver of some $Q_T\in\mathcal P$. Even more, this means by the hereditarity of $\mathcal P$ that each finite full subquiver $T$ of $Q$ has the property $\mathcal P$. In particular, the finite full subquivers $Q_n$ of $Q$ induced by the vertex sets $\{1,\dots,n\}$ for $n\geq1$ each belong to $\mathcal P$. Since $Q_n\to Q$ in $\AF$ and $\mathcal P$ is closed in $\AF$, one obtains that $Q\in\mathcal P$ as desired.
    \end{proof}

%% Section 5: Infinite mutation sequences %%
\section{Infinite mutation sequences and Theorems \ref{AF-inf-mut-seq-thm} and \ref{LF-inf-mut-seq-thm}}\label{sec-5}

In this section, we focus our attention on infinite sequences of mutations. Since the mutation maps $\mu_x$ are homeomorphisms on both $\LF$ and $\AF$ (see Proposition~\ref{mut-cont-prop}), one has continuous group actions on $\LF$ and $\AF$ by the countable discrete group $G:=\langle g_i\mid g_i^2=\varepsilon,i\in\NN\rangle$ freely generated by countably many generators of order $2$. Of course, this group only consists of \emph{finite} sequences of mutations. However, one might think to ask which continuous maps $\LF\to\LF$ or $\AF\to\AF$ can be recovered as \emph{infinite} mutation sequences. It turns out that there are \emph{no} such maps in the $\AF$ case and that, in the $\LF$ case, only the \emph{finite} mutation sequences can be obtained this way.

We first need to make sense of infinite sequences of mutations.

\begin{Def}\label{inf-mut-seq-def}
    Let $\mathbf i=(i_1,i_2,\dots)$ be an infinite sequence of natural numbers. For any quiver $Q$ in $\LF$ (resp. $\AF$), we define $\mu_\mathbf i(Q)$ to be the limit of the sequence $(\mu_{i_m}\circ\dots\circ\mu_{i_1}(Q))_{m\geq1}$ in $\LF$ (resp. $\AF$) if it exists, and we leave $\mu_\mathbf i(Q)$ undefined otherwise. In the former case, we say that the mutation sequence $\mu_\mathbf i$ \emph{converges on $Q$}, and in the latter we say $\mu_\mathbf i$ \emph{diverges on $Q$.}

    For an infinite mutation sequence $\mu_\mathbf i$, we let $\mathcal C_\LF(\mu_\mathbf i)$ (resp. $\mathcal C_\AF(\mu_\mathbf i)$) denote the set of all $Q\in\LF$ (resp. $Q\in\AF$) on which $\mu_\mathbf i$ converges. We similarly define $\mathcal D_\LF(\mu_\mathbf i)$ (resp. $\mathcal D_\AF(\mu_\mathbf i)$) to be the set of all $Q\in\LF$ (resp. $Q\in\AF$) on which $\mu_\mathbf i$ diverges.
\end{Def}

\begin{Ex}\label{inf-mut-seq-ex1}
    Consider the infinite mutation sequence $\mu_{(1,2,3,\dots)}=\dots\mu_3\circ\mu_2\circ\mu_1$ applied to the type $\mathbb A_\infty$ quiver $Q\in\LF$ pictured below:
    % https://q.uiver.app/#q=WzAsNixbMSwwLCIxIl0sWzIsMCwiMiJdLFszLDAsIjMiXSxbNCwwLCI0Il0sWzUsMCwiXFxkb3RzIl0sWzAsMCwiUToiXSxbMCwxXSxbMSwyXSxbMiwzXSxbMyw0XV0=
    \[\begin{tikzcd}[cramped]
    	{Q:} & 1 & 2 & 3 & 4 & \dots
    	\arrow[from=1-2, to=1-3]
    	\arrow[from=1-3, to=1-4]
    	\arrow[from=1-4, to=1-5]
    	\arrow[from=1-5, to=1-6]
    \end{tikzcd}\]
    
    The first three quivers in the sequence $(\mu_n\circ\dots\circ\mu_1(Q))_{n\geq1}$ are as follows:
    % https://q.uiver.app/#q=WzAsMTgsWzEsMCwiMSJdLFsyLDAsIjIiXSxbMywwLCIzIl0sWzQsMCwiNCJdLFs1LDAsIlxcZG90cyJdLFswLDAsIlxcbXVfMShRKToiXSxbMCwxLCJcXG11XzIoXFxtdV8xKFEpKToiXSxbMSwxLCIxIl0sWzIsMSwiMiJdLFszLDEsIjMiXSxbNCwxLCI0Il0sWzUsMSwiXFxkb3RzIl0sWzAsMiwiXFxtdV8zKFxcbXVfMihcXG11XzEoUSkpKToiXSxbMSwyLCIxIl0sWzIsMiwiMiJdLFszLDIsIjMiXSxbNCwyLCI0Il0sWzUsMiwiXFxkb3RzIl0sWzEsMl0sWzIsM10sWzMsNF0sWzEsMF0sWzcsOF0sWzksOF0sWzksMTBdLFsxMCwxMV0sWzE2LDE3XSxbMTYsMTVdLFsxNCwxNV0sWzEzLDE0XV0=
    \[\begin{tikzcd}[cramped]
    	{\mu_1(Q):} & 1 & 2 & 3 & 4 & \dots \\
    	{\mu_2(\mu_1(Q)):} & 1 & 2 & 3 & 4 & \dots \\
    	{\mu_3(\mu_2(\mu_1(Q))):} & 1 & 2 & 3 & 4 & \dots
    	\arrow[from=1-3, to=1-2]
    	\arrow[from=1-3, to=1-4]
    	\arrow[from=1-4, to=1-5]
    	\arrow[from=1-5, to=1-6]
    	\arrow[from=2-2, to=2-3]
    	\arrow[from=2-4, to=2-3]
    	\arrow[from=2-4, to=2-5]
    	\arrow[from=2-5, to=2-6]
    	\arrow[from=3-2, to=3-3]
    	\arrow[from=3-3, to=3-4]
    	\arrow[from=3-5, to=3-4]
    	\arrow[from=3-5, to=3-6]
    \end{tikzcd}\]
    In general, $\mu_n\circ\dots\circ\mu_1(Q)$ is nearly the same as $Q$, the only difference being that the arrow $n\to n+1$ originally in $Q$ is reversed. Thus, arbitrarily large overfull subquivers induced by finite subsets of vertices eventually stabilize, and moreover the limiting quiver is $Q$ itself. Therefore, we have that $\mu_{(1,2,3,\dots)}(Q)=Q$ in $\LF$ (and by Lemma~\ref{weak-str-comp-lemma}, the same holds in $\AF$).
\end{Ex}

\begin{Ex}\label{inf-mut-seq-ex2}
    Consider the infinite mutation sequence $\mu_{(2,3,4,\dots)}=\dots\mu_4\circ\mu_3\circ\mu_2$ applied to the same type $\mathbb A_\infty$ quiver $Q$ as above. The first few quivers in the sequence $(\mu_{n+1}\circ\dots\circ\mu_2(Q))_{n\geq1}$ are as follows:
    % https://q.uiver.app/#q=WzAsMjEsWzEsMCwiMSJdLFsyLDAsIjIiXSxbMywwLCIzIl0sWzQsMCwiNCJdLFs1LDAsIjUiXSxbMCwwLCJcXG11XzIoUSk6Il0sWzAsMSwiXFxtdV8zKFxcbXVfMihRKSk6Il0sWzEsMSwiMSJdLFsyLDEsIjIiXSxbMywxLCIzIl0sWzQsMSwiNCJdLFs1LDEsIjUiXSxbMCwyLCJcXG11XzQoXFxtdV8zKFxcbXVfMihRKSkpOiJdLFsxLDIsIjEiXSxbMiwyLCIyIl0sWzMsMiwiMyJdLFs0LDIsIjQiXSxbNSwyLCI1Il0sWzYsMCwiXFxkb3RzIl0sWzYsMSwiXFxkb3RzIl0sWzYsMiwiXFxkb3RzIl0sWzIsM10sWzMsNF0sWzEwLDExXSxbMTQsMTVdLFswLDIsIiIsMCx7ImN1cnZlIjotMX1dLFsyLDFdLFsxLDBdLFs3LDEwLCIiLDAseyJjdXJ2ZSI6MX1dLFsxMCw5XSxbOCw5XSxbOSw3LCIiLDIseyJjdXJ2ZSI6MX1dLFs0LDE4XSxbMTEsMTldLFsxNywyMF0sWzE2LDEzLCIiLDIseyJjdXJ2ZSI6LTF9XSxbMTUsMTZdLFsxNywxNl0sWzEzLDE3LCIiLDIseyJjdXJ2ZSI6LTF9XV0=
    \[\begin{tikzcd}[cramped]
    	{\mu_2(Q):} & 1 & 2 & 3 & 4 & 5 & \dots \\
    	{\mu_3(\mu_2(Q)):} & 1 & 2 & 3 & 4 & 5 & \dots \\
    	{\mu_4(\mu_3(\mu_2(Q))):} & 1 & 2 & 3 & 4 & 5 & \dots
    	\arrow[curve={height=-12pt}, from=1-2, to=1-4]
    	\arrow[from=1-3, to=1-2]
    	\arrow[from=1-4, to=1-3]
    	\arrow[from=1-4, to=1-5]
    	\arrow[from=1-5, to=1-6]
    	\arrow[from=1-6, to=1-7]
    	\arrow[curve={height=12pt}, from=2-2, to=2-5]
    	\arrow[from=2-3, to=2-4]
    	\arrow[curve={height=12pt}, from=2-4, to=2-2]
    	\arrow[from=2-5, to=2-4]
    	\arrow[from=2-5, to=2-6]
    	\arrow[from=2-6, to=2-7]
    	\arrow[curve={height=-12pt}, from=3-2, to=3-6]
    	\arrow[from=3-3, to=3-4]
    	\arrow[from=3-4, to=3-5]
    	\arrow[curve={height=-12pt}, from=3-5, to=3-2]
    	\arrow[from=3-6, to=3-5]
    	\arrow[from=3-6, to=3-7]
    \end{tikzcd}\]
    The general term $\mu_{n+1}\circ\dots\circ\mu_2(Q)$ in this sequence is given by the following:
% https://q.uiver.app/#q=WzAsMTEsWzAsMCwiXFxtdV97bisxfVxcY2lyY1xcZG90c1xcY2lyY1xcbXVfMihRKToiXSxbMSwwLCIxIl0sWzIsMCwiMiJdLFszLDAsIjMiXSxbNCwwLCI0Il0sWzUsMCwiXFxkb3RzIl0sWzYsMCwibiJdLFs3LDAsIm4rMSJdLFs4LDAsIm4rMiJdLFs5LDAsIm4rMyJdLFsxMCwwLCJcXGRvdHMiXSxbMiwzXSxbMyw0XSxbNCw1XSxbNSw2XSxbNiw3XSxbOCw3XSxbOSwxMF0sWzgsOV0sWzEsOCwiIiwxLHsiY3VydmUiOi0yfV0sWzcsMSwiIiwxLHsiY3VydmUiOi0yfV1d
\[\begin{tikzcd}[cramped,column sep=small]
	{\mu_{n+1}\circ\dots\circ\mu_2(Q):} & 1 & 2 & 3 & 4 & \dots & n & {n+1} & {n+2} & {n+3} & \dots
	\arrow[curve={height=-24pt}, from=1-2, to=1-9]
	\arrow[from=1-3, to=1-4]
	\arrow[from=1-4, to=1-5]
	\arrow[from=1-5, to=1-6]
	\arrow[from=1-6, to=1-7]
	\arrow[from=1-7, to=1-8]
	\arrow[curve={height=-24pt}, from=1-8, to=1-2]
	\arrow[from=1-9, to=1-8]
	\arrow[from=1-9, to=1-10]
	\arrow[from=1-10, to=1-11]
\end{tikzcd}\]
    On the one hand, the overfull subquiver induced by $\{1\}\subseteq\NN$ does not stabilize in this sequence as $n\to\infty$, so $\mu_{(2,3,4,\dots)}$ diverges on $Q$ in $\LF$. On the other hand, every full subquiver of $\mu_{(2,3,4,\dots,n+1)}(Q)$ stabilizes as $n\to\infty$, meaning $\mu_{(2,3,4,\dots)}(Q)=Q'$ in $\AF$ for the following quiver $Q'$:
    % https://q.uiver.app/#q=WzAsNixbMSwwLCIxIl0sWzIsMCwiMiJdLFszLDAsIjMiXSxbNCwwLCI0Il0sWzUsMCwiXFxkb3RzIl0sWzAsMCwiUSc6Il0sWzEsMl0sWzIsM10sWzMsNF1d
    \[\begin{tikzcd}[cramped,column sep=small]
    	{Q':} & 1 & 2 & 3 & 4 & \dots
    	\arrow[from=1-3, to=1-4]
    	\arrow[from=1-4, to=1-5]
    	\arrow[from=1-5, to=1-6]
    \end{tikzcd}\]
    Note that $Q'$ is itself locally finite. Since $\mu_{(2,3,4,\dots)}$ diverges on $Q$ in $\LF=\LF^\NN_{str}$ while $\mu_{(2,3,4,\dots)}(Q)=Q'$ in $\LF^\NN_{weak}$, one obtains another proof of the fact that $\LF^\NN_{str}$ and $\LF^\NN_{weak}$ are not homeomorphic via the identity map (though Corollary~\ref{LF^N-weak-not-Polish-cor} already establishes that they are not homeomorphic \emph{at all}).
\end{Ex}

Our first result of this section almost completely characterizes the (non-)density of $\mathcal C_\AF(\mu_\mathbf i)$ and $\mathcal D_\AF(\mu_\mathbf i)$ in $\AF$.

\begin{T-AF-inf-seq}
    Let $\mu_{\mathbf i}:=\ldots\mu_{i_3}\mu_{i_2}\mu_{i_1}$ be an infinite mutation sequence on $\AF$. Then the following hold of the sets $\mathcal D_\AF(\mu_\mathbf i)$ and $\mathcal C_\AF(\mu_\mathbf i)$:
        \begin{enumerate}[label=(\roman*)]
            \item $\mathcal D_\AF(\mu_\mathbf i)$ is dense in $\AF$.
            \item If there does not exist $i\in\NN$ appearing infinitely many times in $\mathbf i$, then $\mathcal C_\AF(\mu_\mathbf i)$ is dense in $\AF$.
            \item If there are only finitely many $i\in\NN$ appearing in $\mathbf i$, then $\mathcal C_\AF(\mu_\mathbf i)$ is not dense in $\AF$.
        \end{enumerate}
\end{T-AF-inf-seq}
\begin{proof}
    We prove part (i) of the statement first. Suppose $\mu_\mathbf i=\dots\mu_{i_3}\mu_{i_2}\mu_{i_1}$ is an infinite mutation sequence, and let $Q\in\AF$. We wish to show that for any finite $V\subseteq\NN$, there exists $Q'\in U_{Q,V}$ on which $\mu_\mathbf i$ diverges. Suppose without loss of generality that $Q$ is finite (by the density of $\mathbf{Fin}$ in $\AF$ and Lemma~\ref{U-W-reanchor-lemma}) and $\text{supp}(Q)\subseteq V$. Suppose first that $\mathbf i$ contains some vertex $i\in\NN$ infinitely many times. Consider the case where $i$ is not isolated in $Q$ (in particular, $i\in V$). Since mutations do not change whether a vertex is isolated, one has that $i$ is not isolated in $\mu_{i_n}\circ\dots\circ\mu_{i_1}(Q)$ for any $n\geq1$. Moreover, mutating at $i$ changes the directions of all arrows incident to $i$, of which there exist at least one by hypothesis. Thus, the full subquiver of $\mu_{i_n}\circ\dots\circ\mu_{i_1}(Q)$ induced by $V$ fails to stabilize as $n\to\infty$, so $\mu_\mathbf i$ diverges on $Q$ in $\AF$. In the case where $i$ is isolated in $Q$ (meaning $i$ may or may not be in $V$), then simply extend $Q$ to a quiver $Q'$ containing an arrow $i\to j$ for some $j\in\NN\setminus V$, $j\neq i$. Now a similar argument just given applies for $Q'$ in place of $Q$: the full subquiver of $\mu_{i_n}\circ\dots\circ\mu_{i_1}(Q)$ induced by $V\cup\{i,j\}$ fails to stabilize as $n\to\infty$.
    
    On the other hand, if no vertex $i\in\NN$ appears infinitely many times in $\mathbf i=(i_1,i_2,\dots)$, we extend $Q$ in the following particular manner. Let $N=\max V$ and let $n\geq1$ be such that $i_k>N+2$ for all $k\geq n$. Note that such an $n$ exists by the hypothesis that no vertex appears infinitely many times in $\mathbf i$. Consider the following arrow finite quiver $A$ on the vertex set $\{N+1,N+2,\dots\}$:

    % https://q.uiver.app/#q=WzAsNixbMCwwLCJOKzEiXSxbMCwyLCJOKzIiXSxbMSwxLCJOKzMiXSxbMiwxLCJOKzQiXSxbMywxLCJOKzUiXSxbNCwxLCJcXGRvdHMiXSxbMCwyXSxbMiwxXSxbMCwzXSxbMywxXSxbMCw0XSxbNCwxXSxbMCw1XSxbNSwxXV0=
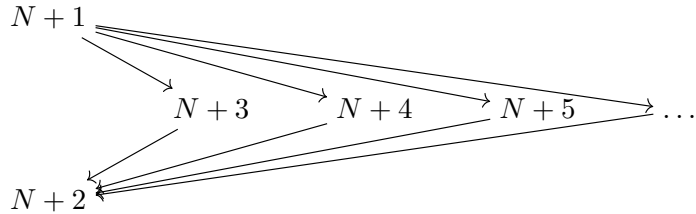
    \[\begin{tikzcd}[cramped]
    	{N+1} &&&& \\
    	& {N+3} & {N+4} & {N+5} & \dots \\
    	{N+2}
    	\arrow[from=1-1, to=2-2]
    	\arrow[from=1-1, to=2-3]
    	\arrow[from=1-1, to=2-4]
    	\arrow[from=1-1, to=2-5]
    	\arrow[from=2-2, to=3-1]
    	\arrow[from=2-3, to=3-1]
    	\arrow[from=2-4, to=3-1]
    	\arrow[from=2-5, to=3-1]
    \end{tikzcd}\]
    \vspace{0.1cm}

    We will place a mutated copy of $A$ on the isolated vertices $\{N+1,N+2,N+3,\dots\}$ in $Q$ to create a new quiver $Q'$ as follows. Replace the isolated vertices $\{N+1,N+2,\dots\}$ of $Q$ with the quiver $A':=\mu_{i_1}\circ\mu_{i_2}\circ\dots\circ\mu_{i_{n-1}}(A)$ to create a new quiver $Q'\in U_{Q,V}$, where $n$ is as above. Note that we may instead write $\mu_{i_{n-1}}\circ\dots\circ\mu_{i_1}(A')=A$, so we have that the full subquiver of $\mu_{i_{n-1}}\circ\dots\circ\mu_{i_1}(Q')$ on the vertices $\{N+1,N+2,\dots\}$ is simply $A$. We then observe that because $i_k>N+2$ for all $k\geq n$, it holds that the only mutations that occur in the tail sequence $\dots\circ\mu_{i_{n+1}}\circ\mu_{i_n}$ are at vertices with label at least $N+3$. By our choice of $A$ above, these mutations are very controlled: they simply turn the arrows incident to $i_k$ in $A$ around and either add or subtract a single arrow from the count of arrows currently in place from $N+1$ to $N+2$ (or vice versa). Since this happens with \emph{every} mutation from $i_k$ onward, the full subquiver of $\mu_{i_k}\circ\dots\circ\mu_{i_1}(Q')$ on the two-vertex set $\{N+1,N+2\}$ fails to stabilize as $k\to\infty$. Therefore, the mutation sequence $\mu_\mathbf i$ diverges on $Q'\in U_{Q,V}$, as desired.

    We now prove part (ii) of the theorem. Suppose $\mathbf i\in\NN^\NN$ is a sequence such that no $i\in\NN$ appears infinitely many times in $\mathbf i$. Let $U_{Q,V}\subseteq\AF$ be a basic open set for some $Q\in\AF$, $V\subseteq\NN$ finite. As before, assume without loss of generality that $Q$ is finite with $\text{supp}(Q)\subseteq V$ and take $N=\max V$. Let $n\geq1$ be such that $i_k>N$ for all $k>n$. Then for all $k>n$, one has $\mu_{i_k}\circ\dots\circ\mu_{i_1}(Q)=\mu_{i_n}\circ\dots\circ\mu_{i_1}(Q)$, so $\mu_\mathbf i$ converges on $Q\in U_{Q,V}$. Hence, we have $U_{Q,V}\cap\mathcal C_\AF(\mu_\mathbf i)\neq\varnothing$. Since $U_{Q,V}$ was chosen arbitrarily, this shows that $\mathcal C_\AF(\mu_\mathbf i)$ is dense in $\AF$.
    
    For part (iii), let $\mathbf i\in\NN^\NN$ be a sequence of natural numbers containing only finitely many $i\in\NN$. Let $V\subseteq\NN$ be the finite set of vertices appearing in $\mathbf i$. By the infinite pigeonhole principle, there exists some $i\in V$ appearing infinitely many times in $\mathbf i$. By our choice of $V$, no $j\in\NN\setminus V$ is used in $\mathbf i$; fix some $j\in\NN\setminus V$. Let $Q$ be the finite quiver on $V\cup\{j\}$ with a single arrow $i\to j$ and so that all vertices $i'\in V$ with $i'\neq i$ are isolated. Then any $Q'\in U_{Q,V\cup\{j\}}$ has $Q$ as its full subquiver on $V\cup\{j\}$. Since $\mathbf i$ only involves vertices in $V$, we see that the mutation sequence $\mu_\mathbf i$ only mutates within the finite full subquiver of $Q'$ induced by $V\cup\{j\}$, which is the same as $Q$. Using the facts that the operations of mutation within a vertex set and restriction to that vertex set commute, that $\mathbf i$ uses $i$ infinitely many times, and that $i$ is not isolated in $Q$, it holds that the full subquiver of $\mu_{(i_1,\dots,i_n)}(Q')$ on $V\cup\{j\}$ fails to stabilize as $n\to\infty$, showing that $\mathcal C_\AF(\mu_\mathbf i)\cap U_{Q,V\cup\{j\}}=\varnothing$ and thus that $\mathcal C_\AF(\mu_\mathbf i)$ is not dense in $\AF$.
\end{proof}

We are not sure, at present, how to close the gap between parts (ii) and (iii) of Theorem~\ref{AF-inf-mut-seq-thm} to account for infinite sequences $\mathbf i\in\NN^\NN$ in which (1) infinitely many $i\in\NN$ appear and (2) there exists some $i\in\NN$ appearing in $\mathbf i$ infinitely many times. This seems to be a bit more subtle and may be the subject of future work.

One interesting consequence of part (i) of this theorem is that there are \emph{no} maps $\AF\to\AF$ which can be expressed as an infinite mutation sequence. Indeed, for any infinite mutation sequence $\mu_\mathbf i$ on $\AF$, there is already a dense set of points in $\AF$ on which $\mu_\mathbf i$ is not defined! In comparison, we will see that the situation is different in the $\LF$ case. In order to proceed, we first must examine the combinatorics of infinite mutation sequences with more scrutiny.

%% Combinatorics of infinite mutation sequences %%
\subsection{Combinatorics of infinite mutation sequences} In this subsection, we define what it means to ``reduce'' an infinite sequence of mutations, then establish some combinatorial properties of sets of vertices which are ``linked'' with respect to an infinite sequence of mutations in a certain sense we will make precise. We choose to examine these aspects of infinite mutation sequences (1) in order to demonstrate that they can be quite complicated and (2) to facilitate a large part of our proof of Theorem~\ref{LF-inf-mut-seq-thm}.

We first consider \emph{reductions} of finite and certain infinite mutation sequences.

\begin{Def}\label{reduction-def}
    Let $\mu_\mathbf i$ be a mutation sequence (either finite or infinite) with no $i\in\NN$ appearing in $\mathbf i$ infinitely many times. We define the \emph{reduction} $\overline\mu_\mathbf i$ of $\mu_\mathbf i$ as the mutation sequence obtained by the following process:
    \begin{enumerate}
        \item Define a \emph{block} $B$ as a nonempty, consecutive substring of $\mathbf i$ consisting of only one $i\in\NN$ which is maximal (with respect to substring inclusion) among consecutive substrings of $\mathbf i$ consisting only of $i$. If $i\in\NN$ is the unique label appearing in a block $B$, we may also call $B$ an \emph{$i$-block.} Let $|B|$ denote the length of the block $B$ as a substring of $\mathbf i$.
        \item Enumerate the (possibly finitely many) blocks of $\mu_\mathbf i$ in order: $B^1_1,B^1_2,B^1_3,\dots$. For each $j\geq1$, let $k^1_j\geq1$ be the least index $k\geq1$ in the block $B^1_j$ (viewed as a substring of $\mathbf i$). In particular, $k^1_1=1$.
        \item Create a new (possibly finite) sequence $R(\mathbf i)$, called the \emph{one-step reduction} of $\mathbf i$, by (i) deleting every even-length block $B^1_j$ from the previous list, (ii) replacing every odd-length $i$-block $B^1_j$ with a single instance of $i$, and (iii) concatenating these modified blocks together in order. We may view $R(\mathbf i)$ as a function on the (possibly finite) domain $K_1:=\{k^1_j\mid B^1_j\text{ is an odd-length block}\}\subseteq\NN$. Let $B^2_1,B^2_2,B^2_3,\dots$ be the (possibly finite) list of the blocks of $R(\mathbf i)$, and for all $j\geq1$ let $k^2_j$ be the least index among the domain $K_1$ of $R(\mathbf i)$ used in the block $B^2_j$.
        \item Iteratively repeat step (3) to create further sequences $R^m(\mathbf i)$ viewed as functions with domains $K_m$, sequences of blocks $B^{m+1}_1,B^{m+1}_2,B^{m+1}_3,\dots$, and corresponding sequences $k^{m+1}_1,k^{m+1}_2,k^{m+1}_3,\dots$ of least indices, all three of which may possibly be finite.
        \item Let $K=\{k_1,k_2,k_3,\dots\}:=\bigcap_{m\geq1}K_m$. We may finally define $\overline\mu_\mathbf i$ as the (possibly finite) mutation sequence $\overline\mu_\mathbf i:=\dots\mu_{i_{k_3}}\circ\mu_{i_{k_2}}\circ\mu_{i_{k_1}}$. For future reference, we also define $R^\omega(\mathbf i):=(i_{k_1},i_{k_2},\dots)\in\NN^{\leq\NN}$ (so that $\overline\mu_\mathbf i=\mu_{R^\omega(\mathbf i)}$).
    \end{enumerate}
    We call a mutation sequence $\mu_\mathbf i$ (with every $i\in\NN$ appearing only finitely many times in $\mathbf i$) \emph{reduced} if $\overline\mu_\mathbf i=\mu_\mathbf i$.
    
    If there is some $m\geq0$ such that $R^m(\mathbf i)=R^{m+1}(\mathbf i)$, we let the \emph{rank} $\text{rk}(\mathbf i)$ of $\mathbf i$ be the least such $m\geq0$. If no such $m$ exists, we set $\text{rk}(\mathbf i)=\omega$. For fixed $i\in\NN$, we also let the \emph{$i$-rank} $\text{rk}_i(\mathbf i)$ of $\mathbf i$ be the least $m\geq0$ such that for all $n\geq m$, the number of instances of $i$ in $R^n(\mathbf i)$ is the same as the number of instances of $i$ in $R^m(\mathbf i)$. Note that such an $m\geq0$ exists since the number of instances of $i$ in the original sequence $\mathbf i$ is finite and the number of instances in $R^m(\mathbf i)$ of $i$ is weakly decreasing as $m$ increases. Furthermore, $\text{rk}(\mathbf i)=\sup\{\text{rk}_i(\mathbf i)\mid i\in\NN\}$.
\end{Def}

\begin{Ex}\label{reduction-ex1}
    Suppose $\mathbf i=(1,1,2,2,3,3,4,4,\dots)$. Then $B^1_1$ is the first (and only) $1$-block and $k^1_1=1$, $B^1_2$ is the only $2$-block and $k^1_2=3$, $B^1_3$ is the only $3$-block and $k^1_3=5$, and so on. Each block is even-length, so $R(\mathbf i)$ is the empty sequence and $K_1=\varnothing$. After iterating, we have that $R^\omega(\mathbf i)$ is still the empty sequence and $K=\varnothing$. Thus, we see that $\overline\mu_{(1,1,2,2,3,3,\dots)}$ is the empty mutation sequence and $\text{rk}(\mathbf i)=1$. Moreover, for all $i\in\NN$, we have $\text{rk}_i(\mathbf i)=1$.
\end{Ex}

\begin{Ex}\label{reduction-ex2}
    Let $\mathbf i=(1,1,2,3,3,2,4,5,6,6,5,4,\dots,\varrho_n^\uparrow,\varrho_n^\downarrow,\dots)$, where for all $n\geq 1$, $\varrho_n^\uparrow$ is the finite sequence given by
    $$\varrho_n^\uparrow=\left(\frac{n(n-1)}2+1,\frac{n(n-1)}2+2,\dots,\frac{(n+1)n}2\right)$$
    and $\varrho_n^\downarrow$ is its reverse. One may readily check that the one-step reduction $R(\mathbf i)$ of $\mathbf i$ is given by $R(\mathbf i)=(2,2,4,5,5,4,7,8,9,9,8,7,\dots)$, which is intuitively seen to be the result of removing all triangular numbers $\frac{(n+1)n}2$ from $\mathbf i$. Taking the two-step reduction $R^2(\mathbf i)$ removes all entries one less than a triangular number from $R(\mathbf i)$. In general for all $m\geq1$, $R^m(\mathbf i)$ is the result of deleting from $\mathbf i$ all labels $i\in\NN$ such that there exists a triangular number in the interval $[i,i+m-1]$. Since consecutive triangular numbers grow arbitrarily far apart, it is never the case that $R^m(\mathbf i)=R^{m+1}(\mathbf i)$ for any $m\geq0$. Furthermore, since every $i\in\NN$ is less than a triangular number, every $i\in\NN$ is eventually removed from $\mathbf i$ at some finite stage $m\geq1$. Therefore, $\overline\mu_\mathbf i$ is the empty mutation sequence and $\text{rk}(\mathbf i)=\omega$. Also, for all $i\in\NN$, we have $\text{rk}_i(\mathbf i)$ is the least $m\geq1$ such that $i+m-1$ is a triangular number.
\end{Ex}

The operation of reduction satisfies a couple of basic properties that are useful to point out with the following lemma.

\begin{Lem}\label{reduction-lemma1}
    Let $\mu_\mathbf i$ be an infinite mutation sequence in which no vertex $i\in\NN$ appears in $\mathbf i$ infinitely many times. Then $\overline\mu_\mathbf i$ is reduced. Additionally, $\mu_\mathbf i$ is reduced itself if and only if $\text{rk}(\mathbf i)=0$.
\end{Lem}
\begin{proof}
    We first show the second part of the claim that $\mu_\mathbf i$ is reduced if and only if $\text{rk}(\mathbf i)=0$. Suppose first that $\mu_\mathbf i$ is reduced. By definition, this means that $\mu_\mathbf i=\overline\mu_\mathbf i$. For each $i\in\NN$, let $|\mu_\mathbf i|_i$ be the total number of instances of $i\in\NN$ appearing in $\mathbf i$. Then we have that $|\mu_\mathbf i|_i=|\overline\mu_\mathbf i|_i$ for all $i\in\NN$. Furthermore, one has the following chain of inequalities for all $i\in\NN$:
    $$|\mu_\mathbf i|_i\geq|\mu_{R(\mathbf i)}|_i\geq|\mu_{R^2(\mathbf i)}|_i\geq\dots\geq|\overline\mu_\mathbf i|_i=|\mu_\mathbf i|_i.$$
    As a consequence, we have equality all the way through. Since one-step reductions do not change the relative order of entries in a sequence, one can then infer that these equalities lift to equalities of sequences. In particular we see that $\mathbf i=R(\mathbf i)$, or equivalently that $\text{rk}(\mathbf i)=0$. On the other hand, if we start with the assumption that $\text{rk}(\mathbf i)=0$, then one immediately has $\mathbf i=R(\mathbf i)=R^2(\mathbf i)=\dots$. It then follows that $\overline\mu_\mathbf i=\mu_\mathbf i$.

    The first part of the claim follows from the second part of the claim if we can show that $\text{rk}(R^\omega(\mathbf i))=0$, which is to say that $R^\omega(\mathbf i)=R(R^\omega(\mathbf i))$. Suppose for the sake of contradiction that this is not true. By definition, there must exist some $p\geq1$ such that $i_{k_p}=i_{k_{p+1}}$. Since $k_p,k_{p+1}\in K$ are consecutive in the natural order on $K=\bigcap_{m\geq1}K_m$, there must exist some $m\geq1$ and some $q\geq1$ (in fact $q\geq p$) such that $k_p=k^m_q$ and $k_{p+1}=k^m_{q+1}$. But since $i_{k^m_q}=i_{k_p}=i_{k_{p+1}}=i_{k^m_{q+1}}$, one has that the indices $k^m_q$ and $k^m_{q+1}$ belong to the same block $B^{m+1}_r$ (for some $r\leq q$) in the next stage. From this, one has that $k_{p+1}=k^m_{q+1}$ is not a member of $K_{m+1}$, contradicting $k_{p+1}\in K$. Hence, it follows that $\text{rk}(R^\omega(\mathbf i))=0$, so $\overline\mu_\mathbf i$ is reduced.
\end{proof}

We will also need the following lemma.

\begin{Lem}\label{reduction-lemma2}
    Let $\mu_\mathbf i$ be an infinite mutation sequence in which no vertex $i\in\NN$ appears in $\mathbf i$ infinitely many times and such that $\overline\mu_\mathbf i$ is finite. Then there exists some $n\geq0$ with the properties that
    \begin{enumerate}
        \item $\overline\mu_\mathbf i=\overline{\mu_{i_n}\circ\dots\circ\mu_{i_1}}$.
        \item $\overline{\dots\circ\mu_{i_{n+2}}\circ\mu_{i_{n+1}}}$ is the empty mutation sequence.
    \end{enumerate}
\end{Lem}
\begin{proof}
    If $\overline\mu_\mathbf i$ is the empty mutation sequence, taking $n=0$ above works. Assume then that $\overline\mu_\mathbf i$ is not the empty mutation sequence. Since $\overline\mu_\mathbf i$ is finite by hypothesis, it follows that $K=\{k_1,\dots,k_\ell\}$ is finite. It also follows that $\ell\geq1$ since $\overline\mu_\mathbf i$ is nonempty. We claim that taking $n=k_\ell$ in the statement of the lemma works.
    
    Our first goal is to show that $\overline{\mu_{i_{k_\ell}}\circ\dots\circ\mu_{i_2}\circ\mu_{i_1}}=\overline\mu_\mathbf i$. Observe that it suffices to show the domain $K'$ of $R^\omega(i_1,i_2,\dots,i_{k_\ell})$ agrees with the domain $K$ of $R^\omega(\mathbf i)$. We begin by observing that since $k_\ell\in K$, the blocks $B^m_{j_m}$ of $R^{m-1}(\mathbf i)$ to which the index $k_\ell$ belongs (as $m\geq1$ varies) each satisfy two properties: they are all of odd-length and $k_\ell$ is the least index of each. Indeed, failure of either of these properties at any stage $m$ would disqualify $k_\ell$ from belonging to $K_m$ and therefore to $K$. Since $k_\ell$ is the least index of $B^m_{j_m}$ for each $m\geq1$, we see that the blocks $B^m_1,\dots,B^m_{j_m-1}$ have indices contained entirely within the set $\{1,2,\dots,k_\ell-1\}$ for all $m\geq1$, and therefore that the changes in these blocks as $m$ varies are dictated entirely by the finite initial segment $(i_1,i_2,\dots,i_{k_\ell-1})$ of $\mathbf i$. This shows that the domain $K'\setminus\{k_\ell\}$ of $R^\omega(i_1,i_2,\dots,i_{k_\ell-1})$ is equal to $K\setminus\{k_\ell\}$, and therefore that $K'=K$. Therefore $\overline{\mu_{i_{k_\ell}}\circ\dots\circ\mu_{i_2}\circ\mu_{i_1}}=\overline\mu_\mathbf i$.

    To show the second part of the claim, we note that the blocks $B^m_{j_m}$ of $R^{m-1}(\mathbf i)$ containing $k_\ell$ agree with the blocks $B'^{\,m}_1$ of $R^{m-1}(i_{k_\ell},i_{k_{\ell+1}},\dots)$ (as $m$ varies) for similar reasons. Moreover, all further blocks $B^m_{j_m+t}$ and $B'^{\,m}_{1+t}$ for $t\geq0$ agree as well. Since $k_\ell=\max K$, we see that $K'=\{k_\ell\}$. Hence, $\overline{\dots\mu_{i_{k_\ell+1}}\circ\mu_{i_{k_\ell}}}=\mu_{k_\ell}$, and therefore the original claim that $\overline{\dots\mu_{i_{k_\ell+2}}\circ\mu_{i_{k_\ell+1}}}$ is the empty mutation sequence follows.
\end{proof}

Additionally, the following property is useful to point out.

\begin{Lem}\label{reduction-lemma3}
    Let $\mu_\mathbf i$ be an infinite mutation sequence in which no vertex $i\in\NN$ appears in $\mathbf i$ infinitely many times and such that $\overline\mu_\mathbf i$ is the empty mutation sequence. Then there exist infinitely many $k\geq1$ such that $\overline{\mu_{i_k}\circ\dots\circ\mu_{i_1}}$ is the empty mutation sequence.
\end{Lem}
\begin{proof}
    It suffices to prove there exists a single $k\geq1$ with the desired property, since the tail $\dots\mu_{i_{k+2}}\circ\mu_{i_{k+1}}$ will also reduce to the empty sequence, allowing us to repeat the argument indefinitely to produce infinitely many such $k\geq1$. Suppose for the sake of contradiction that $\mu_\mathbf i$ is an infinite mutation sequence with the specified properties but such that there exists no $k\geq1$ with $\overline{\mu_{i_k}\circ\dots\circ\mu_{i_1}}$ equal to the empty mutation sequence. Let $\ell\geq1$ be the minimal length of $R^\omega(i_1,\dots,i_k)$ over all $k\geq1$. We induct on $\ell$.

    For our base case, suppose $\ell=1$. Take $k\geq1$ so that $\overline{\mu_{i_k}\circ\dots\circ\mu_{i_1}}=\mu_r$ for some $r\in\NN$. By the definition of reduction, there exists some $1\leq h\leq k$ so that the surviving instance of $r$ occurs at position $h$, that every block $B^m_{j_m}$ of $R^{m-1}(i_1,\dots,i_k)$ containing the index $h$ is an $r$-block of odd length, and that $B^m_{j_m-1}$ is not an $r$-block for all $m\geq1$. Because $\overline\mu_\mathbf i$ is the empty sequence, however, there exists some $k'>k$ and some least $m\geq1$ so that $h\not\in K'_m$ for $K'_m$ the domain of $R^m(i_1,\dots,i_{k'})$. Taking some minimal such $k'$ in fact forces $R^\omega(i_1,\dots,i_{k'})=\varnothing$. Indeed, by minimality the greatest block in $R^{m-1}(i_1,\dots,i_{k'})$ must be an $r$-block, and furthermore it must merge with the $r$-block of $R^{m-1}(i_1,\dots,i_{k'})$ at the next stage $m$. This implies that there are no $s$-blocks for $s\neq r$ between these two $r$-blocks. Since all blocks below index $h$ must vanish as $m\to\infty$ (by the assumption $\overline{\mu_{i_k}\circ\dots\circ\mu_{i_1}}=\mu_r$ for the instance of $r$ at index $h$), we do in fact have $R^\omega(i_1,\dots,i_{k'})=\varnothing$. This contradicts the minimality of the length $\ell$ of $R^\omega(i_1,\dots,i_k)$.

    For the inductive hypothesis, suppose the claim holds for all $\ell<d$. We show that the claim holds for $\ell=d$. Take $k\geq1$ so that $\overline{\mu_{i_k}\circ\dots\circ\mu_{i_1}}=\mu_{r_d}\circ\dots\circ\mu_{r_1}$ for some $r_1,\dots,r_d\in\NN$. We argue similarly to the base case. Take $1\leq h\leq k$ to be such that the (largest) surviving instance of $r_d$ occurs at position $h$, that every block $B^m_{j_m}$ of $R^{m-1}(i_1,\dots,i_k)$ containing the index $h$ is an $r_d$-block of odd length, and that $B^m_{j_m-1}$ is not an $r_d$-block for all $m\geq1$. Since $\overline\mu_\mathbf i$ is the empty sequence, there exists some $k'>k$ and some least $m\geq1$ so that $h\not\in K'_m$ for the domain $K'_m$ of $R^m(i_1,\dots,i_{k'})$. Proceeding similarly to the base case allows us to conclude that $\overline{\mu_{i_{k'}}\circ\dots\circ\mu_{i_1}}=\mu_{r_{d-1}}\circ\dots\circ\mu_{r_1}$, contradicting the minimality of $\ell=d$. This completes the proof.
\end{proof}

Our next goal is to classify the sets of vertices which ``contribute nontrivially'' to an infinite mutation sequence $\mu_\mathbf i$ for $\mathbf i\in\NN^\NN$. We begin with a definition.

\begin{Def}\label{inf-seq-subsequence-def}
    Let $\mathbf i=(i_1,i_2,\dots)\in\NN^\NN$ and $S\subseteq\NN$. We say that the \emph{subsequence of $\mathbf i$ induced by $S$} is the (possibly finite) sequence $\mathbf i_S=(i'_1,i'_2,i'_3,\dots)\in\NN^{\leq\NN}$ obtained by removing all instances of all $i\in\NN\setminus S$ from $\mathbf i$ and keeping all instances of $i\in S$ in $\mathbf i$.
\end{Def}

Given an infinite sequence $\mathbf i\in\NN^\NN$, some sets $S\subseteq\NN$ induce ``nontrivial'' subsequences $\mathbf i_S$ with respect to the operation of reduction. We make this precise as follows.

\begin{Def}\label{linked-sets-families-def}
    Let $\mathbf i\in\NN^\NN$ be an infinite sequence such that each $i\in\NN$ appears only finitely many times in $\mathbf i$. We say that a set $S\subseteq\NN$ is \emph{linked with respect to $\mathbf i$} if it holds that $R^\omega(\mathbf i_S)\neq\varnothing$ (and otherwise we say $S$ is \emph{unlinked with respect to $\mathbf i$}). We let $\mathcal L_\mathbf i\subseteq\mathcal P(\NN)$ be the family of all $S\subseteq\NN$ which are linked with respect to $\mathbf i$, which we call the \emph{linking family of $\mathbf i$}.
\end{Def}

We now collect a few basic properties of the linking family $\mathcal L_\mathbf i$.

\begin{Lem}\label{linking-lemma}
    Let $\mathbf i\in\NN^\NN$ be such that each $i\in\NN$ appears only finitely many times in $\mathbf i$. Then the following hold of the linking family $\mathcal L_\mathbf i$:
    \begin{enumerate}[label=(\roman*)]
        \item $\varnothing\not\in\mathcal L_\mathbf i$.
        \item $\NN\in\mathcal L_\mathbf i$ if and only if $R^\omega(\mathbf i)\neq\varnothing$.
        \item The family $\mathcal L_\mathbf i\subseteq\mathcal P(\NN)$ is upward closed under inclusion.
    \end{enumerate}
\end{Lem}
\begin{proof}
    Part (i) follows immediately by Definition~\ref{linked-sets-families-def}, and part (ii) follows by Definition~\ref{linked-sets-families-def} and the observation that $\mathbf i_\NN=\mathbf i$. For part (iii), we will instead show that $\mathcal P(\NN)\setminus\mathcal L_\mathbf i$ is downward closed under inclusion. Suppose that $S\subseteq T\subseteq\NN$ and that $T$ is unlinked with respect to $\mathbf i$, meaning $R^\omega(\mathbf i_T)=\varnothing$. We wish to show that $R^\omega(\mathbf i_S)=\varnothing$ as well. Compute $R^\omega(\mathbf i_S)$ and $R^\omega(\mathbf i_T)$ according to Definition~\ref{reduction-def} and, while doing so, keep track of the respective blocks $B^{S,m}_1,B^{S,m}_2,\dots$ and $B^{T,m}_1,B^{T,m}_2,\dots$ and their respective sets of least indices $K_{S,m}:=\{k^{S,m}_1,k^{S,m}_2,\dots\}$ and $K_{T,m}:=\{k^{T,m}_1,k^{T,m}_2,\dots\}$ at each stage $m\geq1$. We claim that for all $m\geq1$, it holds that $K_{S,m}\subseteq K_{T,m}$. From this claim, it will follow that $K_S=\bigcap_{m\geq1} K_{S,m}\subseteq\bigcap_{m\geq1}K_{T,m}=K_T=\varnothing$ (and therefore also that $R^\omega(\mathbf i_S)=\varnothing$), where $K_S$ and $K_T$ are as in the final step of Definition~\ref{reduction-def}.
    
    To see that the claim is true, we observe that for all $m\geq1$ the blocks $B^{S,m}_1,B^{S,m}_2,\dots$ can \emph{almost} be obtained from the blocks $B^{T,m}_1,B^{T,m}_2,\dots$ by removing all $j$-blocks for $j\in T\setminus S$; the only catch is that we may end up with consecutive $i$-blocks for $i\in S$ by doing so (which cannot occur by the first step of Definition~\ref{reduction-def}), and we may also end up with blocks not corresponding to any block $B^{S,m}_\ell$. However, the claim that $K_{S,m}\subseteq K_{T,m}$ is recovered from this deletion by the following observation: the $i$-blocks $B^{S,m}_\ell$ for $i\in S$ may be naturally partitioned into (some of) the $i$-blocks $B^{T,m}_p$. This is readily proved by induction on $m\geq1$; we leave it as an exercise with Definition~\ref{reduction-def} for the reader to confirm (note that there is mild casework involved when the $i$-blocks $B^{T,m}_p$ partitioning some $B^{S,m}_\ell$ are odd-length while the $i$-block $B^{S,m}_\ell$ itself is not). In particular, the least index of $B^{S,m}_\ell$ for any $\ell\geq1$ is the least index of $B^{T,m}_p$ for some $p\geq1$, showing $K_{S,m}\subseteq K_{T,m}$ and finishing the proof.
\end{proof}

In fact, the linking family $\mathcal L_\mathbf i$ satisfies several more intricate properties, which we also establish now.

\begin{Prop}\label{linking-prop}
    Let $\mathbf i\in\NN^\NN$ be such that each $i\in\NN$ appears only finitely many times in $\mathbf i$ and so that for all $k\geq1$, it holds that $i_k\neq i_{k+1}$ (equivalently, $\mathbf i=R(\mathbf i)=R^\omega(\mathbf i)$ is reduced). Then the following four properties hold of $\mathcal L_\mathbf i$:
    \begin{enumerate}[label=(\roman*)]
        \item For any nonempty finite set $\varnothing\neq S\subseteq\NN$, there are only finitely many finite sets $S'\supseteq S$ such that $S'\in\mathcal L_\mathbf i$ \emph{and} $S'$ is minimal (with respect to inclusion) in $\mathcal L_\mathbf i$.
        \item There exist infinitely many finite sets $S\in\mathcal L_\mathbf i$ which are minimal in $\mathcal L_\mathbf i$.
        \item For every $S\in\mathcal L_\mathbf i$, there is some finite $S'\subseteq S$ belonging to $\mathcal L_\mathbf i$ which is minimal in $\mathcal L_\mathbf i$.
        \item All minimal $S\in\mathcal L_\mathbf i$ are finite.
    \end{enumerate}
\end{Prop}
\begin{proof}
    We work toward (i) first. We claim that it suffices to show that this condition holds for $S=\{i\}\subseteq\NN$ a singleton; suppose under the assumption that (i) holds for singletons $S$ that there is some finite $T\subseteq\NN$ contained in infinitely many minimal finite sets $S'\in\mathcal L_\mathbf i$. Then for any $i\in T$, the singleton $\{i\}$ is contained in infinitely many such $S'$ as well, a contradiction. Therefore, it suffices to show that (i) holds for singletons.

    To proceed, we define the \emph{convex hull of $i\in\NN$ with respect to $\mathbf i$} to be the set $\text{conv}_\mathbf i(i)\subseteq\NN$ given by
    $$\text{conv}_\mathbf i(i)=\left\{j\in\NN\mid\exists k_1\leq k\leq k_2:i=i_{k_1}=i_{k_2}\text{ and }j=i_k\right\}.$$
    Note that $\text{conv}_\mathbf i(i)=\varnothing$ if $i$ does not appear in $\mathbf i$ and $i\in\text{conv}_\mathbf i(i)$ otherwise. Observe further that $\text{conv}_\mathbf i(i)$ is finite by the assumption that no $i\in\NN$ appears in $\mathbf i$ infinitely many times. We claim two important things: that $\text{conv}_\mathbf i(i)\in\mathcal L_\mathbf i$ for all $i$ appearing in $\mathbf i$ and, moreover, that all finite minimal sets $S'\in\mathcal L_\mathbf i$ containing $i$ are subsets of $\text{conv}_\mathbf i(i)$. If we can prove these, then we have shown part (i) of the proposition: there are only finitely many $S'\subseteq\text{conv}_\mathbf i(i)$, and therefore only finitely many finite \emph{minimal} $S'\in\mathcal L_\mathbf i$ containing $i$.

    To argue toward either of these two claims, we will appeal to the group $G=\langle g_i\mid g_i^2=\varepsilon,i\in\NN\rangle$ and certain endomorphisms of $G$. Elements of $G$ can be naturally identified with finite sequences $(j_1,j_2,\dots,j_\ell)\in\NN^{<\NN}$ with $\ell\geq0$ such that $j_k\neq j_{k+1}$ for all $1\leq k<\ell$: simply take such a $\mathbf j=(j_1,j_2,\dots,j_\ell)$ to the group element $g_\mathbf j=g_{j_1}g_{j_2}\dots g_{j_\ell}\in G$. Note that $g_\mathbf j=g_{R^\omega(\mathbf j)}$; the group $G$ ``automatically'' performs reduction of sequences. For each $i\in\NN$, define an endomorphism $\varphi_i:G\to G$ on the generators of $G$ by
    $$\varphi_i(g_j)=\begin{cases}
        \varepsilon & \text{if }j=i\\
        g_j & \text{if }j\neq i.
    \end{cases}$$
    We note the following useful property: for any $g\in G$ and $i\in\NN$, it holds that $\varphi_i(g)=g$ if and only if $g$ is in the image of $\varphi_i$. Certainly the ``only if'' direction holds, so suppose that $g$ is in the image of $\varphi_i$, say $g=\varphi_i(h)$ for $h\in G$. Since $\varphi_i$ is idempotent, we have that $\varphi_i(g)=\varphi_i^2(h)=\varphi_i(h)=g$.
    
    Another useful property is that for $\mathbf j=(j_1,\dots,j_\ell)\in\NN^{<\NN}$ with $j_k\neq j_{k+1}$ for all $1\leq k<\ell$ as above and for any $i\in\NN$, it holds that $i$ appears in $\mathbf j$ if and only if $g_\mathbf j$ is not in the image of $\varphi_i$. We argue this using the \emph{length function} $L:G\to\ZZ_{\geq0}$ given by $L(g_{(j_1,\dots,j_\ell)}):=\ell$. Note that two facts hold of the length function $L$: that for all $g,h\in G$, we have $L(gh)\leq L(g)+L(h)$ and that for all $g\in G$, $i\in\NN$, we have $L(\varphi_i(g))\leq L(g)$. Now suppose first that $\mathbf j$ contains an instance of $i$, and let $\mathbf j':=\mathbf j_{\NN\setminus\{i\}}$ be the subsequence of $\mathbf j$ obtained by removing all instances of $i$. This yields
    $$L(\varphi_i(g_\mathbf j))=L(g_{\mathbf j'})=|R^\omega(\mathbf j')|<|\mathbf j|=L(g_\mathbf j),$$
    so $\varphi_i(g_\mathbf j)\neq g_\mathbf j$. By our previous useful property, this implies that $g_\mathbf j$ is not in the image of $\varphi_i$. For the converse, suppose that $\mathbf j$ does not contain any instances of $i$. Then, by the definition of $\varphi_i$, we see that $\varphi_i(g_\mathbf j)=g_\mathbf j$, so $g_\mathbf j$ is in the image of $\varphi_i$. Altogether, this gives us a group-theoretic criterion for detecting when an element $i\in\NN$ appears in a reduced finite sequence $\mathbf j$.

    We now establish that $\text{conv}_\mathbf i(i)$ is linked with respect to $\mathbf i$ for $i$ appearing in $\mathbf i$. Observe that we may write $\gamma:=g_{\mathbf i_{\text{conv}_\mathbf i(i)}}=\tau'\sigma\tau$ for $\tau',\sigma,\tau\in G$ such that:
    \begin{enumerate}
        \item $\sigma\neq\varepsilon$, i.e. $\sigma$ comes from a substring of $\mathbf i_{\text{conv}_\mathbf i(i)}$ with nonempty reduction.
        \item $\varphi_i(\sigma)\neq\sigma$, i.e. $g_i$ appears in the reduced expression for $\sigma$.
        \item $\varphi_i(\tau)=\tau$, $\varphi_i(\tau')=\tau'$, i.e., $g_i$ does not appear in the reduced expressions for $\tau$ or $\tau'$.
    \end{enumerate}
    Indeed, we may take $\sigma=g_\mathbf j$ for the consecutive substring $\mathbf j=(i_{e_1},i_{e_1+1},\dots,i_{e_2-1},i_{e_2})$ of $\mathbf i$, where $i_{e_1}$ is the first appearance of $i$ in $\mathbf i$ and $i_{e_2}$ is the last. We then take $\tau'$ to capture all instances of any $j\in\text{conv}_\mathbf i(i)$ appearing at indices less than $e_1$ and we take $\tau$ to capture all instances of any $j\in\text{conv}_\mathbf i(i)$ appearing at indices after $e_2$.
    
    We claim that $\gamma\neq\varepsilon$, or equivalently that $R^\omega(\mathbf i_{\text{conv}_\mathbf i(i)})\neq\varnothing$. Suppose for the sake of contradiction that $\gamma=\varepsilon$. Rewriting $\gamma=\tau'\sigma\tau$, we then see that $\sigma=(\tau\tau')^{-1}$. Applying $\varphi_i$ to both sides yields
    $$\varphi_i(\sigma)=\varphi_i((\tau\tau')^{-1})=(\tau\tau')^{-1}=\sigma,$$
    a contradiction with $\varphi_i(\sigma)\neq\sigma$.

    Toward the latter of our two important claims, we conduct a similar argument to show that all finite minimal sets $S'\in\mathcal L_\mathbf i$ are subsets of $\text{conv}_\mathbf i(i)$. Suppose for the sake of contradiction that $S'\in\mathcal L_\mathbf i$ is a finite minimal set containing $i$ which is not contained in $\text{conv}_\mathbf i(i)$; take some $j\in S'\setminus \text{conv}_\mathbf i(i)$. We again may write $\gamma:=g_{\mathbf i_{S'}}=\tau'\sigma\tau$ for $\tau',\sigma,\tau\in G$ such that:
    \begin{enumerate}
        \item $\sigma\neq\varepsilon$, i.e. $\sigma$ comes from a substring of $\mathbf i_{\text{conv}_\mathbf i(i)}$ with nonempty reduction.
        \item $\varphi_i(\sigma)\neq\sigma$, i.e. $g_i$ appears in the reduced expression for $\sigma$.
        \item $\varphi_i(\tau)=\tau$, $\varphi_i(\tau')=\tau'$, i.e., $g_i$ does not appear in the reduced expressions for $\tau$ or $\tau'$.
        \item $\varphi_j(\sigma)=\sigma$, i.e., $g_j$ does not appear in the reduced expression for $\sigma$.
        \item Either $\varphi_j(\tau')\neq\tau'$ or $\varphi_j(\tau)\neq\tau$ (or both). That is, there is an instance of $g_j$ in at least one of the reduced expressions for $\tau$ and $\tau'$. We may assume without loss of generality that $\varphi_j(\tau')\neq\tau'$.
        \item $\gamma\neq\varepsilon$, i.e., $R^\omega(\mathbf i_{S'})\neq\varnothing$.
    \end{enumerate}
    Indeed, condition (4) here is capturing the fact that $j\not\in\text{conv}_\mathbf i(i)$, while condition (5) is requiring our use of at least one instance of $j$ in $\mathbf i_{S'}$. We also may reformulate the minimality condition on $S'\in\mathcal L_\mathbf i$ in the following group-theoretic fashion: it is equivalent to say that $\gamma\in\bigcap_{s\in S'}\ker\varphi_s$. Indeed, taking away all instances of $s\in S'$ in $\mathbf i_{S'}$ will yield a sequence $\mathbf i_{S'\setminus\{s\}}$ which reduces to $\varnothing$ by the minimality of $S'$, so $\varphi_s(\gamma)=\varepsilon$ for all $s\in S'$.

    We now obtain a contradiction in a similar manner to before. Take $\varphi_j$ on both sides of $\gamma=\tau'\sigma\tau$ to see that
    \begin{align*}
        \varepsilon&=\varphi_j(\tau'\sigma\tau)\\
        &=\varphi_j(\tau')\varphi_j(\sigma)\varphi_j(\tau)\\
        &=\varphi_j(\tau')\sigma\varphi_j(\tau),\\
    \end{align*}
    yielding $\sigma=\varphi_j(\tau\tau')^{-1}$. Then, applying $\varphi_i$ to both sides of this equation (and noting that $\varphi_i\circ\varphi_j=\varphi_j\circ\varphi_i$) yields
    \begin{align*}
        \varphi_i(\sigma)&=\varphi_i(\varphi_j(\tau\tau'))^{-1}\\
        &=\varphi_j(\varphi_i(\tau\tau'))^{-1}\\
        &=\varphi_j(\tau\tau')^{-1}\\
        &=\sigma,
    \end{align*}
    a contradiction to condition (2) above. Thus, we cannot have a finite minimal $S'\in\mathcal L_\mathbf i$ containing $i$ which is not contained in $\text{conv}_\mathbf i(i)$. This concludes the proof of part (i).

    Part (ii) follows from part (i) and the fact that there are infinitely many $i\in\NN$ appearing in $\mathbf i$. Indeed, for any $i\in\NN$ appearing in $\mathbf i$, by the argument for part (i) provided above, there exists at least one finite minimal set $S'\in\mathcal L_\mathbf i$ containing $i$, and there exist only finitely many such $S'$. Let $m\in\NN$ be the maximum element (in the natural order on $\NN$) of all such $S'$. Repeat this argument for the least $i'>m$ appearing in $\mathbf i$ to obtain more finite minimal sets; repeat this argument ad infinitum to obtain infinitely many.

    Finally, for parts (iii) and (iv), we first observe that part (iv) is an immediate consequence of part (iii), so this is all that is left to show. Let $S\in\mathcal L_\mathbf i$. We want to show there is some finite minimal $S'\subseteq S$ which is itself linked with respect to $\mathbf i$. Note that it suffices to show that there is some finite $S'\subseteq S$ which is linked. Assume without loss of generality that $S$ is infinite. If $R^\omega(\mathbf i_S)\neq\varnothing$ is finite, simply let $S'$ be the finite set of $j\in\NN$ appearing in $R^\omega(\mathbf i_S)$. Then $S'$ is linked in $\mathbf i$ precisely because $R^\omega(\mathbf i_{S'})=R^\omega(\mathbf i_S)$. Hence, we may suppose that $R^\omega(\mathbf i_S)$ is infinite; replace $\mathbf i_S$ with $R^\omega(\mathbf i_S)$ without loss of generality. Now $\mathbf i_S$ is an infinite reduced sequence with no $i\in\NN$ appearing infinitely many times. In particular, we may apply our argument for part (i) to $\mathbf i_S$ to obtain for any $i\in S$ a finite set $\text{conv}_{\mathbf i_S}(i)\subseteq S$ containing $i$ which is linked with respect to $\mathbf i_S$. Then it also follows that $\text{conv}_{\mathbf i_S}(i)$ is linked with respect to the original sequence $\mathbf i$, so it holds that $S':=\text{conv}_{\mathbf i_S}(i)\in\mathcal L_\mathbf i$ meets our conditions. This completes the proof.
\end{proof}

We now provide an explicit construction to demonstrate a kind of converse to Lemma~\ref{linking-lemma} and Proposition~\ref{linking-prop}: if $\mathcal L\subseteq\mathcal P(\NN)$ is a nonempty, proper, upward closed subset satisfying the conclusions of Proposition~\ref{linking-prop}, then there is some $\mathbf i\in\NN^\NN$ giving rise to $\mathcal L$.

\begin{Prop}\label{linking-constr-prop}
    Let $\mathcal L\subseteq\mathcal P(\NN)$ be a nonempty, proper, upward closed subset of $\mathcal P(\NN)$ satisfying the conclusions (i)-(iv) of Proposition~\ref{linking-prop} above. Then there exists some $\mathbf i\in\NN^\NN$ which is reduced and is such that no $i\in\NN$ appears infinitely many times in $\mathbf i$ with the property that $\mathcal L=\mathcal L_\mathbf i$.
\end{Prop}
\begin{proof}
    The rough idea is as follows. We work directly with the minimal sets in $\mathcal L$. Since $\mathcal L$ is assumed to satisfy condition (iv) above, all such minimal sets are finite, so there are at most countably many of them. Because $\mathcal L$ satisfies condition (ii) above, there are infinitely many of them, so there are precisely countably many. Condition (iii) combined with upward closure demonstrate that $\mathcal L$ is characterized by its finite minimal sets. We will find an ``irreducible'' finite word corresponding to each of these finite minimal sets and string them together in sequence to form $\mathbf i$. Condition (i) is used to ensure no ``collisions'' occur when stringing together these irreducible words.

    Let $\mathbf j=(j_1,\dots,j_\ell)\in\NN^{<\NN}$. We say $\mathbf j$ is \emph{irreducible} if (1) $\mathbf j$ is reduced, (2) $\mathbf j\neq\varnothing$, and (3) for all $j\in\NN$ appearing in $\mathbf j$, the sequence $\mathbf j'=\mathbf j_{\NN\setminus\{j\}}$ reduces to the empty sequence, i.e. $R^\omega(\mathbf j')=\varnothing$. Note that this entails that the set of $j\in\NN$ appearing in $\mathbf j$ is the unique member of $\mathcal L_\mathbf j$. Let $\mathbf j^\text{rev}=(j_\ell,\dots,j_1)$ be the reverse of $\mathbf j$.

    Let $\mathbf j$ be irreducible. We claim that for any $i\in\NN$ not appearing in $\mathbf j$, the concatenation $\hat{\mathbf j}:=(\mathbf j,i,\mathbf j^\text{rev},i)$ is irreducible. Indeed, removing $i$ yields $(\mathbf j,\mathbf j^\text{rev})$, which clearly reduces to the empty sequence. Removing $j\neq i$ appearing in $\mathbf j$ yields $\hat{\mathbf j'}=(\mathbf j',i,(\mathbf j')^\text{rev},i)$ for $\mathbf j'=\mathbf j_{\NN\setminus\{j\}}$. Reducing this yields
    \begin{align*}
        R^\omega(\hat{\mathbf j'})&=R^\omega(\mathbf j',i,(\mathbf j')^\text{rev},i)\\
        &=R^\omega(R^\omega(\mathbf j'),i,R^\omega((\mathbf j')^\text{rev}),i)\\
        &=R^\omega(i,i)\\
        &=\varnothing,
    \end{align*}
    where $R^\omega(\mathbf j')=\varnothing$ by the irreducibility of $\mathbf j$. Altogether, this gives us a method for constructing irreducible strings using strictly larger (finite) numbers of vertices $j\in\NN$.

    As a quick illustration of this, we recursively build an irreducible string $\mathbf j$ using the set $\{1,2,5,8\}\subseteq\NN$. We pick a starting irreducible word of length one, say $\mathbf j_1:=(5)$. We use the above construction with $i=2$ to obtain $\mathbf j_2:=(5,2,5,2)$. Proceed again with $i=8$ to reach $\mathbf j_3:=(5,2,5,2,8,2,5,2,5,8)$. Lastly, we use $i=1$ to land upon
    $$\mathbf j=\mathbf j_4:=(5,2,5,2,8,2,5,2,5,8,1,8,5,2,5,2,8,2,5,2,5).$$
    Of course, the order we chose to include the elements of $\{1,2,5,8\}$ was arbitrary; many different irreducible words may be created using this set of natural numbers. The reader is encouraged to check that this string is indeed irreducible.

    We now proceed with our construction of $\mathbf i$ with linking family $\mathcal L$. Since there are only countably many minimal sets $S\in\mathcal L$ by our observation above, we enumerate them $S_1,S_2,S_3,\dots$. Na\"ively, we would simply construct irreducible words $\mathbf j(S_n)$ for all $n\geq1$ with supports $S_n$ as above and concatenate them end-to-end to obtain the sequence $(\mathbf j(S_1),\mathbf j(S_2),\dots)$. Since $\mathcal L$ satisfies conclusion (i) of Proposition~\ref{linking-prop} for singletons, there are only finitely many $S_n$ containing any given $i\in\NN$, and therefore only finitely many $\mathbf j(S_n)$ containing any given $i\in\NN$. This yields that no $i\in\NN$ appears infinitely many times in the constructed sequence. However, we \emph{may} accidentally create an unreduced sequence this way; it is possible the string $\mathbf j(S_n)$ ends with the same $i\in\NN$ that $j(S_{n+1})$ begins with. Thus, we need to be slightly more careful.

    We string together irreducible words $\mathbf j(S_n)$ constructed recursively according to the above-described procedure as follows. At stage $s=1$, let $\mathbf i_1=\mathbf j(S_1)$. Remove $S_1$ from the above list $S_1,S_2,\dots$. Now at stage $s+1$ for some $s\geq1$, assume $\mathbf i_s$ is constructed and that the remaining list of minimal sets is $S_{k_1},S_{k_2},\dots$ (with $k_1<k_2<k_3<\dots$ in $\NN$). Declare $\mathbf i_{s+1}:=(\mathbf i_s,\mathbf j(S_{k_t}))$ for $t\geq1$ the least such that $\mathbf j(S_{k_t})$ begins with a different number than the one $\mathbf i_s$ ends with. Note that such a $t\geq1$ exists by the fact that $\mathcal L$ satisfies conclusion (i) of Proposition~\ref{linking-prop} for singletons $S\subseteq\NN$. Take $S_{k_t}$ out of the list above and proceed to the next stage. We let $\mathbf i:=\bigcup_{s\geq1}\mathbf i_s$ be the union of the strings $\mathbf i_s$. By construction, $\mathbf i$ is reduced.

    We claim that every minimal set $S_k$ is eventually used to create part of $\mathbf i_s$ for some $s\geq1$. Indeed, if $S_k$ is unused at all stages $s\geq1$, this implies that there exist infinitely many $S_{k'}$ such that $\mathbf j(S_{k'})$ ends with the same number that $\mathbf j(S_k)$ begins with. This cannot occur by the fact that there are only finitely many $S_{k'}$ with $S_{k'}\cap S_k\neq\varnothing$ once again due to the fact that $\mathcal L$ satisfies conclusion (i) of Proposition~\ref{linking-prop}.

    It remains to show that $\mathcal L_\mathbf i=\mathcal L$. To show the reverse inclusion, suppose $S\in\mathcal L$ is minimal in $\mathcal L$. We claim that $S$ is linked with respect to $\mathbf i$. Since the minimal sets of $\mathcal L$ form an antichain in the poset $(\mathcal P(\NN),\subseteq)$, we know that for any other minimal $S'\in\mathcal L$ it holds that $S\cap S'\subsetneq S'$. In particular, $R^\omega(\mathbf j(S')_S)=R^\omega(\mathbf j(S')_{S\cap S'})=\varnothing$ by the irreducibility of $\mathbf j(S')$. Examining $\mathbf i=(\mathbf j(S_{\pi_1}),\mathbf j(S_{\pi_2}),\dots)$ (for some permutation $\pi$ of $\NN$), we make the following computation:
    \begin{align*}
    R^\omega(\mathbf i_S)&=R^\omega(\mathbf j(S_{\pi_1})_S,\mathbf j(S_{\pi_2})_S,\dots)\\
    &=R^\omega(R^\omega(\mathbf j(S_{\pi_1})_S),R^\omega(\mathbf j(S_{\pi_2})_S),\dots)\\
    &=R^\omega(R^\omega(\mathbf j(S)_S))\\
    &=R^\omega(\mathbf j(S))\\
    &=\mathbf j(S)\\
    &\neq\varnothing.
    \end{align*}
    Thus, we have that $S$ is linked with respect to $\mathbf i$.

    To show the forward inclusion, suppose that $S\not\in\mathcal L$ is finite. Then for all $S'\in\mathcal L$ minimal, we have that $S\cap S'\subsetneq S'$. A similar computation to the above shows that $R^\omega(\mathbf i_S)=\varnothing$, so $S$ is unlinked with respect to $\mathbf i$. This completes our verification that the construction works as intended.
\end{proof}

We lastly show that every $\mathbf i\in\NN^\NN$ which is reduced and is such that no $i\in\NN$ appears infinitely often in $\mathbf i$ has a certain kind of ``normal form'' obtained as $R^\omega(\mathbf i_S)$ for some infinite $S\subseteq\NN$, which will be helpful in the proof of Theorem~\ref{LF-inf-mut-seq-thm}.

\begin{Lem}\label{inf-seq-normal-form-lemma}
    Let $\mathbf i\in\NN^\NN$ be reduced and such that no $i\in\NN$ appears infinitely often in $\mathbf i$. Then there exists an infinite subset $S\subseteq\NN$ so that $R^\omega(\mathbf i_S)=(\mathbf j_1,\mathbf j_2,\dots)$ for some finite reduced sequences $\mathbf j_k\in\NN^{<\NN}$ with:
    \begin{enumerate}[label=(\roman*)]
        \item $\mathbf j_k\neq\varnothing$, that is, $\mathbf j_k$ is nontrivial.
        \item For any two $k_1\neq k_2$, the set of $i\in\NN$ appearing in $\mathbf j_{k_1}$ is disjoint from the set of $i\in\NN$ appearing in $\mathbf j_{k_2}$.
    \end{enumerate}
\end{Lem}
    \begin{proof}
        Let $\mathbf i\in\NN^\NN$ be reduced with no $i\in\NN$ appearing infinitely often in $\mathbf i$. By part (ii) of Proposition~\ref{linking-prop}, there exist infinitely many finite minimal linked sets in $\mathcal L_\mathbf i$. Furthermore, for any finite $i\in\NN$, there are only finitely many minimal members of $\mathcal L_\mathbf i$ containing $\{i\}$ by part (i) of Proposition~\ref{linking-prop}. We may use these observations to construct an infinite family $\{S_n\}_{n\geq1}$ of pairwise disjoint, finite linked sets in $\mathcal L_\mathbf i$. Indeed, fix $S_1\in\mathcal L_\mathbf i$ to be any finite minimal linked set. By the infinitude of minimal linked sets in $\mathcal L_\mathbf i$ and the fact that each $i\in\NN$ appears in only finitely many minimal linked sets, we may pick another finite minimal linked set $S_2\in\mathcal L_\mathbf i$ so that the first entry of $\mathbf i$ containing a member of $S_2$ occurs after the last entry of $\mathbf i$ containing a member of $S_1$. Repeat this process to generate a family $\{S_n\}_{n\geq1}$ of pairwise disjoint, finite, minimal linked sets in $\mathcal L_\mathbf i$, and let $S:=\bigcup_{n\geq1}S_n$.

        We claim that the set $S\subseteq\NN$ we have just constructed satisfies the desired conditions. We will see this by showing that $R^\omega(\mathbf i_S)=(R^\omega(\mathbf i_{S_1}),R^\omega(\mathbf i_{S_2}),\dots)$, since the two requisite properties of the segments $\mathbf j_k=R^\omega(\mathbf i_{S_k})$ follow immediately by the fact that the sets $S_k$ are linked with respect to $\mathbf i$ and pairwise disjoint. Since the sets $S_k$ are pairwise disjoint, the blocks created during every step of the reduction algorithm in Definition~\ref{reduction-def} of $\mathbf i_S$ never ``cross between the $S_k$,'' i.e., there is never a stage $s$ at which there exists an $i$-block $B$ of $R^s(\mathbf i_S)$ simultaneously containing some $i_{p_1}\in S_{k_1}$ and some $i_{p_2}\in S_{k_2}$ for $k_1\neq k_2$. Indeed, this would imply $i_{p_1}=i=i_{p_2}$, a contradiction with $S_{k_1}\cap S_{k_2}=\varnothing$. Because of this, the reduction algorithm can be independently run for each $\mathbf i_{S_k}$, yielding $R^\omega(\mathbf i_S)=(R^\omega(\mathbf i_{S_1}),R^\omega(\mathbf i_{S_2}),\dots)$ as hoped.
    \end{proof}

%% c-vector stabilization phenomenon %%
\subsection{A c-vector stabilization phenomenon} Having dealt with the \emph{mutation sequence} combinatorics needed for the proof of Theorem~\ref{LF-inf-mut-seq-thm}, we now turn to the \emph{quiver} side of what is needed. Since our goal is to work with arbitrary mutation sequences, we will need to be able to find quivers with desirable properties after applying these arbitrary mutation sequences.

We begin with definitions of certain kinds of quivers.

\begin{Def}\label{abundant-def}
    Let $Q:X\times X\to\ZZ$ be an arrow finite quiver without loops or oriented 2-cycles. We say that $Q$ is \emph{abundant (or 2-complete)} if for all $x,y\in X$, it holds that $|Q(x,y)|\geq 2$.
\end{Def}

\begin{Def}[\cite{warkentin_exchange_2014}, Definition 2.1]\label{fork-def}
    Let $Q:X\times X\to\ZZ$ be an arrow finite quiver without loops or oriented 2-cycles and let $r\in X$ be a vertex of $Q$. Assume further that $Q$ is abundant that $Q$ is not acyclic (as a consequence of this, we have $|X|\geq3$). Write $Q^-(r):=\{i\in X\mid Q(i,r)>0\}\subseteq X$ and $Q^+(r):=\{j\in X\mid Q(r,j)>0\}\subseteq X$. We say that $Q$ is a \emph{fork with point of return $r$} if $Q$ satisfies the following:
    \begin{enumerate}[label=(F\arabic*)]
        \item For all $i\in Q^-(r)$ and all $j\in Q^+(r)$, it holds that $Q(j,i)>\max\{Q(i,r),Q(r,j)\}$.
        \item The full subquivers of $Q$ induced by $Q^-(r)$ and $Q^+(r)$ are both acyclic. Equivalently, by the previous property and the abundance of $Q$, the full subquiver of $Q$ induced by $X\setminus\{r\}$ is acyclic.
    \end{enumerate}
    We note that the point of return of a fork is necessarily unique.
\end{Def}

Forks were introduced in M. Warkentin's dissertation in 2014 \cite{warkentin_exchange_2014}, and have since been used by a number of authors to establish structural properties about quivers, their mutations, and their mutation exchange graphs. One very critical fact about forks is the following:

\begin{Lem}[\cite{warkentin_exchange_2014}, Lemmas 2.5, 2.8]\label{warkentin-fork-lemma}
    Let $Q:X\times X\to\ZZ$ be a fork or an abundant acyclic quiver and suppose $n:=|X|<\infty$. Suppose $k\in X$ is not the point of return of $Q$ if $Q$ is a fork or that it is not the (unique) source or the (unique) sink of $Q$ if $Q$ is abundant acyclic. Then $\mu_k(Q)$ is a fork with point of return $k$ and the total number of arrows in $\mu_k(Q)$ is strictly greater than the total number of arrows in $Q$.

    Moreover, in the case that $Q$ is a fork with point of return $r\in X$, let $\Gamma$ be the undirected graph whose vertices are labeled quivers mutation-equivalent to $Q$ and whose edges represent single mutations between these quivers (we will refer to $\Gamma$ as the \emph{mutation exchange graph} of $Q$). If $e$ is the edge of $\Gamma$ connecting $Q$ and $\mu_r(Q)$, then the connected component of $\Gamma\setminus e$ containing $Q$ is an infinite complete $(n-1)$-ary rooted tree with root $Q$.
\end{Lem}

This lemma has the important consequence that if $F$ is a fork with $n\geq3$ vertices with point of return $r$ and $\mathbf i=(i_1,i_2,\dots,i_\ell)$ is any finite reduced mutation sequence on $F$ with $i_1\neq r$, then for any $1\leq k\leq\ell$ the quiver $\mu_{(i_1,\dots,i_k)}(F)$ is a fork with point of return $i_k$. Such a sequence $\mathbf i$ was dubbed \emph{fork-preserving} in \cite{ervin_geometry_2024}.

The following order on the vertices of an abundant acyclic quiver will prove useful for us. We will also use special kinds of mutation sequences on abundant acyclic quivers, which we introduce now as well.

\begin{Def}\label{acyc-order-notation-def}
    Let $Q$ be an abundant acyclic quiver on a vertex set $X$. By the fact that $Q$ is acyclic and the underlying undirected graph of $Q$ is complete, the orientations of the arrows of $Q$ determine a unique linear order $\prec$ on $X$ by setting $x\prec y$ for $x,y\in X$ if and only if $Q(x,y)>0$. We call this the \emph{(unique) acyclic order of $Q$}.

    Suppose $Q$ is an abundant acyclic quiver with acyclic order $v_n\prec v_{n-1}\prec\dots\prec v_1$ and $\mathbf i$ is a reduced sequence on $\{v_1,\dots,v_n\}$. We say that $\mathbf i$ is \emph{triangular} with respect to $Q$ if, for any $1\leq i,j\leq n$ so that the first appearance of $v_i$ in $\mathbf i$ precedes the first appearance of $v_j$ in $\mathbf i$, it holds that $i<j$. We say $\mathbf i$ is \emph{strongly triangular} with respect to $Q$ if each vertex $v_j$ of $Q$ is used at least once in $\mathbf i$. In particular, any mutation sequence on $Q$ consisting only of mutations at sinks is triangular, while no mutation sequence consisting only of mutations at sources is (unless $n=1$).
\end{Def}

An important fact about the acyclic order on the subquiver of a fork not containing its point of return is the following:

\begin{Lem}[\cite{ervin_unrestricted_2024}, Lemma 3.4]\label{ervin-acyc-order-lemma}
    Let $F$ be a fork with point of return $r$ on a finite vertex set $X$ of size $n$. Let $v_1\prec v_2\prec\dots\prec v_{n-1}$ be the unique acyclic order on $F\setminus\{r\}$. Fix some $1\leq j\leq n$ and let $F'=\mu_{v_j}(F)$. Then $F'\setminus\{v_j\}$ has unique acyclic ordering
    $$r\prec v_1\prec\dots\prec v_{j-1}\prec v_{j+1}\prec\dots\prec v_{n-1}$$
    if $F(r,v_j)>0$. In this case, it holds that $(F')^-(v_j)=\{v_{j+1},v_{j+2},\dots,v_{n-1}\}$ and $(F')^+(v_j)=\{r,v_1,v_2,\dots,v_{j-1}\}$.
    
    Otherwise, in the case that $F(v_j,r)>0$, it holds that $F'\setminus\{v_j\}$ has unique acyclic ordering
    $$v_1\prec\dots\prec v_{j-1}\prec v_{j+1}\prec\dots\prec v_{n-1}\prec r.$$
    In this case, $(F')^-(v_j)=\{v_{j+1},v_{j+2},\dots,v_{n-1},r\}$ and $(F')^+(v_j)=\{v_1,v_2,\dots,v_{j-1}\}$.
\end{Lem}

We now introduce a natural way to extend a quiver $Q$ to a larger quiver $\hat{Q}$ with ``frozen'' vertices. This construction and its well-studied properties will allow us to prove a key technical lemma for a construction in the proof of Theorem~\ref{LF-inf-mut-seq-thm}.

\begin{Def}\label{framing-def}
    Let $Q:X\times X\to\ZZ$ be a quiver with vertex set $X$. Let $X'$ denote a copy of the set $X$ disjoint from $X$ itself which consists of an element $x'\in X'$ for every element $x\in X$. The \emph{framing} of $Q$ is the quiver $\hat Q:(X\sqcup X')\times(X\sqcup X')\to\ZZ$ on $X\sqcup X'$ extending $Q:X\times X\to\ZZ$ by adding one arrow from $x$ to $x'$ for all $x\in X$. We refer to the vertices of $\hat Q$ in $X'$ as \emph{frozen} vertices and forbid ourselves from ever mutating at any of them. The vertices in the original vertex set $X$ are said to be \emph{mutable}.

    For any finite mutation sequence $\mu_\mathbf i$ on $Q$ and any $x\in X$, let $c^{\mathbf i}_{x}\in\ZZ^X$ be the integer vector with entries indexed by $X$ whose $y$-th entry is equal to $\mu_\mathbf i(\hat{Q})(x,y')$. We refer to this vector as the \emph{c-vector of $x$ in $\mu_\mathbf i(Q)$}. A mutable vertex $x$ of $\mu_\mathbf i(\hat{Q})$ is said to be \emph{green} if all of the entries of its c-vector are non-negative, which we denote by $c^{\mathbf i}_{x}\geq0$, and it is said to be \emph{red} if all of the entries of its c-vector are non-positive, written $c^{\mathbf i}_{x}\leq0$.
\end{Def}

We make note of the following very famous, very useful result in the theory of cluster algebras about the colors of vertices due to Gross, Hacking, Keel, and Kontsevich \cite{gross_canonical_2014}:

\begin{Thm}[\cite{gross_canonical_2014}, Corollary 5.5]\label{sign-coh-thm}
    Let $Q$ be a finite quiver. Then any mutable vertex of any quiver mutation-equivalent to $\hat Q$ is either green or red (and never both).
\end{Thm}

Ultimately, we will work to understand the colors of vertices in a quiver obtained from an abundant acyclic quiver by (strongly) triangular mutation sequences. The following is a result of Ervin \cite{ervin_unrestricted_2024} specialized for our purposes.

\begin{Lem}[\cite{ervin_unrestricted_2024}, Proposition 5.2, Lemma 5.4]\label{abundant-acyc-colors-lemma}
    Let $Q$ be an abundant acyclic quiver with sink $v$. If $\mathbf i$ is a reduced mutation sequence on $Q$ beginning with $v$ and ending with some vertex $r$ (possibly $v$ itself), then the following hold:
    \begin{enumerate}[label=(\roman*)]
        \item If there are arrows in $\mu_\mathbf i(Q)$ from $r$ to a green vertex $j\neq r$, then either $r$ is green in $\mu_\mathbf i(Q)$ or $c^\mathbf i_{j,k}\geq-c^\mathbf i_{r,k}$ for all vertices $k$ of $Q$. Dually, if there are arrows in $\mu_\mathbf i(Q)$ from a red vertex $j\neq r$ to $r$, then either $r$ is red in $\mu_\mathbf i(Q)$ or $-c^\mathbf i_{j,k}\geq c^\mathbf i_{r,k}$ for all vertices $k$ of $Q$.
        \item If there are arrows in $\mu_\mathbf i(Q)$ from $r$ to a red vertex $j\neq r$, then $r$ is green in $\mu_\mathbf i(Q)$. Dually, if there are arrows in $\mu_\mathbf i(Q)$ from a green vertex $j\neq r$ to $r$, then $r$ is red in $\mu_\mathbf i(Q)$.
        \item Suppose $i\neq j$ are two vertices in $Q$ different from $r$ and that there are arrows from $i$ to $j$ in $\mu_\mathbf i(Q)$. Equivalently, $i\prec j$ in the unique acyclic order on $\mu_\mathbf i(Q)\setminus\{r\}$. If $j$ is green in $\mu_\mathbf i(Q)$, then so is $i$. Dually, if $i$ is red in $\mu_\mathbf i(Q)$, then so is $j$. Alternatively, it holds that the set of green vertices in $\mu_\mathbf i(Q)\setminus\{r\}$ is downward closed under $\prec$. We say that the $\prec$-greatest green vertex in $\mu_\mathbf i(Q)\setminus\{r\}$ is, if it exists, the \emph{last green vertex} of $\mu_\mathbf i(Q)$.
    \end{enumerate}
\end{Lem}

We are now ready to show the following ``c-vector stabilization'' phenomenon for abundant acyclic quivers, which we believe may be of independent interest. We suspect that several of the assumptions made here can be lifted with some care, but we will be able to make do with the assumptions we use here.

\begin{Prop}\label{c-vector-stab-prop}
    Let $Q$ be an abundant acyclic quiver with vertex set $[n+1]$ and acyclic ordering $n+1\prec n\prec\dots\prec 1$. Let $\mathbf i$ be a triangular mutation sequence on $Q$ such that
    \begin{enumerate}[label=(\roman*)]
        \item Every $i\in[n]$ appears at least once in $\mathbf i$.
        \item $n+1$ does \emph{not} appear in $\mathbf i$.
    \end{enumerate}
    In other words, $\mathbf i$ is a strongly triangular mutation sequence on $Q\setminus\{n+1\}$. Then the c-vector $c^\mathbf i_{n+1}$ of $n+1$ in $\mu_\mathbf i(F)$ does not depend on $\mathbf i$ in the following sense: for any other $\mathbf i'$ also satisfying the above conditions, it holds that $c^\mathbf i_{n+1}=c^{\mathbf i'}_{n+1}$.
\end{Prop}
    \begin{proof}
        We induct on $n\geq 1$. For the base case, suppose $Q$ is an abundant acyclic quiver with vertex set $\{1,2\}$ and with acyclic order $2\prec 1$. Suppose there are exactly $m\geq 2$ arrows from $2$ to $1$ in $Q$. There is only one strongly triangular mutation sequence on the set $\{1\}$, which is the sequence $\mathbf i=(1)$, so the claim automatically follows.

        To proceed, we in fact need to provide a stronger inductive hypothesis than just independence of the $(n+1)^\text{st}$ c-vector of $\mu_\mathbf i(Q)$ on $\mathbf i$; we will add to our inductive hypothesis that for any strongly triangular $\mathbf i$ on $Q\setminus\{n+1\}$, it holds that $n+1$ is the last green vertex of $\mu_\mathbf i(Q)$. The base case of $n=1$ does indeed meet this requirement, since the vertex $1$ is red in $\mu_1(Q)$ while $2$ is green.

        Suppose this strong inductive hypothesis holds for some $n\geq1$. Let $Q$ be an abundant acyclic quiver on $[n+2]$ with acyclic order $n+2\prec n+1\prec\dots\prec 1$ and let $\mathbf i$ be a strongly triangular mutation sequence on the vertices $[n+1]$. Write $\mathbf i=(\mathbf i_1,\mathbf i_2)$ for the unique $\mathbf i_1$ and $\mathbf i_2$ so that $\mathbf i_1$ contains no instances of $n+1$ and $\mathbf i_2$ begins with a mutation at $n+1$. Consider the full subquivers $Q':=Q\setminus\{n+2\}$ and $Q'':=Q\setminus\{n+1\}$ of $Q$. Their unique acyclic orderings are $n+1\prec n\prec\dots\prec 1$ and $n+2\prec n\prec\dots\prec 1$ respectively. Note that, by the triangularity of $\mathbf i$ and the fact that $\mathbf i_1$ contains no instances of $n+1$, we have that $\mathbf i_1$ is strongly triangular on $Q\setminus\{n+1,n+2\}$. By induction, we know two things:
        \begin{enumerate}
            \item The last green vertices of $\mu_{\mathbf i_1}(Q')$ and $\mu_{\mathbf i_1}(Q'')$ are $n+1$ and $n+2$ respectively.
            \item The c-vectors $c^{\mathbf i_1}_{n+1}$ and $c^{\mathbf i_1}_{n+2}$ of $n+1$ and $n+2$ in $\mu_{\mathbf i_1}(Q)$ are independent of the choice of $\mathbf i_1$. To be very precise, the c-vector of $n+1$ in $\mu_{\mathbf i_1}(Q')$ and the c-vector of $n+2$ in $\mu_{\mathbf i_1}(Q'')$ are independent of $\mathbf i_1$, but these vectors are essentially already the c-vectors of $n+1$ and $n+2$ in $Q$; extending the c-vectors of $n+1$ in $\mu_{\mathbf i_1}(Q')$ and of $n+2$ in $\mu_{\mathbf i_1}(Q'')$ to also have $c^{\mathbf i_1}_{n+1,n+2}=c^{\mathbf i_1}_{n+2,n+1}=0$ gives the actual c-vectors of $n+1$ and $n+2$ in $\mu_{\mathbf i_1}(Q)$.
        \end{enumerate}
        We use this to claim that $n+1$ is the last green vertex of $\mu_{\mathbf i_1}(Q)$. Indeed, we know it cannot be $n+2$, since $n+2\prec n+1$ in $Q$ (and consequently in $\mu_{\mathbf i_1}(Q)$ by the fact that sink and source mutations of an abundant acyclic quiver simply shift its acyclic order, combined with Lemma~\ref{ervin-acyc-order-lemma} if $\mu_{\mathbf i_1}(Q)$ ends up being a fork). That $n+1$ is green in the first place follows from the fact that $\mathbf i_1$ does not mutate at $n+1$. Further, letting $r$ be the last mutation in $\mathbf i_1$, if there is some green $j\neq r$ in $\mu_{\mathbf i_1(Q)}$ with $n+1\prec j$, then the same is true in $\mu_{\mathbf i_1}(Q')$, a contradiction to the fact that $n+1$ is the last green vertex in $\mu_{\mathbf i_1}(Q')$. A similar argument shows that there does not exist a $j$ so that $n+2\prec j\prec n+1$ in $\mu_{\mathbf i_1}(Q)\setminus\{r\}$.

        From this, we now wish to show that $n+2$ is the last green vertex in $\mu_{n+1}(\mu_{\mathbf i_1}(Q))$, after which we will then establish that $n+2$ is the last green vertex in $\mu_\mathbf i(Q)$. If $n+1$ is not the sink in $\mu_{\mathbf i_1}(Q)\setminus\{r\}$, then there exists some red $j\succ n+1$ in $\mu_{\mathbf i_1}(Q)\setminus\{r\}$. Since $n+1$ is green and there are arrows from $n+1$ to $j$ in $\mu_{\mathbf i_1}(Q)$, we know that mutation at $n+1$ does not change the c-vector of $j$. That is, $j$ remains red in $\mu_{n+1}(\mu_{\mathbf i_1}(Q))$. Combining Lemma~\ref{ervin-acyc-order-lemma} with Lemma \ref{abundant-acyc-colors-lemma}(iii) shows that $n+2$ is the last green vertex of $\mu_{n+1}(\mu_{\mathbf i_1}(Q))$ in this case.
        
        On the other hand, if $n+1$ is the sink in $\mu_{\mathbf i_1}(Q)\setminus\{r\}$, then we must be careful. In the case that $\mu_{\mathbf i_1}(Q)$ is acyclic, we have that $n+1$ is its sink and $r=n$ is its source; the only way for $\mu_{\mathbf i_1}(Q)$ to be acyclic is with $\mathbf i_1=(1,2,\dots,n)$, since anything different yields a fork by Lemma~\ref{warkentin-fork-lemma}. Then mutation at $n+1$ shows that $\mu_{n+1}(\mu_{\mathbf i_1}(Q))$ is abundant acyclic with unique acyclic ordering $n+1\prec n\dots\prec1\prec n+2$. In particular, $n+2$ is green, so it is the last green vertex of $\mu_{n+1}(\mu_{\mathbf i_1}(Q))$. In the other case, where $\mu_{\mathbf i_1}(Q)$ is a fork, we see that there must be arrows from $n+1$ to $r$ since $n+1$ is the sink in $\mu_{\mathbf i_1}(Q)\setminus\{r\}$ (otherwise $r$ has no predecessors at all, contradicting the non-acyclicity of forks). By Lemma~\ref{ervin-acyc-order-lemma}, we then know that the acyclic order on $\mu_{n+1}(\mu_{\mathbf i_1}(Q))\setminus\{n+1\}$ will end with $\dots\prec n+2\prec r$ since the acyclic order on $\mu_{\mathbf i_1}(Q)\setminus\{r\}$ ends with $\dots\prec n+2\prec n+1$. Therefore, we want to show that $r$ is red in $\mu_{n+1}(\mu_{\mathbf i_1}(Q))$. Since $n+1$ has arrows into $r$ in $\mu_{\mathbf i_1}(Q)$ and $n+1$ is green in $\mu_{\mathbf i_1}(Q)$, it is equivalent to show that $r$ is red in $\mu_{\mathbf i_1}(Q)$. This follows by Lemma~\ref{abundant-acyc-colors-lemma}(ii) applied with the green vertex $n+1\in(\mu_{\mathbf i_1}(Q))^-(r)$. So $r$ is red in $\mu_{n+1}(\mu_{\mathbf i_1}(Q))$, making $n+2$ the last green vertex of $\mu_{n+1}(\mu_{\mathbf i_1}(Q))$.

        Now we may establish that $n+2$ is the last green vertex of $\mu_{\mathbf i}(Q)$. We induct on the length of $\mathbf i_2$; the base case with $|\mathbf i_2|=1$ is precisely what we just did. Suppose the claim holds for all reduced sequences on $[n+1]$ of length $\ell\geq 1$ beginning with $n+1$, and suppose $\mathbf i_2=(\mathbf i_2',k)$ has length $\ell+1$ and ends with some vertex $k\in [n+1]$. Note that since $\mathbf i_2$ has length at least $2$ and $\mathbf i_1$ has length at least $n$, the sequence $\mathbf i$ has length at least $n+2$. By the pigeonhole principle and the fact that we do not mutate at the vertex $n+2$ of $Q$, the mutation sequence $\mathbf i$ is not a sink sequence. As a consequence, $\mu_\mathbf i(Q)$ must be a fork in this case. By induction, $n+2$ is the last green vertex of $\mu_{\mathbf i_2'}(\mu_{\mathbf i_1}(Q))$. If $k\prec n+2$ in the acyclic order on $\mu_{\mathbf i_2'}(\mu_{\mathbf i_1}(Q))\setminus\{k'\}$, where $k'$ is the last mutation of $\mathbf i_2'$, then Lemma~\ref{ervin-acyc-order-lemma} shows that $n+2$ remains the last green vertex of $\mu_{\mathbf i_2}(\mu_{\mathbf i_1}(Q))$, no matter whether $k$ is a predecessor or successor of $k'$ in $\mu_{\mathbf i_2'}(\mu_{\mathbf i_1}(Q))$. On the other hand, if $n+2\prec k$, this same argument applies \emph{except if the $\prec$-order on $\mu_{\mathbf i_2'}(\mu_{\mathbf i_1}(Q))\setminus\{k'\}$ ends with $\dots\prec n+2\prec k$}. If this happens, the vertex $k$ is red by the inductive hypothesis that $n+2$ is the last green vertex of $\mu_{\mathbf i_2'}(\mu_{\mathbf i_1}(Q))$. Lemma~\ref{abundant-acyc-colors-lemma}(i) and the fact that $k$ is a predecessor of $k'$ in $\mu_{\mathbf i_2'}(\mu_{\mathbf i_1}(Q))$, either $k'$ is itself red in $\mu_{\mathbf i_2'}(\mu_{\mathbf i_1}(Q))$ (in which case mutation at $k$ will not change this) or $-c^{(\mathbf i_1,\mathbf i_2')}_{k,j}\geq c^{(\mathbf i_1,\mathbf i_2')}_{k',j}$ for all $j$. In the latter case, letting $m\geq 2$ be the number of arrows from $k$ to $k'$ in $\mu_{\mathbf i_2'}(\mu_{\mathbf i_1}(Q))$ and unpacking the algebraic form of the definition of mutation at $k$ (Definition~\ref{quiv-mut-def}) show that the new c-vector for $k'$ in $\mu_\mathbf i(Q)=\mu_{\mathbf i_2}(\mu_{\mathbf i_1}(Q))$ satisfies
        $$c^\mathbf i_{k',j}=c^{(\mathbf i_1,\mathbf i_2')}_{k',j}+mc^{(\mathbf i_1,\mathbf i_2')}_{k,j}\leq(m-1)c^{(\mathbf i_1,\mathbf i_2')}_{k,j}\leq 0$$
        for all vertices $j$ (where the last equality follows by the fact that $k$ is red in $\mu_{\mathbf i_2'}(\mu_{\mathbf i_1}(Q))$). Either way, $k'$ is red in $\mu_\mathbf i(Q)$, showing that $n+2$ is the last green vertex of $\mu_\mathbf i(Q)$ as desired.

        Knowing now that $n+2$ is the last green vertex of $\mu_\mathbf i(Q)$, we will see that the independence of $c^\mathbf i_{n+2}$ from $\mathbf i$ is rather immediate. Indeed, we know by induction that $c^{\mathbf i_1}_{n+2}$ is independent of the strongly triangular sequence $\mathbf i_1$ on $[n]$; as long as the number of arrows between $n+1$ and $n+2$ does not change as we apply the mutation sequence $\mathbf i_1$, then $c^{(\mathbf i_1,n+1)}_{n+2}$ is also independent of $\mathbf i_1$. This number of arrows cannot change, since every $i\in[n]$ either has arrows into both of $n+1$ and $n+2$ or out of both of $n+1$ and $n+2$, as this is true in $Q$, in any acyclic quivers we may see along the sequence $\mathbf i_1$, and also in any forks we see along the way by Lemma~\ref{ervin-acyc-order-lemma}. So $c^{(\mathbf i_1,n+1)}_{n+2}$ is independent of $\mathbf i_1$ as well. Finally, by the fact that $n+2$ is the last green vertex in $\mu_\mathbf i(Q)$, we see that no mutable vertex in $\mu_\mathbf i(\hat Q)$ will be in the middle of a path of length two from a frozen vertex to $n+2$ or vice versa (except for possibly the last vertex $k$ appearing in $\mathbf i$, but the reducedness condition prevents us from mutating there twice in a row). Indeed, vertices in $[n+1]\setminus\{k\}$ with arrows to frozen vertices are green, and so also carry arrows to $n+2$ by the fact that $n+2$ is the last green vertex of $\mu_\mathbf i(Q)$. Similarly, vertices in $[n+1]\setminus\{k\}$ with arrows from frozen vertices are red, and so also carry arrows from $n+2$ (note that we only need $n+2$ to be green for this part, not necessarily the \emph{last} green vertex).

        This completes the proof of the claim.
    \end{proof}

To demonstrate that this phenomenon is not trivially true (since, a priori, it may happen that the c-vector of $n+1$ cannot change \emph{at all}), we provide the following example.

\begin{Ex}\label{c-vec-stability-ex}
    For each $n\geq1$, let $Q_n$ be the abundant acyclic quiver on the vertices $[n+1]$ with exactly $2$ arrows from $i$ to $j$ for all $i>j$ in $[n+1]$, and let $\mathbf i$ be a strongly triangular mutation sequence for $Q_n\setminus\{n+1\}$ (which is equal to the quiver $Q_{n-1}$ if $n>1$). Then the c-vector of $n+1$ in $\mu_\mathbf i(Q_n)$ has the following entries:
    $$c^\mathbf i_{n+1,i}=\begin{cases}
        2\cdot3^{n-i} & \text{if }i\leq n\\
        1 & \text{if }i=n+1.
    \end{cases}$$
    By Proposition~\ref{c-vector-stab-prop} above, it suffices to show this for the sequence $\mathbf i=(1,2,\dots,n)$. We induct on $n$ to do so. In the base case of $n=1$, we see that $Q_1$ is the two-vertex quiver with two arrows from $2$ to $1$. Mutating $\hat{Q}_1$ at $1$ shows that $c^{(1)}_{2,1}=2=2\cdot3^{1-1}$ and $c^{(1)}_{2,2}=1$ as claimed. Now suppose we have shown the above to be true for some $n\geq1$, and consider the quiver $Q_{n+1}$. By induction and the fact that $Q_{n+1}\setminus\{n+2\}=Q_n$, we know that the c-vector $c^{(1,2,\dots,n)}_{n+1}$ takes the following entries:
    $$c^{(1,2,\dots,n)}_{n+1,i}=\begin{cases}
        2\cdot3^{n-i} & \text{if }i\leq n\\
        1 & \text{if }i=n+1\\
        0 & \text{if }i=n+2.
    \end{cases}$$
    Similarly, we may use induction and the fact that $Q_{n+1}\setminus\{n+1\}\cong Q_n$ to see that the c-vector $c^{(1,2,\dots,n)}_{n+2}$ takes the following entries:
    $$c^{(1,2,\dots,n)}_{n+2,i}=\begin{cases}
        2\cdot3^{n-i} & \text{if }i\leq n\\
        0 & \text{if }i=n+1\\
        1 & \text{if }i=n+2.
    \end{cases}$$
    Mutating $\hat{Q}_{n+1}$ at $n+1$ once (and using the fact that there are still exactly two arrows from $n+2$ to $n+1$ in $\mu_{(1,\dots,n)}(\hat{Q}_{n+1})$) shows us that:
    \begin{align*}
    c^{(1,2,\dots,n,n+1)}_{n+2,i}&=\begin{cases}
        2\cdot3^{n-i}+2\cdot2\cdot3^{n-i} & \text{if }i\leq n\\
        0+2\cdot1 & \text{if }i=n+1\\
        1+2\cdot0 & \text{if }i=n+2
    \end{cases}\\
    &=\begin{cases}
        2\cdot3^{n+1-i} & \text{if }i\leq n\\
        2 & \text{if }i=n+1\\
        1 & \text{if }i=n+2
    \end{cases}\\
    &=\begin{cases}
        2\cdot3^{n+1-i} & \text{if }i\leq n+1\\
        1 & \text{if }i=n+2,
    \end{cases}
    \end{align*}
    completing the induction.
\end{Ex}

As an immediate corollary of Proposition~\ref{c-vector-stab-prop} and Example~\ref{c-vec-stability-ex} above, we have the following immediate corollary that will appear in the proof of Theorem~\ref{LF-inf-mut-seq-thm}:

\begin{Cor}\label{c-vec-off-diag-cor}
    Let $\mathbf i$ be a nontrivial, reduced mutation sequence on $[n]$ for $n\geq1$. Then there exists a quiver $Q$ on $[n+1]$ and some $i\in[n]$ so that the c-vector of $n+1$ in $\mu_\mathbf i(Q)$ satisfies $c^\mathbf i_{n+1,i}>0$.
\end{Cor}
    \begin{proof}
        Take $Q$ to be $Q_n$ as in Example~\ref{c-vec-stability-ex}, relabeling the vertices of $Q_n$ as needed to ensure that $\mathbf i$ is triangular (while keeping $n+1$ labeled $n+1$). Then $c^\mathbf i_{n+1,i}=2\cdot3^{n-\sigma(i)}>0$ for some permutation $\sigma$ of $[n]$ and all $i\neq n+1$.
    \end{proof}

%% LF infinite sequence density theorem %%
\subsection{Proof of Theorem~\ref{LF-inf-mut-seq-thm}} We are finally ready to state and prove the density theorem for infinite mutation sequences in $\LF$.

\begin{T-LF-inf-seq}
    Let $\mu_{\mathbf i}:=\ldots\mu_{i_3}\mu_{i_2}\mu_{i_1}$ be an infinite mutation sequence on $\LF$. Exactly one of the following possibilities occurs:
        \begin{enumerate}[label=(\roman*)]
            \item $\mathcal C_\LF(\mu_\mathbf i)=\LF$ (and therefore $\mathcal D_\LF(\mu_\mathbf i)=\varnothing$). This occurs precisely when every vertex appearing in $\mathbf i$ appears only finitely many times and the reduction $\overline\mu_\mathbf i$ is finite.
            \item $\mathcal C_\LF(\mu_\mathbf i)$ and $\mathcal D_\LF(\mu_\mathbf i)$ are both dense in $\LF$. This occurs precisely when every vertex appearing in $\mathbf i$ appears only finitely many times and the reduction $\overline\mu_\mathbf i$ is infinite.
            \item $\mathcal C_\LF(\mu_\mathbf i)$ is not dense. This happens precisely when the set $S_\infty\subseteq\NN$ of vertices appearing infinitely many times in $\mathbf i$ is nonempty. In this case, $\mathcal D_\LF(\mu_\mathbf i)$ is dense in $\LF$ if and only if one (or both) of the following occurs:
            \begin{enumerate}
                \item $S_\infty\subseteq\NN$ is infinite.
                \item $\overline\mu_{\mathbf i_{\NN\setminus S_\infty}}$ is infinite, where $\mathbf i_{\NN\setminus S_\infty}$ is the subsequence of $\mathbf i$ consisting only of those $i\in\NN$ that do not appear infinitely many times in $\mathbf i$.
            \end{enumerate}
        \end{enumerate}
\end{T-LF-inf-seq}

\begin{proof}
    %% case (i) %%
    Let $\mu_\mathbf i$ be an infinite mutation sequence with $\mathbf i=(i_1,i_2,\dots)\in\NN^\NN$. Suppose first that each vertex in $\mathbf i$ appears only finitely many times. We consider the two cases in which $\overline\mu_\mathbf i$ is either finite or infinite.

    Assume that $\overline\mu_\mathbf i$ is finite. We claim in part (i) that, under the present assumptions, $\mathcal C_\LF(\mu_\mathbf i)=\LF$. More specifically, we will see that for any $Q\in\LF$, it holds that $\mu_\mathbf i(Q)=\overline\mu_\mathbf i(Q)$. To prove this claim, we use a trick about the strong topology that we have not yet used: we argue that the overfull subquivers induced by \emph{singletons} all stabilize in the sequence $(\mu_{i_n}\circ\dots\circ\mu_{i_1}(Q))_{n\geq1}$ and then reason that this is sufficient for convergence in $\LF$.
    
    Toward the first part of this trick, we see by Lemma~\ref{reduction-lemma2} that there exists some $n\geq0$ so that $\overline{\mu_{i_n}\circ\dots\circ\mu_{i_1}}=\overline\mu_\mathbf i$ and $\overline{\dots\mu_{i_{n+2}}\circ\mu_{i_{n+1}}}$ is the empty mutation sequence. Letting $n\geq0$ be such an $n$, our claim reduces to showing that $\mu_\mathbf i$ converges on $Q$ in $\LF$ to $\mu_{i_n}\circ\dots\circ\mu_{i_1}(Q)$. By precomposing $\mu_\mathbf i$ with $\mu_{i_1}\circ\dots\circ\mu_{i_n}$ and using the fact that $\overline{\dots\mu_{i_{n+2}}\circ\mu_{i_{n+1}}}$ is the empty mutation sequence, we may assume without loss of generality that $\overline\mu_\mathbf i$ is the empty sequence.
    
    With this assumption, suppose $v\in\NN$ is some vertex of $Q$. Since $Q$ is locally finite, it follows that the overfill of $Q$ induced by the singleton $\{v\}$ is finite as well; let $V\subseteq\NN$ be the support of this overfill. We claim that there is some $n\geq1$ so that for all $k\geq n$, we have $\mu_{i_k}\circ\dots\circ\mu_{i_1}(Q)\in W_{Q,\{v\}}$. By Lemma~\ref{reduction-lemma3}, there exist infinitely many $p\geq1$ so that $\overline{\mu_{i_p}\circ\dots\circ\mu_{i_1}}$ is the empty mutation sequence. Take $n$ to be some such $p$ with the additional property that $i_k>\max V$ for all $k\geq n$. We show that this $n$ is such that $\mu_{i_k}\circ\dots\circ\mu_{i_1}(Q)\in W_{Q,\{v\}}$ for all $k\geq n$. We proceed by induction on $k\geq n$. If $k=n$, then $\mu_{i_n}\circ\dots\circ\mu_{i_1}(Q)=Q\in W_{Q,\{v\}}$ by the fact that $\overline{\mu_{i_n}\circ\dots\circ\mu_{i_1}}$ is the empty mutation sequence. Suppose now that $k>n$ and that the claim holds for $k-1$. Knowing that $\mu_{i_{k-1}}\circ\dots\circ\mu_{i_1}(Q)\in W_{Q,\{v\}}$, assume for the sake of contradiction that $\mu_{i_k}\circ\dots\circ\mu_{i_1}(Q)\not\in W_{Q,\{v\}}$. Then it must be true that $i_k\in V$. Indeed, in order for mutation at $i_k$ to affect the arrows in the overfill of $\mu_{i_{k-1}}\circ\dots\circ\mu_{i_1}(Q)$ induced by $\{v\}$ (which is the same as the overfill of $Q$ induced by $\{v\}$ by hypothesis), either some arrows in this overfill must have changed direction by step (2) of the mutation process in Definition~\ref{quiv-mut-def} or some arrows must have changed multiplicity by step (1). Neither is possible without $i_k\in V$. Of course, this is a contradiction by our choice of $n$ so that $i_k>\max V$ for all $k>n$. Thus, we must have that $\mu_{i_k}\circ\dots\circ\mu_{i_1}(Q)\in W_{Q,\{v\}}$ for all $k>n$. This gives us the first part of our trick.

    For the second part of the trick, note that if the overfull subquivers induced by singletons of the quivers in a sequence $(Q_n)_{n\geq1}$ in $\LF$ stabilize, then we have stabilization of all overfull subquivers induced by finite vertex sets. Indeed, let $V\subseteq\NN$ be finite, and suppose $V=\{v_1,\dots,v_p\}$. Letting $n\geq1$ be the least $n$ such that for all $k>n$ the overfull subquivers of $Q_k$ induced by the singletons $\{v_s\}$ for $s=1,\dots,p$ stabilize gives us the stabilization of the overfull subquivers of $Q_k$ on $V$ as well.

    Altogether, this completes the argument that $\mathcal C_\LF(\mu_\mathbf i)=\LF$.
    \vspace{0.25cm}

    %% case (ii) %%
    Suppose now that $\mu_\mathbf i$ is an infinite mutation sequence such that every $i\in\NN$ appears in $\mathbf i$ only finitely many times and such that the reduction $\overline\mu_\mathbf i$ is infinite as in part (ii). We wish to show that both $\mathcal C_\LF(\mu_\mathbf i)$ and $\mathcal D_\LF(\mu_\mathbf i)$ are dense in $\LF$.
    
    The density of $\mathcal C_\LF(\mu_\mathbf i)$ is straightforward: the dense subspace $\mathbf{Fin}\subseteq\LF$ is contained in the domain $\mathcal C_\LF(\mu_\mathbf i)$ of convergence. Indeed, every $i\in\NN$ appears only finitely many times in $\mathbf i$, and there are only finitely many vertices in the support of any $Q\in\mathbf{Fin}$, rendering cofinitely many entries in $\mathbf i$ inert for any fixed finite $Q$.
    
    The density of $\mathcal D_\LF(\mu_\mathbf i)$ in $\LF$ is more interesting. We first observe that it suffices to show that $\mathcal D_\LF(\overline\mu_\mathbf i)$ is dense. Indeed, if $\overline\mu_{\mathbf i}$ diverges on $Q\in\LF$, so does $\mu_\mathbf i$. This is because $(\mu_{i_{k_\ell}}\circ\dots\circ\mu_{i_{k_1}}(Q))_{\ell\geq1}$ is a subsequence of $(\mu_{i_n}\circ\dots\circ\mu_{i_1}(Q))_{n\geq1}$, where $K=\{k_1,k_2,\dots\}$ is the (infinite) domain of $R^\omega(\mathbf i)$ as in Definition~\ref{reduction-def}. With this in mind, suppose without loss of generality that $\mu_\mathbf i=\overline\mu_\mathbf i$.

    Fix a basic open set $W_{Q,V}\cap\LF\subseteq\LF$ in $\LF$ for some $Q\in\LF$ and $V\subseteq\NN$ finite. Assume as usual that $Q$ is finite with $\text{supp}(Q)\subseteq V$. Let $N=\max V$. By the assumption that no $i\in\NN$ appears infinitely many times in $\mathbf i$, we see that there exists some $n\geq1$ so that for all $k\geq n$, it holds that $i_k>N+1$. We will build a locally finite quiver $A$ on the vertex set $\{N+1,N+2,\dots\}$ on which the infinite (reduced) sequence $\mu_{(i_n,i_{n+1},i_{n+2},\dots)}$ diverges. If we can do this, then placing a copy of $A':=\mu_{i_1}\circ\dots\circ\mu_{i_{n-1}}(A)$ on the isolated vertices $\{N+1,N+2,\dots\}$ of our finite quiver $Q$ yields a quiver $Q'\in W_{Q,V}\cap\LF$ on which $\mu_\mathbf i$ diverges.

    Let $\mathbf i'=(i_n,i_{n+1},i_{n+2},\dots)$. By the assumption that $i_k>N+1$ for all $k\geq n$ and the plan to use just the tail $\mu_{\mathbf i'}$ when building $A$, we have the liberty to use $N+1$ as a frozen vertex. Since $\mathbf i'$ is infinite, reduced, and contains each $i\in\NN$ only finitely many times, we may apply Lemma~\ref{inf-seq-normal-form-lemma} to find some subset $S\subseteq\{N+1,N+2,\dots\}$ so that $R^\omega(\mathbf i'_S)=(\mathbf j_1,\mathbf j_2,\dots)$ for finite, nontrivial, reduced sequences $\mathbf j_k\in\NN^{<\NN}$ pairwise not using any common vertices. We may assume without loss of generality that $N+1\in S$, since $N+1$ does not appear in $\mathbf i'$ at all and therefore does not change the outcome of the lemma. We will in fact use this set $S\subseteq\{N+1,N+2,\dots\}$ as our domain for $A$ and keep every vertex in $\{N+1,N+2,\dots\}\setminus S$ isolated. With this requirement on $A$, it holds that $\mu_{\mathbf i'}$ diverges on $A$ if and only if $\mu_{\mathbf i'_S}$ diverges on $A$. To show the latter, it suffices by an argument we have previously made to show that $\overline\mu_{\mathbf i'_S}=\mu_{R^\omega(\mathbf i'_S)}$ diverges on $A$.

    We are now ready to build $A$. For each sequence $\mathbf j_k$ appearing in the normal form for $R^\omega(\mathbf i'_S)$ above, let $S_k\subseteq S$ be the finite set of vertices appearing in $\mathbf j_k$ and let $n_k:=|S_k|$. Also for each $\mathbf j_k$, use Corollary~\ref{c-vec-off-diag-cor} to pick a quiver $A_k$ with vertices $S_k\sqcup\{\ast_k\}$ (for some temporary formal symbol $\ast_k$ not used in $\mathbf j_k$) so that the c-vector $c^{\mathbf j_k}_{\ast_k}$ of $\ast_k$ in $\mu_{\mathbf j_k}(A_k)$ has a positive entry $c^{\mathbf j_k}_{\ast_k,p_k}>0$ for some $p_k\in S_k$. We glue these quivers $A_k$ together into our quiver $A$ on $S$ by first creating an arrow from $p_1$ to $N+1$ and then identifying $\ast_k$ with $p_{k+1}$ for all $k\geq1$. Importantly, $A$ is a locally finite quiver; this whole construction would be rendered moot if it were not.

    It remains to check that $\mu_{R^\omega(\mathbf i'_S)}$ diverges on $A$. To see this, we will show that the overfull subquiver of $\mu_{(\mathbf j_1,\dots,\mathbf j_k)}(A)$ induced by $\{N+1\}$ fails to stabilize as $k\to\infty$. Then this sequence would be a divergent subsequence of the relevant sequence of finite partial mutations for $\mu_{R^\omega(\mathbf i'_S)}$ applied to $A$, so it would follow that $\mu_{R^\omega(\mathbf i'_S)}$ diverges on $A$. We proceed by induction on $k\geq1$. One way to view the arrow from the vertex $p_1$ of the subquiver $A_1$ to the vertex $N+1$ is by viewing $N+1$ to be the same as the frozen vertex $p_1'$ in the framed quiver $\hat A_1$. Then applying $\mu_{\mathbf j_1}$ to $A_1$ yields that $c^{\mathbf j_1}_{p_2,N+1}>0$ in $\mu_{\mathbf j_1}(A)$, or equivalently, there exist arrows from $p_2$ to $N+1$ in $\mu_{\mathbf j_1}(A)$. Furthermore, since $\mathbf j_1$ does not mutate at any vertices outside of $S_1$ and since any two of the $A_k$ overlap in at most one vertex (which is not itself in $S_1$), none of the other full subquivers $A_k$ for $k\neq 1$ are changed by the mutation sequence $\mu_{\mathbf j_1}$. Additionally, there are also no arrows in $\mu_{\mathbf j_1}(A)$ from any $s\in S\setminus(S_1\cup\{p_2\})$ to $N+1$ for the same reason.

    For the inductive step, suppose we have mutated through $(\mathbf j_1,\dots,\mathbf j_k)$ so far for $k\geq1$, that we have some positive number $m$ of arrows from $p_{k+1}$ to $N+1$, and that there are no arrows in $\mu_{(\mathbf j_1,\dots,\mathbf j_k)}(A)$ from any $s\in S\setminus\left(\{p_{k+1}\}\cup\bigcup_{t=1}^k S_t\right)$ to $N+1$. Then the full subquiver of $A$ induced by $A_{k+1}$ and $N+1$ can \emph{almost} be thought about as the full subquiver of $\hat A_{k+1}$ including all mutable vertices and the frozen vertex $p_{k+1}'$, where $N+1$ plays the role of $p_{k+1}'$, but with some number $m\geq1$ arrows from $p_{k+1}$ to $p_{k+1}'$ instead of necessarily just one arrow. However, since $N+1$ is not present in $\mathbf j_{k+1}$, it follows from how mutation works that the mutation sequence $\mu_{\mathbf j_{k+1}}$ applied to $A_{k+1}\cup\{N+1\}$ yields the same result as if $\mu_{\mathbf j_{k+1}}$ were applied to the full subquiver of $\hat A_{k+1}$ induced by its mutable vertices and the frozen vertex $p_{k+1}'$ and then all counts of arrows incident to $p_{k+2}'$ were scaled by a factor of $m$. Thus, there are exactly $m\cdot c^{(\mathbf j_1,\dots,\mathbf j_k,\mathbf j_{k+1})}_{p_{k+2},p_{k+1}}>0$ arrows from $p_{k+2}$ to $N+1$ in $\mu_{(\mathbf j_1,\dots,\mathbf j_k,\mathbf j_{k+1})}(A)$. For the same reasons given in the base case, there are no arrows in $\mu_{(\mathbf j_1,\dots,\mathbf j_k,\mathbf j_{k+1})}(A)$ from any vertex in $s\in S\setminus\left(\{p_{k+2}\}\cup\bigcup_{t=1}^{k+1} S_t\right)$ to $N+1$, finishing the inductive step.

    Then we see that the overfull subquiver of $\mu_{(\mathbf j_1,\dots,\mathbf j_k)}(A)$ induced by $\{N+1\}$ fails to stabilize as $k\to\infty$. This in turn shows that $\mu_{R^\omega(\mathbf i_S')}$ diverges on $A$. Unpacking further, this shows that $\mu_\mathbf i$ diverges on the quiver $Q'\in W_{Q,V}$ described above, proving part (ii) of the claim.

    For part (iii), suppose there is some vertex $i\in\NN$ appearing infinitely many times in $\mathbf i$. Consider the quiver $Q\in\LF$ with a single arrow $i\to i+1$. Then any $Q'\in W_{Q,\{i,i+1\}}\cap\LF$ has $i$ as a non-isolated vertex. Mutating $Q'$ finitely many times in any direction preserves the fact that $i$ is non-isolated. Combining this with the fact that $i$ appears infinitely often in $\mathbf i$, we see that the overfull subquiver of $\mu_{(i_1,\dots,i_n)}(Q')$ induced by $\{i\}$ fails to stabilize as $n\to\infty$. Thus, we have that $\mu_\mathbf i$ diverges on all $Q'\in W_{Q,\{i,i+1\}}\cap\LF$, so $\mathcal C_\LF(\mu_\mathbf i)\cap W_{Q,\{i,i+1\}}=\varnothing$. Since $W_{Q,\{i,i+1\}}\cap\LF$ is open in $\LF$, this shows that $\mathcal C_\LF(\mu_\mathbf i)$ is not dense in $\LF$.

    Let $\varnothing\neq S_\infty\subseteq\NN$ be the set of vertices appearing infinitely many times in $\mathbf i$. Consider the subspace $\LF^{\NN\setminus S_\infty}\subseteq\LF$ of quivers on $\NN\setminus S_\infty$, thought about as quivers on $\NN$ with only isolated vertices in $S_\infty$. Note that the strong topology on $\LF^{\NN\setminus S_\infty}_{str}$ agrees with the subspace topology from $\LF$ under this identification. If $S_\infty$ is finite, then this subspace is the clopen subspace given by $W_{Q,S_\infty}\cap\LF$, where $Q$ is taken to be the quiver with no arrows. In particular, $\LF\setminus\LF^{\NN\setminus S_\infty}$ is not dense (since it fails to intersect its open complement). On the other hand, if $S_\infty$ is infinite, then we claim that $\LF^{\NN\setminus S_\infty}$ has empty interior, or equivalently that its complement is dense. To see this, let $W_{Q,V}\cap\LF$ be a basic open set in $\LF$ for some $Q\in\LF$, $V\subseteq\NN$ finite. Assume $Q$ is finite and that $\text{supp}(Q)\subseteq V$ as usual. Since $V$ is finite and $S_\infty$ is infinite, there is some $v\in S_\infty\setminus V$ with $v\in S_\infty\setminus V$ as well. Then we may extend $Q$ to a quiver $Q'\in W_{Q,V}\cap\LF$ with an arrow $v\to v+1$. Since membership in $\LF^{\NN\setminus S_\infty}\subseteq\LF$ requires that all vertices in $S_\infty$ are isolated and $v\in S_\infty$ is not isolated in $Q'$, we see that $Q'\not\in\LF^{\NN\setminus S_\infty}$. Thus, the complement of $\LF^{\NN\setminus S_\infty}$ is dense in $\LF$ if $S_\infty$ is infinite.
    
    By the argument given two paragraphs ago, we can see that $\mathcal C_\LF(\mu_\mathbf i)\subseteq\LF^{\NN\setminus S_\infty}$. Thus, if $S_\infty$ is infinite, then the density of $\mathcal D_\LF(\mu_\mathbf i)$ in $\LF$ follows right away by the above argument. Otherwise, if $S_\infty$ is finite, the density of $\mathcal D_\LF(\mu_\mathbf i)$ in $\LF$ is equivalent to its density in the clopen subspace $\LF^{\NN\setminus S_\infty}$. Note that, since all quivers in this subspace have isolated vertices in $S_\infty$, the intersections of $\mathcal D_\LF(\mu_\mathbf i)$ and $\mathcal D_\LF(\mu_{\mathbf i_{\NN\setminus S_\infty}})$ with this subspace $\LF^{\NN\setminus S_\infty}$ agree. Thus, it suffices to show the density of $\mathcal D_\LF(\mu_{\mathbf i_{\NN\setminus S_\infty}})$ in this subspace. Note that $\mathbf i_{\NN\setminus S_\infty}$ does not contain any $i\in\NN$ infinitely often by the definition of $S_\infty\subseteq\NN$. Thus, we may refer to the cases (i) and (ii) that we have already proved to check the density of $\mathcal D_\LF(\mu_{\mathbf i_{\NN\setminus S_\infty}})$ in $\LF^{\NN\setminus S_\infty}_{str}$. By (i) and (ii), it holds that $\mathcal D_\LF(\mu_{\mathbf i_{\NN\setminus S_\infty}})$ is dense in $\LF^{\NN\setminus S_\infty}_{str}$ if and only if $\overline\mu_{\mathbf i_{\NN\setminus S_\infty}}$ is infinite. Therefore, in the case that $S_\infty$ is finite, we see that $\mathcal D_\LF(\mu_\mathbf i)$ is dense in $\LF$ if and only if the same holds.

    This completes the proof.
\end{proof}

%% The Fra\"iss\'e quiver and its mutation classes %%
\section{The Fra\"iss\'e quiver and its mutation classes}\label{sec-6}

This section is devoted to one very specific countably infinite, arrow finite quiver, which we call the \emph{Fra\"iss\'e quiver}. To construct this quiver, we provide a brief overview of the theory of Fra\"iss\'e limits. We provide specialized definitions and state the main structural result in the context of quivers only, but the set-up here extends to the much broader model-theoretic context of structures in first-order languages.

\begin{Def}\label{homog-def}
    A quiver $Q$ is called \emph{homogeneous} if every isomorphism between finite full subquivers of $Q$ extends to an automorphism of $Q$.
\end{Def}

\begin{Rmk}\label{weak-homog-rmk}
    One important consequence of homogeneity is the following ``extension'' property. Given a homogeneous quiver $Q$, two finite quivers $A,B$, and three embeddings of full subquivers $f:A\hookrightarrow Q$, $g:B\hookrightarrow Q$, and $h:A\hookrightarrow B$, it holds that there exists an embedding $h':B\hookrightarrow Q$ such that $h'\circ h=f$. Indeed, one has that there is an isomorphism $\varphi:g(h(A))\to f(A)$ of finite full subquivers of $Q$. By the homogeneity of $Q$, there exists an automorphism $\Phi:Q\to Q$ extending $\varphi$. Then $\Phi$ restricts to precisely the embedding $h':B\hookrightarrow Q$ that we are looking for.
\end{Rmk}

The following definition will allow us to build countable homogeneous quivers out of special classes of finite quivers.

\begin{Def}\label{fraisse-class-def}
    Let $\mathcal K$ be a class of finite quivers. We call $\mathcal K$ a \emph{Fra\"iss\'e class} if it satisfies the following properties:
    \begin{itemize}
        \item $\mathcal K$ is \emph{closed under isomorphism.}
        \item $\mathcal K$ has the \emph{hereditary property (HP)}: if $A\in\mathcal K$ and $B$ is a full subquiver of $A$, then $B$ is in $\mathcal K$ as well.
        \item $\mathcal K$ has the \emph{joint embedding property (JEP):} if $A,B\in\mathcal K$, then there is some $C\in\mathcal K$ such that $A$ and $B$ are isomorphic to full subquivers of $C$.
        \item $\mathcal K$ has the \emph{amalgamation property (AP):} given quivers $A,B,C\in\mathcal K$ and embeddings $f:A\hookrightarrow B$, $g:A\hookrightarrow C$, there exists $D\in\mathcal K$ and embeddings $f':B\hookrightarrow D$, $g':C\hookrightarrow D$ with the property that $f'\circ f=g'\circ g$.
    \end{itemize}
\end{Def}

It turns out that Fra\"iss\'e classes provide a complete characterization of countably infinite homogeneous quivers. We state the following result without proof, as it is a specialization of a well-known theorem in model theory.

\begin{Prop}[\cite{schlicht_lecture_nodate}, Theorem 2.2.4, Fra\"iss\'e's Theorem for quivers]\label{fraisse-thm}
    For any quiver $Q$, let $\text{Age}(Q)$ be the class of finite quivers which embed into $Q$. If $K$ is a Fra\"iss\'e class of quivers, then there exists a unique (up to isomorphism) countable homogeneous quiver $Q_\mathcal K$ with $\text{Age}(Q_\mathcal K)=\mathcal K$. The quiver $Q_\mathcal K$ is called the \emph{Fra\"iss\'e limit} of $\mathcal K$.
    
    Conversely, if $Q$ is infinite and homogeneous, then $\text{Age}(Q)$ is a Fra\"iss\'e class.
\end{Prop}

\begin{Def}\label{fraisse-quiver-def}
    Let $\mathcal Q$ be the Fra\"iss\'e class of all finite quivers (the reader should verify using Definition~\ref{fraisse-class-def} that $\mathcal Q$ is indeed a Fra\"iss\'e class). If $F$ is a Fra\"iss\'e limit of $\mathcal Q$, we call $F$ a \emph{Fra\"iss\'e quiver.} By Fra\"iss\'e's Theorem above, the isomorphism type of $F$ is uniquely determined.
\end{Def}

We are now ready to state and prove a general result about the mutation classes of Fra\"iss\'e limits of certain Fra\"iss\'e classes.

\begin{T-Fraisse-singleton}
    Let $\mathcal K$ be a Fra\"iss\'e class of quivers which is additionally closed under mutation. Then, up to identifying isomorphic quivers, the (finite-)mutation-equivalence class of its Fra\"iss\'e limit $Q_\mathcal K$ is a singleton.
\end{T-Fraisse-singleton}
    \begin{proof}
        We proceed by a general back-and-forth argument. Let $K\in\AF$ be isomorphic to the Fra\"iss\'e limit $Q_\mathcal K$ of the Fra\"iss\'e class $\mathcal K$, and let $i\in\NN$. We wish to show that $K$ is isomorphic to $\mu_i(K)$, so we must exhibit a bijection $f:\NN\to\NN$ such that $\mu_i(K)\circ(f\times f)=K$ (viewing $K,\mu_i(K):\NN\times\NN\to\ZZ$ as functions in this equality). To do so, we provide finite approximations to such an $f$ in the form $f_n:A_n\to B_n$ for finite subsets $A_n,B_n\subseteq\NN$ of size $n$ for all $n\geq1$ such that:
        \begin{enumerate}
            \item for all $n\geq1$, $f_n$ is an isomorphism of finite full subquivers of $K$ and $\mu_i(K)$,
            \item $A_1\subseteq A_2\subseteq\dots\subseteq\bigcup_{n\geq1}A_n=\NN$,
            \item $B_1\subseteq B_2\subseteq\dots\subseteq\bigcup_{n\geq1}B_n=\NN$, and
            \item for all $n\geq 1$, $f_{n+1}$ agrees with $f_n$ on $A_n$.
        \end{enumerate}
        If we successfully carry out this construction in a way that satisfies the above requirements, then we immediately have that $f:=\bigcup_{n\geq1}f_n$ is an isomorphism $K\to\mu_i(K)$.
        
        We begin our construction by setting $A_1=B_1=\{i\}$ and $f_1:i\mapsto i$. Clearly this satisfies requirement (1) above. To continue our recursion, we alternate between ``back'' and ``forth.'' More specifically, we break into cases for even and odd $n$ and build $f_n$ out of $f_{n-1}$ differently (though similarly) in either case.

        Suppose first that $n$ is even, that $A_{n-1}$, $B_{n-1}$, and $f_{n-1}$ are defined, and that we have satisfied all of the requirements (1)--(3) so far (where to satisfy (2) and (3) ``so far,'' we mean that $A_i\subseteq A_j$ and $B_i\subseteq B_j$ for all $1\leq i\leq j\leq n-1$). We will deal with the ``back''  half of our construction now. Let $b\in\NN$ be the least element of $\NN\setminus B_{n-1}$, and define $B_n:=B_{n-1}\cup\{b\}$. Since $i\in B_1\subseteq B_n$, we have that the full subquiver of $\mu_i(K)$ on the vertices $B_n$ agrees with the mutation $\mu_i(Q)$ at $i$ of the full subquiver $Q$ of $K$ on the vertices $B_n$. Because $K$ is the Fra\"iss\'e limit of the class $\mathcal K$, we have that $Q\in\mathcal K$. Moreover, since $\mathcal K$ is closed under mutation, we have that $\mu_i(Q)\in\mathcal K$ as well. Therefore, there exists some embedding $g:\mu_i(Q)\hookrightarrow K$. Letting $Q'$ be the restriction of $K$ to $A_{n-1}$, we have by Remark~\ref{weak-homog-rmk} (applied to the three embeddings $\text{id}_{A_{n-1}}:Q'\hookrightarrow K$, $f_{n-1}:Q'\hookrightarrow\mu_i(Q)$, and $g:\mu_i(Q)\hookrightarrow K$) that there exists a different embedding $h:\mu_i(Q)\hookrightarrow K$ with $h\circ f_{n-1}=\text{id}_{A_{n-1}}$. We therefore may define $A_n$ to be the vertex set of $h(\mu_i(Q))\subseteq K$, and note right away that it has the property that $A_{n-1}\subseteq A_n$ and $|A_n\setminus A_{n-1}|=1$. Let $f_n$ take the unique element of $A_n\setminus A_{n-1}$ to the unique $b\in B_n\setminus B_{n-1}$ and define $f_n(a)=f_{n-1}(a)$ for all $a\in A_{n-1}$. We note then that $f_n$ is an isomorphism onto $\mu_i(Q)$ by construction.

        For the ``forth'' part, suppose $n$ is odd, that $A_{n-1}$, $B_{n-1}$, and $f_{n-1}$ are defined, and that we have satisfied all of the requirements (1)--(3) so far. Let $a\in\NN$ be the least element of $\NN\setminus A_{n-1}$, and define $A_n:=A_{n-1}\cup\{a\}$. Let $Q$ be the subquiver of $K$ induced by the vertices $B_{n-1}$. Since $i\in B_{n-1}$, we have that the subquiver of $\mu_i(K)$ induced by the vertices $B_{n-1}$ is precisely $\mu_i(Q)$. Using the isomorphism $f_{n-1}$ from the subquiver $Q'$ of $K$ induced by the vertices $A_{n-1}$ to the subquiver $\mu_i(Q)$ of $\mu_i(K)$ induced by the vertices $B_{n-1}$ which fixes $i$ (recall that $f_{n-1}(i)=f_1(i)=i$), we have that $\mu_i(Q')$ is isomorphic to $Q$. Now let $Q''$ be the subquiver of $K$ induced by the vertices $A_n$ and notice that $Q'$ embeds into $Q''$ by the identity on $A_{n-1}$. Similarly, we have that $\mu_i(Q')$ embeds into $\mu_i(Q'')$ by the identity on $A_{n-1}$. Since $Q''\in\mathcal K$ and $\mathcal K$ is closed under mutation, we have that $\mu_i(Q'')\in\mathcal K$ as well. Now we see that $Q\cong\mu_i(Q')\subseteq\mu_i(Q'')$, so by Remark~\ref{weak-homog-rmk}, there exists an embedding $g$ of $\mu_i(Q'')$ into $K$ so that $g$ restricts to the subquiver embedding of $Q$ into $K$. Let $B_n$ be the vertex set of $g(\mu_i(Q''))$, and note that both $B_{n-1}\subseteq B_n$ and $|B_n\setminus B_{n-1}|=1$. Then one naturally has that the full subquiver of $\mu_i(K)$ induced by $B_n$ is isomorphic to $Q''$, and moreover that defining $f_n(a)$ to be the unique element in $B_n\setminus B_{n-1}$ (and otherwise extending $f_{n-1}$) witnesses such an isomorphism.

        After completing the above construction, it is straightforward to check that all four requirements (1)--(4) are satisfied, completing the proof.
    \end{proof}

This theorem demonstrates that there is an extreme amount of rigidity imposed when we consider \emph{finite} mutation sequences on \emph{unlabeled} Fra\"iss\'e limits of Fra\"iss\'e classes of quivers: these unlabeled quivers do not change after applying finite sequences of mutations! Should other interesting mutation-closed Fra\"iss\'e classes of quivers exist (see Section~\ref{sec-8} for more discussion on this), this theorem may prove useful in their study. On the other hand, the following theorem shows that this rigidity is entirely abandoned when considering \emph{infinite} mutation sequences on \emph{labeled} copies of the Fra\"iss\'e quiver, i.e., copies of the Fra\"iss\'e quiver $Q_\mathcal Q$ in $\AF$:

\begin{T-Fraisse-inf-seq}
    Let $F\in\AF$ be a Fra\"iss\'e quiver (Definition~\ref{fraisse-quiver-def}) and let $Q\in\AF$ be arbitrary. Then there exists an infinite mutation sequence $\mu_\mathbf i$ which converges on $F$ to $Q$, i.e., $\mu_\mathbf i(F)=Q$.
\end{T-Fraisse-inf-seq}
    \begin{proof}
        Let $F\in\AF$ be a Fra\"iss\'e quiver, and let $Q\in\AF$ be arbitrary. To construct an infinite mutation sequence $\mu_\mathbf i$ such that $\mu_\mathbf i(F)=Q$ in $\AF$, it is equivalent to build $\mu_\mathbf i$ so that for every $n\geq 1$, there exists some $\ell\geq 1$ with the property that for all $k\geq\ell$ it holds that $\mu_{i_k}\circ\dots\circ\mu_{i_1}(F)\in U_{Q,\{1,\dots,n\}}$. We will build finite initial segments $\mathbf i_n=(i_1,\dots,i_{k_n})$ of $\mathbf i$ inductively which satisfy the following requirements:
        \begin{enumerate}
            \item For all $n\geq1$, $\mathbf i_n=(i_1,\dots,i_{k_n})$ is a strict initial segment of $\mathbf i_{n+1}=(i_1,\dots,i_{k_{n+1}})$, i.e., $k_n<k_{n+1}$. This guarantees that $\mathbf i:=\bigcup_{n\geq1}\mathbf i_n$ exists and is an infinite sequence.
            \item For all $n\geq1$ and all $k_n<j\leq k_{n+1}$, it holds that $\mu_{i_j}\circ\dots\circ\mu_{i_1}(F)\in U_{Q,\{1,\dots,n\}}$.
        \end{enumerate}
        Note that the second condition in fact implies the following stronger statement: for all $n\geq1$ and all $j>k_n$, it holds that $\mu_{i_j}\circ\dots\circ\mu_{i_1}(F)\in U_{Q,\{1,\dots,n\}}$. Hence, the above conditions are enough to guarantee convergence of $\mu_\mathbf i(F)$ to $Q$ in $\AF$.

        We begin with the base case $n=1$. Set $\mathbf i_1=\varnothing$ to be the empty sequence. As long as $\mathbf i_2\neq\varnothing$, we have satisfied condition (1) above. Moreover, condition (2) above is vacuously satisfied, as it holds that $U_{Q,\{1\}}=\AF$ for \emph{any} quiver $Q\in\AF$.

        Now suppose we have defined $\mathbf i_n$ in such a way that conditions (1) and (2) above are satisfied for all $\mathbf i_m$ with $m<n$. Our goal is to define $\mathbf i_{n+1}$ as a strict extension of $\mathbf i_n$ in order to satisfy these conditions for $\mathbf i_n$. We will consider two cases.
        
        In the first case, suppose $\mu_{\mathbf i_n}(F)$ is in $U_{Q,\{1,\dots,n\}}$ (note that this is not guaranteed by the $\mathbf i_{n-1}$ case; all we know so far for sure is that $\mu_{\mathbf i_n}(F)\in U_{Q,\{1,\dots,n-1\}}$). To produce $\mathbf i_{n+1}$, we will extend $\mathbf i_n$ by mutating at a single vertex in $F$ which has no arrows to or from the vertices $\{1,\dots,n\}$ in $\mu_{\mathbf i_n}(F)$. By Theorem~\ref{Fraisse-singleton-thm} above, we know that $\mu_{\mathbf i_n}(F)$ is itself isomorphic to the Fra\"iss\'e quiver. Since the quiver $P$ obtained by adjoining an isolated vertex to the full subquiver of $\mu_{\mathbf i_n}(F)$ on $\{1,\dots,n\}$ is finite (and therefore belongs to the Fra\"iss\'e class of finite quivers), we may use Remark~\ref{weak-homog-rmk} to find a vertex $v\in\NN$ such that $\mu_{\mathbf i_n}(F)$ restricted to $\{1,\dots,n\}\cup\{v\}$ is isomorphic to $P$. Then we may extend $\mathbf i_n=(i_1,\dots,i_{k_n})$ to $\mathbf i_{n+1}$ by setting $\mathbf i_{n+1}=(i_1,\dots,i_{k_n},v)$. This lets us meet condition (1). Since $v$ is isolated in $\mu_{\mathbf i_n}(F)$ restricted to $\{1,\dots,n\}\cup\{v\}$, we have that the restriction of $\mu_{\mathbf i_{n+1}}(F)$ to $\{1,\dots,n\}$ agrees with the restriction of $\mu_{\mathbf i_n}(F)$ to $\{1,\dots,n\}$. Furthermore, this implies that $\mu_{\mathbf i_{n+1}}(F)\in U_{Q,\{1,\dots,n\}}$, as we know by hypothesis that $\mu_{\mathbf i_n}(F)\in U_{Q,\{1,\dots,n\}}$. Thus, we have that for $k_n<j\leq k_{n+1}=k_n+1$, $\mu_{i_j}\circ\dots\circ\mu_{i_1}(F)\in U_{Q,\{1,\dots,n\}}$, fulfilling condition (2).

        On the other hand, suppose that $\mu_{\mathbf i_n}(F)\in U_{Q,\{1,\dots,n-1\}}\setminus U_{Q,\{1,\dots,n\}}$. This means that there exist vertices $v_1,\dots,v_d\in\{1,\dots,n-1\}$ such that $\mu_{\mathbf i_n}(F)(v_r,n)\neq Q(v_r,n)$ for $r=1,\dots,d$, $d\geq1$. To correct this, we again use the fact that $\mu_{\mathbf i_n}(F)$ is the Fra\"iss\'e quiver by Theorem~\ref{Fraisse-singleton-thm} and that it therefore carries the extension property of Remark~\ref{weak-homog-rmk}. For each $r=1,\dots,d$, let $\varepsilon_r\in\{+1,-1\}$ be the sign of $Q(v_r,n)-\mu_{\mathbf i_n}(F)(v_r,n)\in\ZZ$. We find by the extension property a full subquiver $P$ of $\mu_{\mathbf i_n}(F)$ on vertices $\{1,\dots,n\}\cup\{w_1,\dots,w_d\}$ for some distinct $w_1,\dots,w_d>n$ with the following arrow counts:
        
        \begin{itemize}
            \item For all $r=1,\dots,d$, one has $\mu_{\mathbf i_n}(F)(v_r,w_r)=\varepsilon_r$ and $\mu_{\mathbf i_n}(F)(w_r,n)=Q(v_r,n)-\mu_{\mathbf i_n}(F)(v_r,n)$.
            \item For all $r=1,\dots,d$ and $v\in\{1,\dots,n-1\}$, $v\neq v_r$ implies $\mu_{\mathbf i_n}(F)(v,w_r)=0$.
            \item For all $r,r'\in\{1,\dots,d\}$, one has $\mu_{\mathbf i_n}(F)(w_r,w_{r'})=0$.
        \end{itemize}
        
        This full subquiver $P$ of $\mu_{\mathbf i_n}(F)$ has the property that mutating at each $w_r$ once (any order will do, as the $w_r$ are not connected to one another directly by arrows) fixes the incorrect arrow counts:
        \begin{align*}
        \mu_{w_d}\circ\dots\circ\mu_{w_1}\circ\mu_{\mathbf i_n}(F)(v_r,n)&=\mu_{w_r}\circ\mu_{\mathbf i_n}(F)(v_r,n)\\
        &=\mu_{\mathbf i_n}(F)(v_r,n)+\mu_{\mathbf i_n}(F)(v_r,w_r)[\mu_{\mathbf i_n}(F)(w_r,n)]_+\\
        &\;\;\;\;+[-\mu_{\mathbf i_n}(F)(v_r,w_r)]_+\mu_{\mathbf i_n}(F)(w_r,n)\\
        &=\mu_{\mathbf i_n}(F)(v_r,n)+Q(v_r,n)-\mu_{\mathbf i_n}(F)(v_r,n)\\
        &=Q(v_r,n).
        \end{align*}
        Moreover, the other arrow counts (those within $\{1,\dots,n-1\}$ and those between $\{1,\dots,n-1\}$ and $n$ which did not need changing) remain unchanged by the construction of the full subquiver $P$. Thus, defining $\mathbf i_{n+1}$ from $\mathbf i_n=(i_1,\dots,i_{k_n})$ by $\mathbf i_{n+1}:=(i_1,\dots,i_{k_n},w_1,\dots,w_d)$ immediately satisfies condition (1) since $d\geq1$. By our above argument, $\mathbf i_{n+1}$ also satisfies condition (2).

        In either case, we have strictly extended $\mathbf i_n$ to $\mathbf i_{n+1}$ in such a way so as to preserve the restriction of $\mu_{\mathbf i_n}(F)$ to $\{1,\dots,n\}$ with every subsequent mutation. Since this stable restriction agrees with the restriction of $Q$ to $\{1,\dots,n\}$, this gives us an infinite mutation sequence $\mathbf i$ with the property that $\mu_\mathbf i(F)=Q$, completing the proof.
    \end{proof}

%% Section 7: The mutation class topology and Theorem 1.4 %%
\section{The mutation class topology and Theorem~\ref{mut-class-top-thm}}\label{sec-7}

Our final main result concerns a topological space, the \emph{mutation class space $\mathcal M$}, considered by Ervin and Jackson in their preprint \cite{ervin_topology_2024}. We begin this section by reviewing their construction of the mutation class space $\mathcal M$ and listing some of its basic properties.

\begin{Def}\label{mut-class-set-def}
    Let $Q:[n]\times[n]\to\ZZ$ be a quiver on the set $[n]=\{1,\dots,n\}$. The set of all quivers $Q'$ on $[n]$ mutation-equivalent to an isomorphic copy of $Q$ is called the \emph{mutation class} of $Q$, denoted $[Q]$. If $P$ is (isomorphic to) a full subquiver of $Q$, we write $[P]\preceq[Q]$ and say that the class $[P]$ \emph{embeds} into the class $[Q]$. The set of all mutation classes of finite quivers is denoted $\mathcal M$.
\end{Def}

Using Lemma~\ref{mut-restr-commute-lemma}, the set $\mathcal M$ of mutation classes is naturally seen to be a partially-ordered set under the embedding relation $\preceq$ on mutation classes. Ervin and Jackson use this fact to define their topology on $\mathcal M$ as follows:

\begin{Def}\label{mut-class-space-def}
    The \emph{mutation class topology} on $\mathcal M$ is the topology whose open sets are the upward-closed subsets of the poset $(\mathcal M,\preceq)$.
\end{Def}

One may readily check that this is indeed a topology on $\mathcal M$ using basic properties of posets. Indeed, every poset induces a topology on its underlying set in this fashion (called its \emph{Alexandrov topology}).

We now collect some basic properties established in \cite{ervin_topology_2024} about the mutation class space $\mathcal M$.

\begin{Prop}\label{mut-class-space-properties-prop}
    The following statements hold about the mutation class space $\mathcal M$:
    \begin{enumerate}[label=(\roman*)]
        \item $\mathcal M$ is connected.
        \item $\mathcal M$ is compact.
        \item Every nonempty open subset of $\mathcal M$ is dense.
        \item Every dense subset of $\mathcal M$ is infinite.
    \end{enumerate}
\end{Prop}

To state and prove Theorem~\ref{mut-class-top-thm}, we define the subspace $\mathcal M'$ of $\mathcal M$ to consist of mutation classes of quivers that either have no isolated vertices or are simply the one-vertex quiver. Note that, since $\mathcal M'$ is a subposet of $\mathcal M$, the subspace topology on $\mathcal M'$ agrees with the Alexandrov topology on $(\mathcal M',\preceq)$.

\begin{T-mut-class-top}
    Let $\mathbf{Fin}\subseteq\AF$ be the subspace of $\AF$ consisting of finite quivers, and let $\sim$ denote the equivalence relation on $\mathbf{Fin}$ induced by (finite-)mutation-equivalence and isomorphism. Then $\mathbf{Fin}/\!\sim$ is homeomorphic to $\mathcal M'$ by the map sending $[Q]_\sim\in\mathbf{Fin}/\!\sim$ to the mutation-equivalence class of $Q'$, the quiver obtained from $Q\in\mathbf{Fin}$ by removing all of its isolated vertices \emph{if $Q$ is not the quiver with no arrows} or to the mutation-equivalence class of the one-vertex quiver if it is.
\end{T-mut-class-top}
    \begin{proof}
        Let $f:\mathbf{Fin}/\!\sim\;\to\mathcal M'$ denote the map of interest. Note first that $f$ is well-defined: up to mutation-equivalence, neither mutation nor relabeling vertices changes the end result of removing isolated vertices from a quiver. One also sees that $f$ is bijective, essentially by definition. The one part about this that one must be careful about is the quiver $Q\in\mathbf{Fin}$ with no arrows. Since the one-vertex quiver is not the image of any nontrivial $Q\in\mathbf{Fin}$, one retains injectivity (and achieves surjectivity) when declaring the image of the quiver with no arrows in $\mathbf{Fin}$ to be the one-vertex quiver. Hence, it only remains to show that $f$ is a homeomorphism.
        
        It will help immensely to have a basis for topology on the quotient space $\mathbf{Fin}/\!\sim$. Note that mutation maps and vertex-relabeling maps are homeomorphisms on $\AF$ (the former by Proposition~\ref{mut-cont-prop} and the latter by the quick fact that $\sigma(U_{Q,v})=U_{\sigma(Q),\sigma(V)}$ for any $\sigma\in\text{Sym}(\NN)$). Using this, we can see that the quotient map $\pi:\mathbf{Fin}\to\mathbf{Fin}/\!\sim$ is open. Indeed, let $U\subseteq\mathbf{Fin}$ be open. Using the fact that $\pi$ is a quotient map, we have $\pi(U)$ is open if and only if $\pi^{-1}(\pi(U))$ is open. One may then observe that
        $$\pi^{-1}(\pi(U))=\bigcup_{g\in G}g(U)$$
        is open, where $G\subseteq\text{Aut}(\AF)$ is the subgroup of homeomorphisms of $\AF$ generated by mutations and vertex-relabelings. Thus, $\pi(U)$ is open, showing that $\pi$ is an open map. Therefore, the images $U'_{Q,V}:=\pi(U_{Q,V}\cap\mathbf{Fin})$ for $Q\in\mathbf{Fin}$ and $V\subseteq\NN$ finite form a basis of open sets of $\mathbf{Fin}/\!\sim$.

        We now show that $f$ is continuous. Let $S\subseteq\mathcal M'$ be an upward-closed set in $(\mathcal M',\preceq)$. Fix $Q\in\mathbf{Fin}$ so that $[Q]_\sim\in f^{-1}(S)$, i.e., we have $f([Q]_\sim)\in S$. We wish to find some finite $V\subseteq\NN$ so that $f(U'_{Q,V})\subseteq S$. We claim $V=\text{supp}(Q)\subseteq\NN$ works (note that this is finite as $Q$ is finite). That is, if $[P]_\sim\in U'_{Q,\text{supp}(Q)}$, we claim that $f([P]_\sim)\in S$. It suffices to show that $f([Q]_\sim)\preceq f([P]_\sim)$. Since $[P]_\sim\in U'_{Q,\text{supp}(Q)}=\pi(U_{Q,\text{supp}(Q)}\cap\mathbf{Fin})$, we have that there exists $R\in\mathbf{Fin}$ so that $R\in U_{Q,\text{supp}(Q)}$ and $P\sim R$. Because we only care about $[P]_\sim$ and not $P$ itself, we can assume without loss of generality that $P=R$. Then $Q\leq P$, since $P\in U_{Q,\text{supp}(Q)}$. Removing isolated vertices from $Q$ and $P$ (or replacing one or more of $Q$ or $P$ with the one-vertex quiver if $Q$ or $P$ is trivial) to produce finite quivers $Q',P'$ preserves this embedding relationship, i.e., $Q'\leq P'$. Thus, we have $[Q']\preceq[P']$ in $\mathcal M'$. Since $f([Q]_\sim)=[Q']$ and $f([P]_\sim)=[P']$, we have the desired embedding relationship $f([Q]_\sim)\preceq f([P]_\sim)$. As stated above, this suffices to show that $f(U'_{Q,\text{supp}(Q)})\subseteq S$, so $f$ is continuous.

        Lastly, we must show that $f$ is open. Let $U'_{Q,V}=\pi(U_{Q,V}\cap\mathbf{Fin})$ be a basic open set in $\mathbf{Fin}/\!\sim$. We wish to show that if $[P]\in f(U'_{Q,V})\subseteq\mathcal M'$, there exists an upward-closed set $S\subseteq\mathcal M'$ containing $[P]$ which is a subset of $f(U'_{Q,V})$. It is sufficient to show that the principal upward-closed set $S_{[P]}:=\{[R]\in\mathcal M'\mid[P]\preceq[R]\}$ is a subset of $f(U'_{Q,V})$. Equivalently, for all $[R]\succeq[P]$ in $\mathcal M'$, we ought to be able to find $[T]_\sim\in U'_{Q,V}$ so that $f([T]_\sim)=[R]$.
        
        Let $[W]_\sim\in U'_{Q,V}$ be such that $f([W]_\sim)=[P]$. Without loss of generality (mutating $W$ as needed), we may assume that $W'=P$, i.e., $P$ is the result of removing all isolated vertices in $W$. Moreover, we may also assume without loss of generality (by mutating and relabeling $W$ as necessary) that $W\in U_{Q,V}\cap\mathbf{Fin}$. Then it is clear that $U_{W,\text{supp}(W)}\cap\mathbf{Fin}\subseteq U_{Q,V}\cap\mathbf{Fin}$. From this, it follows that $f(U'_{W,\text{supp}(W)})\subseteq f(U'_{Q,V})\subseteq\mathcal M'$.
        
        We will show that $f(U'_{W,\text{supp}(W)})=S_{[P]}$. The forward inclusion is immediate: any quiver with $W$ as a restriction has $W'=P$ as a restriction as well, so $f([T]_\sim)\in S_{[P]}$ for all $[T]_\sim\in U'_{W,\text{supp}(W)}$. The reverse inclusion proceeds as follows. Let $[R]\in S_{[P]}$, i.e., $[P]\preceq[R]$. Without loss of generality (mutating $R$ as needed), we may assume $P\leq R$. To construct a preimage of $[R]$ under $f$ that lies in $U'_{W,\text{supp}(W)}$, we extend $W$ as follows. Let $m=\max(\text{supp}(W))\in\NN$. Now place arrows between vertices in $\text{supp}(W)\subseteq\NN$ and $\{m+i\mid1\leq i\leq |R|-|P|\}\subseteq\NN$ so as to create an isomorphic copy $T$ of $R$ with $T\in U_{W,\text{supp}(W)}$. Then it follows that $[T]_\sim\in U'_{W,\text{supp}(W)}\cap\mathbf{Fin}\subseteq U'_{Q,V}\cap\mathbf{Fin}$ is such that $f([T]_\sim)=[R]\in S_{[P]}$, proving the claim.
        \end{proof}

%% Section 8: Future work %%
\section{Future work}\label{sec-8}

We finally wish to highlight several viable areas for further work with infinite quivers through our topological perspective.

First, there are multiple possibilities for expanding and tightening some of our results in Section~\ref{sec-4}. Namely, one ought to be able to provide similar results for other quiver properties--what are their Borel complexities (if they are even Borel)? Are they dense, and do they have empty interior? Further, one should be able to tighten our bounds on complexity and provide hardness/completeness results for the properties discussed in Section~\ref{sec-4}.

In this vein, one may benefit from paying special attention to \emph{surface type} quivers. These quivers arise from triangulations of orientable surfaces with sets of marked points. Several works have made progress on the case of surfaces with infinitely many marked points in recent years, see \cite{holm_cluster_2012, august_cluster_2023, liu_cluster_2017, baur_transfinite_2018, canakci_infinite_2019, canakci_cluster_2025} for a few examples. A reasonably interesting task would be to relate existing notions of convergence of infinite mutation sequences for triangulations of surfaces to our notions in $\LF$ and $\AF$. It is not terribly difficult to cook up an example of an infinite mutation sequence for a triangulation of a surface which converges in the sense of \cite{canakci_infinite_2019} but which diverges on the underlying locally finite quiver in $\LF$ (after choosing a labeling of the arcs of the surface with natural numbers). However, it seems much more likely that convergent infinite mutation sequences of triangulations should yield convergent infinite mutation sequences on their corresponding quivers when considered in $\AF$. Future work should make this connection precise and explore other connections to infinite surface cluster combinatorics.

Next, one could ask about the complexities of various sets related to infinite mutation sequences. For example, one could ask about the complexity of the sets $\mathbf{ME}^\infty_\LF\subseteq\LF^2$ and $\mathbf{ME}^\infty_\AF\subseteq\AF^2$ of pairs of quivers $(Q,Q')$ such that there exists an infinite mutation sequence $\mu_\mathbf i$ converging on $Q$ with $\mu_\mathbf i(Q)=Q'$ (with convergence understood relative to the choice of space). At the time of preparation of this article, the best upper bound the author is able to establish is that these sets are $\mathbf\Sigma^1_1$. To show these sets to be $\mathbf\Pi^1_1$ as well is to prove that they are Borel (cf. Remark~\ref{proj-hierarchy-rmk}), though this would still leave open the question of where they fall in the Borel hierarchy. Showing that these sets are $\mathbf\Pi^1_1$ necessarily would have to take advantage of quiver mutation combinatorics in an interesting, non-trivial way. It is also entirely possible, on the other hand, that these sets are $\mathbf\Sigma^1_1$-complete, providing a convincing, quantitative piece of evidence that infinite quiver mutations can be quite complicated. Daringly, we also remark that a third phenomenon is possible: that there exists a model of ZFC in which these sets are $\mathbf\Sigma^1_1\setminus\mathbf\Pi^1_1$ but not $\mathbf\Sigma^1_1$-complete. (The existence of these kinds of sets is independent of ZFC, so necessarily this could only hold in some models and not all of them; this seems unlikely.)

We also offer some questions motivated by our preliminary results about the Fra\"iss\'e quiver in Section~\ref{sec-6}. Perhaps the question with the most immediate implications for work on \emph{finite} quivers is the following: which other classes of quivers are Fra\"iss\'e classes (cf. Definition~\ref{fraisse-class-def})? Definitionally, these classes represent hereditary properties of quivers. Many known hereditary properties of interest already happen to carry the joint embedding property merely by way of taking disjoint unions of quivers. With this, our question reduces to the following: which hereditary properties of quivers have the amalgamation property? We suggest that this in fact poses several independently difficult problems: can we amalgamate mutation-acyclic quivers to obtain new mutation-acyclic quivers? Can we amalgamate quivers with reddening sequences (or maximal green sequences) to obtain new quivers with reddening sequences (or maximal green sequences)? How about quivers with a finite (pre-)forkless part? Another question to ask is about the representation-theoretic implications of the existence of infinite quivers like the Fra\"iss\'e quiver. In a sense we do not attempt to make precise here, the category of finite-dimensional representations of the Fra\"iss\'e quiver should embed the representation categories of all finite loop- and 2-cycle-free quivers (and in many different ways at that). Moreover, it should do so rather ``symmetrically'' due to the homogeneity property of the Fra\"iss\'e quiver, in a way which directly mimics the automorphism-based definition of homogeneity (Definition~\ref{homog-def}) on the functorial level. That the Fra\"iss\'e quiver is isomorphic to all of its finite mutations (cf. Theorem~\ref{Fraisse-singleton-thm}) may prove interesting in this context as well.

The link established in Section~\ref{sec-7} to the mutation class space of Ervin and Jackson also invites further inquiry. For example, one may ask about potential generalizations of our spaces to include \emph{skew-symmetrizable quivers} that maintain the same relationship to their \emph{skew-symmetrizable mutation class space} as ours do to their original space (cf. Theorem~\ref{mut-class-top-thm}). We may also lift their questions about Banff and Louise quivers to our context as well: how complicated are these sets? What do their closures look like? Are they different sets (when defined appropriately in the infinite setting)?

Lastly, we remark that all work in this paper was carried out in the setting of \emph{non-effective} descriptive set theory. Future work should aim to effectivize as many results we have stated as possible in order to provide information about the computational complexity of infinite quiver properties. That said, it is unfortunately the case that effective versions of the above results likely will not translate cleanly to the finitary setting, and studying the computational complexities of finitary quiver problems almost certainly must use a different set of tools entirely. Nevertheless, we hope that the results provided here remain a useful heuristic when discussing the complexity of quiver mutation problems as a whole.

\bibliography{references}
\bibliographystyle{plain}

\end{document}